\documentclass[12pt]{amsart}

 \usepackage{amscd, graphicx, psfrag}
 \usepackage{ifthen}    
 \usepackage{amssymb}   
 \usepackage{amstext}
 \usepackage{amsfonts}
 \usepackage{amsmath}   
 \usepackage{amscd}     
 \usepackage{latexsym}  
 \usepackage{color}
 \usepackage{epsfig}
 \usepackage{epic}
 \usepackage{pstricks, pst-node,pst-text,pst-tree}
 \usepackage[all]{xy}
 \usepackage{array}
 \usepackage{amsthm}
 \usepackage{eucal, mathrsfs} 
 \usepackage{hyperref}

\newcommand{\numberset}[1]{\ensuremath{\mathbb{#1}}}    
\newcommand{\C}{\numberset{C}}  
\newcommand{\N}{\numberset{N}}  
\newcommand{\R}{\numberset{R}}  
\newcommand{\Z}{\numberset{Z}}  
\newcommand{\PP}{\numberset{P}}  

\newcommand{\inner}[2]{ \langle {#1}, {#2} \rangle}

\newcommand{\caff}{\mathcal{A}\mathit{ff}}

\theoremstyle{definition}

\newtheorem{thm}{Theorem}[section]
\newtheorem{prop}[thm]{Proposition}
\newtheorem{lem}[thm]{Lemma}
\newtheorem{cor}[thm]{Corollary}
\newtheorem{rem}[thm]{Remark}
\newtheorem{ex}[thm]{Example}
\newtheorem{defi}[thm]{Definition}

\newtheorem{ass}[thm]{Assumption}
\newtheorem{con}[thm]{Conjecture}

\newtheorem*{mthm}{Main Theorem}

\DeclareMathOperator{\im}{Im}

\DeclareMathOperator{\Crit}{Crit} 

\DeclareMathOperator{\conv}{Conv} 
\DeclareMathOperator{\cone}{Cone}
\DeclareMathOperator{\id}{Id}
 \DeclareMathOperator{\Gl}{Gl}
\DeclareMathOperator{\Sl}{Sl}
\DeclareMathOperator{\spn}{span}

\DeclareMathOperator{\Log}{Log}
\DeclareMathOperator{\aff}{Aff}
\DeclareMathOperator{\bary}{Bar}
\DeclareMathOperator{\inter}{Int}

\newboolean{mynotes}\setboolean{mynotes}{true}

\newcommand{\mycomments}[1]{
           \ifthenelse{\boolean{mynotes}}
                      {#1}{}
           }

\begin{document}

\title{Conifold transitions via affine geometry and mirror symmetry}
\author{Ricardo Casta\~no-Bernard and Diego Matessi}

\begin{abstract}
Mirror symmetry of Calabi-Yau manifolds can be understood via a Legendre duality between a pair of certain affine manifolds with singularities called tropical manifolds. In this article, we study conifold transitions from the point of view of Gross and Siebert \cite{G-Siebert2003, G-Siebert2007, GrSi_re_aff_cx}.  We introduce the notions of tropical nodal singularity, tropical conifolds, tropical resolutions and smoothings. We interpret known global obstructions to the complex smoothing and symplectic small resolution of
compact nodal Calabi-Yaus in terms of certain tropical $2$-cycles containing the nodes in their associated tropical conifolds. We prove that the existence of such cycles implies the simultaneous vanishing of the obstruction to smoothing the original Calabi-Yau \emph{and} to resolving its mirror. We formulate a conjecture suggesting that the existence of these cycles should imply that the tropical conifold can be resolved and its mirror can be smoothed, thus showing that the mirror of the resolution is a smoothing. We partially prove the conjecture for certain configurations of nodes and for some interesting examples.
\end{abstract}

\maketitle


\section{Introduction}
A geometric transition between a pair of smooth varieties is the process of deforming the first variety to a singular one and then obtaining the second one by resolving the singularities. The first variety is called a smoothing and the second one a resolution. In  \cite{morrison-look}, Morrison conjectures that, in certain circumstances, mirror symmetry should map a pair of smooth Calabi-Yau manifolds, related by a geometric transition, to another pair, also related by a geometric transition but with the roles reversed, so that the mirror of a smoothing should be a resolution and vice-versa. This idea is supported by evidences and examples. Morrison also suggests that a new understanding of this phenomenon could come from the SYZ interpretation of mirror symmetry as a duality of special Lagrangian torus fibrations. Building on ideas of Hitchin, Gross, Gross-Wilson, Kontsevich-Soibelman and others on the SYZ conjecture \cite{Hitchin}, \cite{Gross_spLagEx}, \cite{G-Wilson2}, \cite{Kontsevich-Soibelman}, \cite{K-S-archim}, Gross and Siebert show that mirror pairs of Calabi-Yau manifolds can be constructed from a Legendre dual pair of affine manifolds with singularities and polyhedral decompositions, also called tropical manifolds \cite{G-Siebert2003, G-Siebert2007, GrSi_re_aff_cx}. In this article we consider the special case of conifold transitions. We introduce the notion of tropical conifold, i.e. of tropical manifold with nodes, and we show that the smoothing/resolution process also has a natural description in this context (this was first observed by Gross \cite{Gross_spLagEx} and Ruan \cite{Ruan}). Indeed the smoothing of a tropical conifold simultaneously induces a resolution of its mirror, but in general the process is obstructed. To study global obstructions we introduce the notion of tropical $2$-cycle in a tropical conifold. Our main result is the following:
\begin{mthm}
The existence of a tropical $2$-cycle containing the nodes in a tropical conifold implies the vanishing of the obstructions to the smoothing of the associated Calabi-Yau variety and to the resolution of its mirror.  
\end{mthm}
See Theorem \ref{good_rel_pos} for the precise statement. We formulate a conjecture claiming that the inverse also holds, i.e. that the vanishing of these obstructions can always be detected by tropical $2$-cycles. Moreover we expect that the existence of a resolution/smoothing of a set of nodes in the tropical conifold should be equivalent to some property expressible in terms tropical $2$-cycles containing the nodes. This would show that the smoothing and the resolution are themselves mirror pairs in the sense of Gross-Siebert. We partially prove the conjecture for some special configurations of nodes and for an interesting family of examples. 

\subsection{Conifold transitions.} A node is the $3$-fold singularity with local equation $xy-zw=0$. A small resolution of a node has a $\PP^1$ as its exceptional cycle, with normal bundle $\mathcal O_{\PP^1}(-1) \oplus \mathcal O_{\PP^1}(-1)$. The smoothing of a node (i.e. $xy-zw=\epsilon$) produces a Langrangian $3$-sphere as a vanishing cycle. A conifold transition is the geometric transition associated with a $3$-fold with nodal singularities (i.e. a ``conifold"). It was proved by Friedman and Tian \cite{friedman-trivial}, \cite{Tian} that a compact complex conifold can be smoothed to a complex manifold if and only if the exceptional cycles of a small resolution satisfy a ``good relation" in homology (see equation (\ref{cx_goodrel})). Similarly, on the symplectic side, it was shown by Smith, Thomas and Yau \cite{STY} that a ``symplectic conifold" has a symplectic (small) resolution, with symplectic exceptional cycles, if and only if the vanishing cycles of a smoothing satisfy a good relation. These two results are a manifestation of the idea that the mirror of a complex smoothing should be a symplectic resolution.

\subsection{SYZ conjecture.} Mirror symmetry is usually computed when the Calabi-Yau manifold is the generic fibre of a family $\psi: \chi \rightarrow \C$, where the special fibre $\chi_0 = \psi^{-1}(0)$ is highly degenerate (e.g. one requires that $\chi_0$ has maximally unipotent monodromy). The Strominger-Yau-Zaslow (SYZ) conjecture \cite{SYZ} claims that mirror Calabi-Yau pairs $X$ and $\check X$ should admit ``dual'' special Lagrangian fibrations, $f: X \rightarrow B$ and $\check f: \check X \rightarrow B$. This original idea has been revised by Gross-Wilson \cite{Gross_spLagEx}, \cite{G-Wilson2} and Kontsevich-Soibelman \cite{Kontsevich-Soibelman}, who claimed that special Lagrangian fibrations should exist only in some limiting sense as the fibre $\chi_s =\psi^{-1}(s)$ approaches the singular fibre $\chi_0$ (see also the survey paper \cite{Gross_SYZ_rev}). This ``limiting fibration" can be described in terms of a certain structure on the base $B$ of the fibration. Here $B$ is a real manifold and the structure on $B$ should contain information concerning the complex and symplectic structure of the Calabi-Yau manifold. Moreover, this data contains intrinsically a duality given by  a Legendre transform. The important fact is that the structure on $B$ should allow the ``reconstruction" of the original Calabi-Yau. This is known as the reconstruction problem. Therefore, finding the mirror of a given family $\psi: \chi \rightarrow \C$ of Calabi-Yau's becomes the process of constructing $B$, with its structure, applying the Legendre transform to obtain the dual base $\check B$, with dual structure, and then reconstructing the mirror family via some reconstruction theorem. For instance, in dimension $2$, Kontsevich and Soibelman \cite{K-S-archim}, construct a rigid analytic $K3$ from an affine structure on $S^2$ with $24$ punctures.

\subsection{Tropical manifolds and mirror symmetry.} In all dimensions, this program has been completed by Gross and Siebert in a sequence of papers \cite{G-Siebert2003}, \cite{G-Siebert2007}, \cite{GrSi_re_aff_cx}. On $B$ they consider the structure of an integral affine manifold with singularities and polyhedral decompositions. Roughly this means that $B$ is obtained by gluing a set of $n$-dimensional integral convex polytopes in $\R^n$ by identifying faces via integral affine transformations (this is the polyhedral decomposition, denoted $\mathcal P$). Then, at the vertices $v$ of $\mathcal P$ one defines a fan structure, which identifies the tangent wedges of the polytopes meeting at $v$ with the cones of a fan $\Sigma_v$ in $\R^n$. For a certain codimension $2$ closed subset $\Delta \subset B$, this structure determines an atlas on $B_0 = B - \Delta$ such that the transition maps are integral affine transformations. The set $\Delta$, called the discriminant locus, is the set of singularities of the affine structure. An additional crucial piece of data is a polarization, which consists of a so-called ``strictly convex multivalued piecewise linear function" $\phi$ on $B$. Such a $\phi$ is specified by the data of a strictly convex piecewise linear function $\phi_v$ defined on every fan $\Sigma_v$, plus compatibility conditions between $\phi_v$ and $\phi_w$ for vertices $v$ and $w$ belonging to a common face. All these data, which we denote by the triple $(B, \mathcal P, \phi)$, are also called a polarized tropical manifold. The ``discrete Legendre transform" associates to $(B, \mathcal P, \phi)$ another triple $(\check B, \check{\mathcal P}, \check \phi)$. Essentially, at a vertex $v$ of $B$, the fan $\Sigma_v$ and function $\phi_v$ provide an $n$-dimensional polytope $\check v$, by the standard construction in toric geometry. Two polytopes $\check v$ and $\check w$, associated to vertices $v$ and $w$ on a common edge of $\mathcal P$, can be glued together along a face using the compatibilities between the pairs $(\Sigma_v, \phi_v)$ and $(\Sigma_w, \phi_w)$. This gives $\check B$ and the polyhedral decomposition $\check {\mathcal P}$. The fan structure and function $\check \phi$ at the vertices of $\mathcal{\check P}$ come from the $n$-dimensional polytopes of $\mathcal P$ essentially using the inverse construction.  

In order to have satisfactory reconstruction theorems it is necessary to put further technical restrictions on $(B, \mathcal P, \phi)$. Gross and Siebert define such conditions and call them ``positivity and simplicity''. For convenience, we will say that a polarized tropical manifold is smooth if it satisfies these nice conditions.  In particular, in the $3$-dimensional case smoothness of $B$ amounts to the fact that $\Delta$ is a $3$-valent graph and the vertices can be of two types: ``positive" or ``negative", depending on the local monodromy of the affine structure. The Gross-Siebert reconstruction theorem \cite{GrSi_re_aff_cx}, ensures that given a smooth polarized tropical manifold $(B, \mathcal P, \phi)$,  it is possible to construct a toric degeneration $\psi: \chi \rightarrow \C$ of Calabi-Yau varieties, such that $B$ is the dual intersection complex of the singular fibre $\chi_0$. The mirror family $\check \psi: \check \chi \rightarrow \C$ is obtained by applying the reconstruction theorem to the Legendre dual $(\check B, \check {\mathcal P}, \check \phi)$. 

The integral affine structure on $B_0 = B-\Delta$ implies the existence of a local system $\Lambda^* \subset T^*B_0$, whose fibres $\Lambda_b \cong \Z^n$ are maximal lattices in $T^*_bB_0$. 
Then one can form the $n$-torus bundle $X_{B_0} = T^*B_0/ \Lambda^*$ over $B_0$. The standard symplectic form on $T^*B_0$ descends to $X_{B_0}$ and the projection $f_0: X_{B_0} \rightarrow B_0$ is a Lagrangian torus fibration. 
In \cite{CB-M}, we proved that if $B$ is a $3$-dimensional smooth tropical manifold then one can form a symplectic compactification of $X_{B_0}$. This is a symplectic manifold $X_B$, containing $X_{B_0}$ as a dense open subset, together with a Lagrangian fibration $f: X_B \rightarrow B$ which extends $f_0$. This is done by inserting suitable singular Lagrangian fibres over points of $\Delta$. Topologically the compactification $X_B$ is based on the one found by Gross in \cite{TMS}. It is expected that $X_B$ should be diffeomorphic to a smooth fibre $\chi_s$ of the family $\psi: \chi \rightarrow \C$ in the Gross-Siebert reconstruction theorem, whose dual intersection complex is $(\check B, \check {\mathcal P}, \check \phi)$. This result has been announced in \cite{Gross_Batirev}, Theorem 0.1. A complete proof for the quintic $3$-fold in $\PP^4$ is found in  \cite{TMS}. We also expect that $X_B$ should be symplectomorphic to $\chi_s$ with a suitable K\"ahler form, although there is no proof of this yet. 
\subsection{Summary of the results.} In dimension $3$, smoothness of $B$ ensures that the general fibre $\chi_s$ of $\psi: \chi \rightarrow \C$ is smooth. In this article we introduce the notion of (polarized) tropical conifold, in which the discriminant locus is allowed to have $4$-valent vertices. Such vertices, which we call (tropical) nodes, are of two types: negative and positive. Away from these nodes, a tropical conifold is a smooth tropical manifold. We believe that the Gross-Siebert reconstruction theorem can be extended also to tropical conifolds, but the general fibre $\chi_s$ should be a variety with nodes. This is hinted by the fact that the local conifold ${xy-wz=0}$ has a pair of torus fibrations which induce on the base $B$ the same structure as in a neighborhood of positive or negative nodes. In fact, we show in Corollary \ref{fibr_conifold} that if $B$ is a tropical conifold, then $X_{B_0}$ can be topologically compactified to a topological conifold $X_B$ (i.e. a singular topological manifold with nodal singularities). An interesting observation is that the Legendre transform of a positive node is the negative node. In particular we also have the mirror conifold $X_{\check B}$. This extends topological mirror symmetry of \cite{TMS} to conifolds. Then we give a local description of the smoothing and resolution of a node in a tropical conifold (see Figures \ref{transition} and \ref{transition2}). It turns out that the Legendre dual of a resolution is indeed a smoothing. At the topological level this was already observed by Gross \cite{Gross_spLagEx} and Ruan \cite{Ruan}, who also discusses a global example.   The interesting question is global: given a compact tropical conifold, can we simultaneously resolve or smooth its nodes? We give a precise procedure to do this.  It turns out that the smoothing of nodes in a tropical conifold simultaneously induces the resolution of the nodes in the mirror. What are the obstructions to the tropical resolution/smoothing?  For this purpose we define the notion of tropical $2$-cycle inside a tropical conifold. These objects resemble the usual notion of a tropical surface as defined for instance by Mikhalkin in \cite{mikhalk_lectures}.  A tropical $2$-cycle is given by a space $S$ and an embedding $j: S \rightarrow B$ with some additional structure. The space $S$ has various types of interior and boundary points. For instance at generic points, $S$ is locally Euclidean, at the codimension $1$ points $S$ is modeled on the tropical line times an interval and at codimension $2$ points $S$ is modeled on the tropical plane (see Figure \ref{tdomain}) and so on. In Theorem \ref{good_rel_pos} we prove that if $j(S)$ contains tropical nodes, then both the vanishing cycles associated to the nodes in $X_B$ and the exceptional curves associated to the nodes in $X_{\check B}$ satisfy a good relation. The idea is that tropical $2$-cycles can be used to construct either $4$-dimensional objects in $X_B$ or $3$-dimensional ones in $X_{\check B}$ (see also Chapter 6 of \cite{diri_branes}, where the local duality between $A$-branes and $B$-branes is explained). 
Thus obstructions vanish on both sides of mirror symmetry. The results of Friedman, Tian and Smith-Thomas-Yau  then lead us to Conjecture \ref{trop_cycle_conj}. It states that any good relation among the vanishing cycles of a set of nodes in $X_B$ is a linear combination of good relations coming from tropical $2$-cycles in $B$. Moreover there should exist some property of these tropical $2$-cycles which is equivalent to the fact that $B$ can be tropically resolved. As a partial confirmation of this conjecture, we prove that the nodes contained in some special configurations of tropical $2$-cycles can always be tropically resolved (Theorems \ref{resolve_nodes_1}, \ref{resolve_nodes_2}, \ref{resolve_-ve_nodes} and Corollary \ref{resolve_nodes_cor}). 

Finally we apply these results to specific examples. We consider the case of Schoen's Calabi-Yau \cite{Schoen}, which is a fibred product of two rational elliptic surfaces. A corresponding tropical manifold has been described by Gross in 
\cite{Gross_Batirev}. It is possible to modify the example in many ways so that we obtain a tropical conifold with various nodes. We show how these nodes can be resolved/smoothed and thus obtain new tropical manifolds. The interesting fact is that this procedure automatically produces the mirror families via discrete Legendre transform and the reconstruction theorems.  For this class of examples we also partially prove Conjecture \ref{trop_cycle_conj}.

\subsection*{Notations.} We denote the convex hull of a set of points $q_1, \ldots, q_r$ in $\R^n$ by $\conv (q_1, \ldots, q_n)$. 
Given a set of vectors $v_1, \ldots, v_r \in \R^n$ the cone spanned by these vectors is the set 
\[ \cone(v_1, \ldots, v_r) = \left \{ \sum_{j=1}^{r} t_j v_j \ | \ t_j \geq 0, \ j=1, \ldots,r  \right \}.  \]

\section{Affine manifolds with polyhedral decompositions}

We give an informal introduction to affine manifolds with singularities and polyhedral decompositions. We refer to \cite{G-Siebert2003} for precise definitions and proofs.

\subsection{Affine manifolds with singularities.} Let $M \cong \Z^n$ be a lattice and define $M_\R=M\otimes_\Z \R$ and let
\[
\aff (M)=M\rtimes \Gl (\Z,n)
\]
be the group of integral affine transformations of $M_\R$. If $M$ and $M'$ are two lattices, then $\aff (M,M')$ is the $\Z$-module of integral affine maps between $M_{\R}$ and $M'_{\R}$. Recall that an integral affine structure $\mathscr A$
on an $n$-manifold $B$ is given by an open cover $\{U_i\}$ and an atlas of charts $\phi_i: U_i\rightarrow M_\R$ whose transition maps $\phi_j\circ\phi_i^{-1}$ are in $\aff (M)$. An integral affine manifold is a manifold $B$ with an integral affine structure $\mathscr A$. A continuous map $f: B \rightarrow B'$ between two integral affine manifolds is integral affine if, locally, $f$ is given by elements of $\aff (M,M')$. 

An {\it affine manifold with singularities} is a triple, $(B,\Delta,\mathscr A)$, where the $B$ is an $n$-manifold, $\Delta \subseteq B$ is a closed subset such that $B_0=B-\Delta$ is dense in $B$ and $\mathscr A$ is an integral affine structure on $B_0$. The set $\Delta$ is called the {\it discriminant locus}. A continuous map $f: B \rightarrow B'$ of integral affine manifolds with singularities is integral affine if $f^{-1}(B_0') \cap B_0$ is dense in $B$ and $f_{|f^{-1}(B_0') \cap B_0}: f^{-1}(B_0') \cap B_0 \rightarrow B_0'$ is integral affine. Furthermore $f$ is an isomorphism of integral affine manifolds with singularities if $f:(B, \Delta) \rightarrow (B', \Delta')$ is a homeomorphism of pairs. 
\subsection{Parallel transport and monodromy.} Given an affine manifold $B$, let $(U,\phi)\in\mathscr A$ be an affine chart with coordinates $u_1, \ldots, u_n$. Then the tangent bundle $TB$ (resp. cotangent bundle $T^{\ast}B$) has a flat connection $\nabla$ defined by
\[ \nabla \partial_{u_{j}} = 0 \ \ \ (\text{resp.} \ \ \nabla du_j = 0), \]
for all $j=1, \ldots, n$ and all charts $(U,\phi)\in\mathscr A$.  Then parallel transport along loops based at $b \in B$ gives the monodromy representation $\tilde \rho: \pi_{1}(B,p) \rightarrow \Gl(T_bB)$. Gross-Siebert also introduce the notion of holonomy representation, which is denoted $\rho$ and has values in $\aff(T_bB)$, $\tilde \rho$ coincides with the linear part of $\rho$. In the case of an affine manifold with singularities $(B,\Delta,\mathscr A)$, the monodromy representation is $\tilde \rho: \pi_1(B_0,p) \rightarrow \Gl(T_bB_0)$. 

Integrality implies the existence of a maximal integral lattice $\Lambda \subset TB_0$ (resp. $\Lambda^* \subset T^*B_0$) defined by
\begin{equation} \label{tangent_lattice}
 \Lambda|_{U} = \spn_{\Z} \inner{ \partial_{u_1}, \ldots}{\partial_{u_n}}  \ \ (\text{resp.} \ \ \Lambda^*|_{U} = \spn_{\Z} \inner{ du_1, \ldots}{du_n}). 
\end{equation}
We can therefore assume that $\tilde \rho$ has values in $\Gl(\Z, n)$. 
\subsection{Polyhedral decompositions.} \label{polyhedral_dec} Rather than recalling here the precise definition of an {\it integral affine manifold with singularities and polyhedral decompositions} (i.e. Definition 1.22, op. cit.), it is better to recall the standard procedure to construct them (see Construction 1.26, op. cit.). We start with a finite collection $\mathcal P'$ of $n$-dimensional integral convex polytopes in $M_\R$. The manifold $B$ is formed by gluing together the polytopes of $\mathcal P'$ via integral affine identifications of their proper faces. Then $B$ has a cell decomposition whose cells are the images of faces of the polytopes of $\mathcal P'$. Denote by $\mathcal P$ this set of cells. We assume that $B$ is a compact manifold without boundary. We now construct the integral affine atlas $\mathcal A$ on $B$. 
First of all, the interior of each maximal cell of $\mathcal P$ can be regarded as the domain of an integral affine chart, since it comes from the interior of a polytope in $M_{\R}$. To define a full atlas we need charts around points belonging to lower dimensional cells.
In fact this will be possible only after removing from $B$ a set $\Delta'$ which we now define. 
Let $\bary (\mathcal P)$ be the first barycentric subdivision of $\mathcal P$. Then define $\Delta'$ to be the union of all simplices of $\bary (\mathcal P)$ not containing a vertex of $\mathcal P$ (i.e. $0$-dimensional cells) or the barycenter of a maximal cell.  For a vertex $v \in \mathcal P$, let $W_v$ be the union of the interiors of all simplices of $\bary (\mathcal P)$ containing $v$. Then $W_v$ is an open neighborhood of $v$ and 
\[ \{ W_v \, | \, v \ \text{is a vertex of} \ \mathcal P \} \cup \{ \inter (\sigma) \,| \, \sigma \in \mathcal P_{\max} \} \]
forms a covering of $B - \Delta'$. A chart on the open set $W_v$ is given by a {\it fan structure} at the vertex $v$ (see Construction 1.26, op. cit.). This construction gives an integral affine atlas on $B - \Delta'$. In many cases the set $\Delta'$ is too crude and the affine structure can be extended to a larger set than $B - \Delta'$. This can be done as follows. Notice that $\Delta'$ is a union of codimension $2$ simplices. Then let $\Delta$ be the union of those simplices around which local monodromy is not trivial. In Proposition 1.27 of op. cit. it is proved that the affine structure on $B- \Delta'$ can be extended to $B - \Delta$.

Gross and Siebert also introduce the crucial notion of {\it toric} polyhedral decomposition. Essentially this condition establishes certain compatibilities between fans $\Sigma_v$ and $\Sigma_w$ at vertices $v$ and $w$ lying on some common cell. We will come back to this in the next two paragraphs. 

\subsection{Local properties of monodromy.} \label{mon_properties} The monodromy representation of affine manifolds with polyhedral decompositions has some useful distinguished properties, which we now describe. First of all notice that $\Delta$ is contained in the codimension $1$ skeleton. Let $\tau$ be a cell of $\mathcal P$ of codimension at least $1$, then it is shown in op. cit. Proposition 1.29 that the tangent space to $\tau$ is monodromy invariant with respect to the local monodromy near $\tau$. More precisely, there exists a neighborhood $U_{\tau}$ of $\inter(\tau)$ such that, if $b \in \tau - \Delta$ and $\gamma \in \pi_1(U_{\tau}- \Delta, b)$, then $\tilde \rho(\gamma)(w) = w$ for every $w$ tangent to $\tau$ in $b$. Moreover (see op. cit. Proposition 1.32) the polyhedral subdivision is toric if and only if for every $\tau$ there exists a neighborhood $U_{\tau}$ of $\inter(\tau)$ such that, if $b \in \tau - \Delta$ and $\gamma \in \pi_1(U_{\tau}- \Delta, b)$, then $\tilde \rho(\gamma)(w) - w$ is tangent to $\tau$ for every $w \in T_bB_0$.

\subsection{Quotient fans.} \label{qfan} If the polyhedral decomposition is toric (see above), then to every cell $\tau \in \mathcal P$ one can associate a complete fan $\Sigma_{\tau}$, called the quotient fan of $\tau$, whose dimension is equal to the codimension of $\tau$. It is defined as follows. Let $b \in \inter(\tau) - \Delta$, then to every $\sigma$ such that $\tau \subset \sigma$, one can associate the tangent wedge of $\sigma$ at $b$. This can be viewed as a convex rational polyhedral cone inside $T_bB_0$ (with lattice structure given by $\Lambda$). The union of all such cones forms a complete fan in $T_bB_0$, which is the pull back of a complete fan in the quotient $T_bB_0 / T_b \tau$. Let $\Sigma_{\tau}$ be such a fan. The toric condition ensures that the quotient spaces $T_bB_0 / T_b \tau$ and the fan $\Sigma_{\tau}$ are independent of $b \in \tau - \Delta$, in fact they can be all identified via parallel transport along paths contained in a suitably small neighborhood $U_{\tau}$ of $\inter(\tau)$. The local properties of monodromy, assuming the toric condition, imply that this identification is independent of the chosen path.

\subsection{Examples.} In the following examples $B$ will be allowed to have boundary or even to be constructed using unbounded polytopes. The construction above can be easily adapted to these cases. 

\begin{ex} {\bf (The focus-focus singularity)}  \label{ff_affine} Here the dimension is $n=2$. The set $\mathcal P'$ is given by two polytopes: a standard simplex and a square $[0,1] \times [0,1]$. Glue them along one edge to form $B$ (see Figure \ref{affine_ff}). Let $e$ be the common edge, and let $v_1$ and $v_2$ be the vertices of $e$. The discriminant locus $\Delta$ consists of the barycenter of $e$.  Consider the fan in $\R^2$ whose $2$-dimensional cones are two adjacent quadrants, i.e. $\cone(e_1, e_2)$ and $\cone(e_1, - e_2)$, where $\{ e_1,e_2 \}$ is the standard basis of $\R^2$. Then the fan structure at $v_j$, $j=1,2$, identifies the tangent wedges of the two polytopes with these two cones, in such a way that the primitive tangent vector to $e$ at $v_j$ is mapped to $e_1$ (see Figure \ref{affine_ff}). 
\begin{figure}[!ht] 
\begin{center}
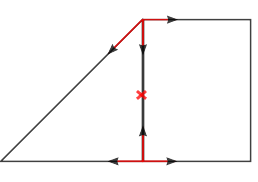
\caption{} \label{affine_ff}
\end{center}
\end{figure}
Consider a loop $\gamma$ which starts at $v_1$, goes into the square, passes through $v_2$ and comes back to $v_1$ while passing inside the triangle. One can easily calculate that $\tilde \rho(\gamma)$, computed with respect to the basis $\{e_1, e_2 \}$, as depicted in Figure \ref{affine_ff}, is the matrix 
\[
\left(
\begin{array}{cc}
  1&  1  \\
 0 &  1   
\end{array}
\right)
\]
The singular point $\Delta$ in this example is called the focus-focus singularity.
\end{ex}

\begin{ex}  {\bf (Generic singularity)} \label{generic_aff}
This a $3$-dimensional example and it is just the product of the previous example by $[0,1]$. Here $\Delta$ consists of a segment. 
\end{ex}

\begin{ex} {\bf (The negative vertex)} \label{-v}
\begin{figure}[!ht] 
\begin{center}
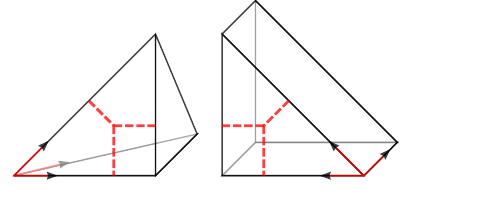
\caption{} \label{-vertex}
\end{center}
\end{figure}
Here $n=3$. Let $P^n$ be the standard simplex in $\R^n$. The set $\mathcal P'$ consists of two polytopes: $P^3$ and $P^2 \times P^1$ which we glue by identifying the triangular face $P^2 \times \{ 0 \}$ with a face of $P^3$. In Figure \ref{-vertex} we have labeled the vertices of these two faces by $v_1, v_2, v_3$ and the identification is done by matching the vertices with the same labeling. 
Now consider the fan in $\R^3$ whose cones are two adjacent octants (i.e. $\cone(e_1, e_2, e_3)$ and $\cone(e_1, e_2, -e_3)$, where $\{e_1, e_2, e_3 \}$ is the standard basis of $\R^3$). At every vertex $v_j$ identify the tangent wedges of the two polytopes with these two cones, in such a way that the tangent wedge to the common face is mapped to $\cone(e_1, e_2)$. There is more than one way to do this (since $\cone(e_1, e_2)$ has non trivial automorphisms), but any choice is a good chart of the affine structure. If one fixes an orientation then a choice can be made so that the chart is oriented. 
The discriminant locus $\Delta$ is the ``Y" shaped figure depicted in (red) dashed lines in Figure \ref{-vertex}. Now let $\gamma_j$ be the path going from $v_3$ to $v_j$ by passing into $P^3$ and then coming back to $v_3$ by passing into $P^2 \times P^1$. It can be easily shown that $\tilde \rho(\gamma_1)$ and $\tilde \rho(\gamma_2)$ are given respectively by the 
matrices
\[
\left(
\begin{array}{ccc}
  1&  0& 1  \\
  0& 1  &  0 \\
  0& 0 &   1
\end{array}
\right), \ \ \ 
\left(
\begin{array}{ccc}
 1& 0  &  0 \\
 0&  1 & 1  \\
 0&  0 &   1
\end{array}
\right)
\]
The vertex of $\Delta$ in this example is called the {\it negative vertex}.
\end{ex}

\begin{ex}  {\bf (The positive vertex)} \label{+v}
\begin{figure}[!ht] 
\begin{center}
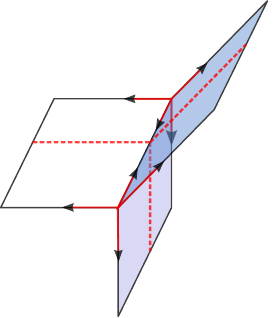
\caption{} \label{+vertex}
\end{center}
\end{figure}
In this case $B = \R^2 \times [0,1]$ with polyhedral decomposition given by the following unbounded polytopes (see Figure \ref{+vertex}):
\[ Q_1 = \{x \geq \max \{y,0 \}, z \in [0,1] \}, \ \ \ Q_2 = \{y \geq \max \{x,0 \}, z \in [0,1] \}, \]
\[ Q_3 = \{x \leq 0, y \leq 0, z \in [0,1] \}.  \]
Let $\Sigma_0$ be the fan whose maximal cones are
\[ \cone(e_1, e_3, -e_1-e_2-e_3), \ \ \cone(e_1,e_2, -e_1-e_2-e_3), \ \ \cone(e_1, e_2, e_3). \]
Then the fan structure at $p_0$ identifies the tangent wedges of $Q_1, Q_2$ and  $Q_3$ at $p_0$ with the first, second and third cone respectively. 
Now let $\Sigma_1$ be the fan whose maximal cones are
\[ \cone(e_1, e_3, -e_2-e_3), \ \ \cone(e_1,e_2, -e_2-e_3), \ \ \cone(e_1, e_2, e_3). \]
The fan structure at $p_1$ identifies the tangent wedges of $Q_1, Q_2, Q_3$ at $p_1$ with the first, second and third cone respectively. Now let $\gamma_j$, $j=1,2$, be the loop which starts at $p_0$, goes to $p_1$ by passing inside $Q_3$ and then comes back to 
$p_0$ by passing inside $Q_j$. Then we have that $\tilde \rho(\gamma_1)$ and $\tilde \rho(\gamma_2)$ are given respectively by the following matrices
\[
\left(
\begin{array}{ccc}
  1&  1& 0  \\
  0& 1  &  0 \\
  0& 0 &   1
\end{array}
\right), \ \ \ 
\left(
\begin{array}{ccc}
 1& 0  &  1 \\
 0&  1 & 0  \\
 0&  0 &   1
\end{array}
\right).
\]
The vertex of $\Delta$ in this example is called the {\it positive vertex}.
\end{ex}

\subsection{MPL-functions} A multi-valued piecewise linear (MPL) function on an affine manifold with singularities $B$ and polyhedral subdivision $\mathcal P$ generalizes the notion of piecewise linear function on a fan $\Sigma$ in toric geometry. Let $U \subset B$ be an open subset. A continuous function $f:U\rightarrow \R$ is said to be {\it (integral) affine} if it is (integral) affine when restricted to $U \cap B_0$. The sheaf of integral affine functions (or just affine) is denoted by $ \caff (B, \Z)$ (resp. $\caff_{\R}(B, \R)$). Notice, for example, that an affine function defined in a neighborhood of the singularity in Example \ref{ff_affine} must be constant along the edge $e$ which contains it. An (integral) {\it piecewise linear} (PL) function on $U$ is a continuous function $f: U \rightarrow \R$ which is (integral) affine when restricted to $U \cap \inter(\sigma)$ for every maximal cell $\sigma \in \mathcal P$ and satisfies the following property: for any $y \in U$, $y \in \inter(\sigma)$ for some $\sigma \in \mathcal P$, there exists a neighborhood $V$ of $y$ and an (integral) affine function $f$ on $V$ such that $\phi-f$ is zero on $V \cap \inter(\sigma)$. Notice that this latter property implies that a PL-function on a neighborhood of the singularity in Example \ref{ff_affine} is constant when restricted to the edge $e$. The sheaf of integral (or just affine) PL-functions is denoted by $	\mathcal{PL}_{\mathcal P}(B, \Z)$ (resp. $\mathcal{PL}_{\mathcal P, \R}(B, \R)$).

When $\mathcal P$ is a toric polyhedral subdivision, then a PL-function satisfies the following property. Given $\sigma \in \mathcal P$ and $y \in \inter(\sigma)$, then, in a neighborhood $U$ of $y$, there is an affine function $f$ such that $\phi-f$ is zero on $U \cap \inter(\sigma)$. This implies that $\phi-f$ descends to a PL-function (in the sense of toric geometry) on the quotient fan $\Sigma_{\sigma}$ defined in \S \ref{qfan}. We denote this function by $\phi_{\sigma}$ and we call it the {\it quotient function}. 

The sheaf $\mathcal{MPL}_{\mathcal P}$ of integral {\it MPL-functions} is defined by the following exact sequence of sheaves
\[ 0 \longrightarrow \caff(B, \Z) \longrightarrow \mathcal{PL}_{\mathcal P}(B, \Z) \longrightarrow \mathcal{MPL}_{\mathcal P} \longrightarrow 0. \]
Given a toric polyhedral subdivision $\mathcal P$ on $B$, an MPL-function $\phi$ on $B$ is said to be {\it (strictly) convex} with respect to $\mathcal P$ if $\phi_{\tau}$ is a (strictly) convex piecewise linear function on the fan $\Sigma_{\tau}$ for every $\tau \in \mathcal P$.  We can now give the following 

\begin{defi} A {\it (polarized) tropical manifold} is a triple $(B, \mathcal P, \phi)$ where $B$ is an integral affine manifold with singularities, $\mathcal P$ a toric polyhedral decomposition and $\phi$ a strictly convex MPL-function with respect to $\mathcal P$ (the polarization). 
\end{defi}

It is worth to point out that this notion of tropical manifold differs from other notions appearing elsewhere in the literature, such as Mikhalkin's tropical varieties in \cite{mikhalk_lectures}. Gross-Siebert's tropical manifolds can be seen as ambient spaces where Mikhalkin's tropical varieties can be embedded. 

\subsection{The discrete Legendre transform.}  Given a tropical manifold $(B, \mathcal P, \phi)$, the discrete Legendre transform produces a second tropical manifold $(\check B, \check{\mathcal P}, \check{\phi})$. Topologically the pair $(\check B, \check \Delta)$ is homeomorphic to the pair $(B, \Delta)$ and the decomposition $\check{\mathcal P}$ is the standard dual cell decomposition (in the sense of topological cell decompositions). What changes is the affine structure.

Given a polytope $\sigma \in \mathcal P$, the MPL-function $\phi$ gives a strictly convex piecewise linear function $\phi_{\sigma}$ on the fan $\Sigma_{\sigma}$. Then we let $\check \sigma$ be the standard Newton polytope associated to the pair $(\Sigma_{\sigma}, \phi_{\sigma})$. Its dimension is equal to the codimension of $\sigma$. Recall that there is an inclusion reversing correspondence between $k$-dimensional cones of $\Sigma_{\sigma}$ and faces of $\check \sigma$ of codimension $k$. Then we let $\check{\mathcal P}' = \{ \check v \, | \, v \ \text{a vertex of} \ \mathcal P \}$. The manifold $\check B$ is obtained from $\check{\mathcal P}'$ as follows.
Suppose that $\tau$ is an edge of $\mathcal P$ having $v$ and $w$ as vertices.  Then, the properties of $\phi$ ensure that $\check \tau$ is isomorphic to an $(n-1)$-dimensional face of both $\check v$ and $\check w$. Thus $\check v$ and $\check w$ can be glued along these faces using the isomorphism with $\check \tau$. It can be shown that the space produced from $\check{\mathcal P}'$ via these gluings is a manifold $\check B$ homeomorphic to $B$, with induced cell decomposition $\check{\mathcal  P}$. 

It remains to define a fan structure at all vertices of $\check{\mathcal P}$. Notice that a vertex of $\check{\mathcal P}$ is the dual of a maximal polytope $\sigma \in \mathcal P$, thus we denote it by $\check \sigma$. The fan $\Sigma_{\check \sigma}$ at $\check \sigma$ is given by the normal fan of $\sigma$. One can show that there is a well defined chart $\iota_{\check \sigma}: W_{\check \sigma} \rightarrow |\Sigma_{\check \sigma}|$. This construction also gives a naturally defined MPL-function $\check \phi$. Locally this is given by the standard strictly convex PL-function $\check \phi_{\check{\sigma}}$ defined on the normal fan of a convex lattice polytope. Thus we have the Legendre dual polarized tropical manifold $(\check B, \check{\mathcal P}, \check \phi)$. 

As an example, it can be easily shown that the negative vertex (Example \ref{-v}) and the positive one (Example \ref{+v}) are related to each other via a discrete Legendre transform with respect to suitably chosen polarizations. 

\subsection{The Gross-Siebert reconstruction theorem.} \label{GSreconstruction} Gross and Siebert consider tropical manifolds which satisfy a further set of technical conditions which they call ``positive and simple".  
To avoid confusion with other uses of the word ``positive" in this paper, we will say that a (polarized) tropical manifold is {\it smooth} if it is ``positive and simple" in the Gross-Siebert sense. 
In dimension $n=2$ or $3$, smoothness amounts to the following. If $n=2$, $\Delta$ consists of a finite set of points and every point of $\Delta$ has a neighborhood which is integral affine isomorphic to a neighborhood of the focus-focus singularity in Example \ref{ff_affine}. If $n=3$, then $\Delta$ is a trivalent graph such that: $(a)$ every point in the interior of an edge of $\Delta$ has a neighborhood which is integral affine isomorphic to a neighborhood of a singular point in Example \ref{generic_aff} (b) every vertex of $\Delta$ has a neighborhood which is integral affine isomorphic to a neighborhood of the ``positive vertex" in Example \ref{+v} or of the ``negative vertex" in Example \ref{-v}. In this definition we should also allow $\Delta$ to be curved. In fact in Examples \ref{generic_aff},  \ref{+v} and \ref{-v} we could take $\Delta$ to be made of curved lines lying inside the same $2$-dimensional face and we would still have a well defined affine structure on $B - \Delta$. For the rest of this paper we restrict to dimensions $n=2$ or $3$. 

In \S 4 of \cite{G-Siebert2003}, Gross and Siebert consider a toric degeneration $\psi: \chi \rightarrow \C$ of varieties, with a relatively ample line bundle $\mathcal L$ on $\chi$ and they associate to it its dual intersection complex which has the structure of a tropical manifold $(B, \mathcal P, \phi)$. A toric degeneration has the property that the central fibre $\chi_0 = \psi^{-1}(0)$ is obtained from a disjoint union of toric varieties by identifying pairs of irreducible toric Weil divisors. This information is encoded in  $(B, \mathcal P, \phi)$, in fact to every vertex $v \in \mathcal P$ we associate the toric variety $S_v$ given by the fan $\Sigma_v$. Now to every edge $\tau$, connecting vertices $v$ and $w$, the intersection between $S_v$ and $S_w$ is the toric divisor $D_{\tau}$ given by the fan $\Sigma_{\tau}$. The polarization $\phi$ determines $\mathcal L_{|\chi_0}$. In \cite{GrSi_re_aff_cx} they prove the important reconstruction theorem
\begin{thm} \label{gs_recon} (Gross-Siebert) Every compact and smooth polarized tropical manifold arises as the dual intersection complex of a toric degeneration. 
\end{thm}
In dimension $n=3$ the generic fibre of the toric degeneration constructed in the above theorem is a smooth manifold and if $B$ is a $3$-sphere then it is also Calabi-Yau. If we consider the discrete Legendre transform $(\check B, \check{\mathcal P}, \check{\phi})$, then we can reconstruct the toric degeneration $\check \psi: \check{\chi} \rightarrow \C$. Gross and Siebert claim that this family is mirror symmetric to the family $\psi: \chi \rightarrow \C$ and provide many evidences of this. For instance Gross, in \cite{Gross_Batirev}, shows that Batyrev-Borisov mirror pairs \cite{BB} of Calabi-Yau manifolds arise in this way (see also the articles \cite{Haase-Zharkov}, \cite{Haase-ZharkovII} and \cite{Haase-ZharkovIII}). 

\section{Lagrangian fibrations}

Now consider an integral affine manifold with singularities $B$ with discriminant locus $\Delta$ and recall the definition (\ref{tangent_lattice}) of the lattice $\Lambda^* \subset T^*B_0$. Then we can define the $2n$-dimensional manifold 
\[ X_{B_0}= T^{\ast}B_0 / \Lambda^*, \]
which, together with the projection $f_0: X_{B_0}\rightarrow B_0$, forms
a $T^n$ fibre bundle. The standard symplectic form on $T^{\ast}B_0$ descends to $X_{B_0}$ and the fibres of $f_0$ are Lagrangian. Clearly the monodromy representation $\tilde \rho$ associated to the flat connection on $T^*B_0$ is also the monodromy of the local system $\Lambda^*$. A ``symplectic compactification" of $X_{B_0}$ is a symplectic manifold $X_B$, together with a surjective Lagrangian fibration $f: X_B \rightarrow B$ such that we have the following commutative diagram
\begin{equation} \label{compactify}
\begin{array}{ccc}
X_{B_0} & \hookrightarrow & X_B \\ 
\downarrow & \  & \downarrow \\ 
B_0& \hookrightarrow & B
\end{array}
\end{equation}
where the vertical arrows are the fibrations and the upper arrow is an open symplectic embedding. We will restrict our attention to the $3$-dimensional case $n=3$. In this case, in \cite{CB-M} we proved that $X_B$ and $f$ can be constructed under the assumption that $B$ is smooth (in the sense of \S \ref{GSreconstruction}). The precise statement of the result in \cite{CB-M} is slightly more delicate due to the fact that near negative vertices the discriminant locus $\Delta$ has to be perturbed so that it has a small codimension $1$ part. We will explain more about this later. Our symplectic construction of $X_B$ is based on the topological construction carried out by Gross in \cite{TMS}. It is expected that $X_B$ should be symplectomorphic to a generic fibre of the reconstructed toric degeneration of Theorem \ref{gs_recon} with some K\"ahler form, where the mirror $\check B$ is the dual intersection complex. In \cite{TMS}, in the case of the quintic $3$-fold in $\PP^4$, Gross proved that $X_B$ is diffeomorphic to a generic quintic and $X_{\check B}$ is diffeomorphic to its mirror (see also \cite{Gross_Batirev}, Theorem 0.1).

The fibration $f$ will have three types of singular fibres: \textit{generic-singular fibres} over edges of $\Delta$; \textit{positive fibres} and \textit{negative fibres} respectively over positive and negative vertices of $\Delta$. Let $U \subseteq B$ be a small open neighborhood, homeomorphic to a $3$-ball, of either an edge, a positive or negative vertex. The idea is to find standard local models of fibrations $f_U: X_U \rightarrow U$, such that if we let $U_0 = U \cap B_0$ and $X_{U_0}= f_U^{-1}(U_0)$, then $f_U: X_{U_0} \rightarrow U_0$ has the structure of a $T^3$-fibre  bundle, i.e. $X_{U_0} = \mathcal E_{U_0} / \Lambda_{U_0}$, where $\mathcal E_{U_0}$ is a rank $3$-vector bundle over $U_0$ and $\Lambda_{U_0}$ is a maximal lattice.  Then, topologically, in order to glue $f_U: X_U \rightarrow U$ to  $f_0: X_{B_0} \rightarrow B_0$ it is enough to show that $\mathcal E_{U_0} / \Lambda_{U_0}$ and $f_0^{-1}(U_0)$ are isomorphic as $T^3$-bundles, i.e. that they have the same monodromy. When $f_U$ is Lagrangian then an isomorphism is provided by action angle coordinates.
\subsection{Local models.} \label{local_fib} We sketch the construction of the local models, for details see \cite{TMS}{\S 2} and \cite{CB-M}. We will denote local fibrations by $f: X \rightarrow U$, instead of the more cumbersome $f_U: X_U \rightarrow U$. The examples will satisfy the following properties:
\begin{itemize}
\item[a)] $X$ is a (real) $6$-dimensional manifold with an $S^1$ action such that $X / S^1$ is a $5$-dimensional manifold. We denote $Y = X / S^1$. 
\item[b)] If $\pi: X \rightarrow Y$ is the projection, the image of the fixed point set of the $S^1$-action is an oriented $2$-dimensional submanifold $\Sigma \subset Y$. 
\item[c)] Let $Y'=Y-\Sigma$. If $X' = \pi^{-1}(Y')$, then $\pi: X' \rightarrow Y'$ is a principal $S^1$-bundle over $Y'$ such that the Chern class $c_1$, evaluated on a small unit sphere in the fibre of the normal bundle of $\Sigma$ is $\pm 1$. 
\item[d)] there exists a regular $T^2$ fibration $\bar f: Y \rightarrow U$  such that $f: X \rightarrow U$ is given by $f:=\bar f \circ\pi$.
\end{itemize}
Clearly the discriminant locus of $f$ will be $\Delta:=\bar f(\Sigma )$. One can readily see that for $b\in\Delta$, the singularities of the fibre $X_b$ occur along $\Sigma\cap \bar f^{-1}(b)$.

The prototypical example of $X$ with an $S^1$ action satisfying $(a)-(c)$ is given by $X= \C^3$ and $S^1$ action given by
\begin{equation} \label{circle_act}
\xi  \cdot (z_1, z_2, z_3) = (\xi z_1, \xi^{-1} z_2, z_3). 
\end{equation}
The quotient \ $Y$ can be identified with \ $\C^2 \times \R$ \ and the map \ $\pi$ \ with $\pi(z) = (z_1z_2, \ z_3, \ |z_1|^2-|z_2|^2)$. Clearly $ \Sigma = \{ (0, u, 0) \in \C^2 \times \R \}$.


\begin{rem} In the examples we will have $Y = T^2 \times U$, where $U$ is homeomorphic to a $3$-ball. In particular the fibres of $Y$ have a linear structure. This fact, together with the $S^1$ action on $X$ and a choice of a section $\sigma_0: U \rightarrow X$, implies that $f^{-1}(U- \Delta)$ has the structure of a $T^3$-fibre bundle $\mathcal E / \Lambda$. 
\end{rem}

\subsection{Generic singular fibration}  \label{ex. (2,2)}
We describe the model for the fibration over a neighborhood $U$ of an edge of $\Delta$. Let $U=D\times (0,1)$, where $D \subset \C$ is the unit disc, and let $Y=T^2\times U$. Let $\Sigma\subset Y$ be a cylinder defined as follows. Let $e_2,e_3$ be a basis of $H_1(T^2,\Z)$. Let $S^1\subset T^2$ be a circle representing the homology class $e_3$. Define $\Sigma=S^1\times\{ 0\}\times (0,1)$. Now, one can construct a manifold $X$ together with an $S^1$ action and a map $\pi: X \rightarrow Y$, such that $X, Y, \Sigma$ and $\pi$ satisfy properties (a)-(c) above (see Proposition 2.5 of \cite{TMS}). We define $f=\bar f\circ\pi$, where $\bar f:Y\rightarrow U$ is the projection. Then $f$ is a $T^3$ fibration with singular fibres homeomorphic to $S^1$ times a fibre of type $I_1$, i.e. a pinched torus, lying over $\Delta:=\{ 0\}\times (0,1)$. 
If $e_1$ is an orbit of the $S^1$ action, one can take $e_1,e_2, e_3$ as a basis of $H_1(X_b, \Z)$, where $X_b$ is a regular fibre. 

In this basis the monodromy associated to a simple loop around $\Delta$ is
\begin{equation}\label{eq matrix g}
T=\left( \begin{array}{ccc}
                 1 & 1 & 0 \\
                 0 & 1  & 0 \\
                 0 & 0  & 1 \end{array} \right).
\end{equation}
If one also chooses a section, then the set of smooth fibres $X_0 = f^{-1}(U - \Delta)$ has the structure of a $T^3$-fibre bundle $\mathcal E / \Lambda$. 

An explicit Lagrangian fibration with this topology is defined as follows. Let
 \begin{equation} \label{edge_space}
  X = \{ (z_1, z_2, z_3) \in \C^2 \times \C \, | \, z_1z_2 -1 \neq 0, \ z_3 \neq 0 \} 
 \end{equation}
with the standard symplectic form induced from $\C^3$ and $U = \R^3$. Define $f: X \rightarrow U$ to be
\begin{equation} \label{edge_fib}
 f(z) = (\log|z_1z_2 - 1|, |z_1|^2 - |z_2|^2,  \log |z_3|), 
\end{equation}
The discriminant locus is $\Delta = \{x_1 = x_2 = 0 \}$. The $S^1$ action is given by (\ref{circle_act}).
The quotient space can be identified with $Y= \R \times (\C^*)^2$ and the map $\pi: X \rightarrow Y$ with $\pi (z_1,z_2, z_3)=(|z_1|^2-|z_2|^2, z_1z_2-1, z_3)$.
Here 
\[ \Sigma = \Crit f =  \{ (0,-1,u), \ u\in \C^* \}. \] 
 Notice that $f$ is actually invariant with respect to the $T^2$ action given by 
\begin{equation} \label{torus_act}
\xi \cdot (z_1, z_2, z_3) = (\xi_1z_1, \xi_1^{-1} z_2, \xi_2 z_3),
\end{equation}
for $\xi =(\xi_1, \xi_2) \in T^2$. The second and third components of $f$ give the moment map with respect to this action.  Then $X/ T^2$ can be identified with $(\C - \{1 \}) \times \R^2$, with coordinates $(u, t_1, t_2)$ and the projection $\pi_{T^2}: X \rightarrow X/T^2$ is given by 
\[ \pi_{T^2}(z_1, z_2, z_3) = (z_1z_2, |z_1|^2- |z_2|^2, \log |z_3|). \]

\subsection{The negative fibration} \label{ex. (2,1)}
This example is a model of a fibration over a neighborhood $U$ of a negative  vertex of $\Delta$. We give two possible versions, which are topologically equivalent. The first one is defined as follows. Let 
\[ \bar Y=T^2 \times \R^2 \]
Define $\Delta \subset \R^2$ to be $\Delta= \{ b_0 \} \cup \Delta_1\cup\Delta_2\cup\Delta_3$ where:
\begin{eqnarray} \label{Yshape}
\ &  b_0 = (0,0),  & \ \\ 
\Delta_1 = \{ (-t, 0) \ | \ t > 0 \}, \ &  \Delta_2 = \{ (0, -t) \ | \ t > 0 \}, \  & \Delta_3 = \{ (t, t) \ | \ t > 0 \}. \nonumber
\end{eqnarray}
So that $\Delta$ is a graph with a trivalent vertex $b_0$ and three legs, $\Delta_i$, $i=1,2,3$ (the shape of a letter ``Y"). Fix a basis $e_2$, $e_3$ for $H_1(T^2,\Z)$. Define $\Sigma \subset \bar Y$ to be a ``pair of pants'' lying over $\Delta$ such that for $i=1,2,3$, $\Sigma \cap (T^2\times\Delta_i)$  is the cylinder $S^1 \times \Delta_i$, where $S^1$ is a circle in $T^2$ representing the classes $-e_3$, $-e_2$ and $e_2+e_3$ respectively. These legs can be glued together over the vertex $b_0$ of $\Delta$ in such a way that $\Sigma \cap (T^2 \times \{ b_0 \})$ is a figure eight curve. Now let 
\[ Y = \bar Y \times \R = T^2 \times \R^3 \]
and identify $\bar Y$ with $\bar Y \times \{ 0 \}$ and $\Delta \subset \R^2$ with $\Delta \times \{ 0 \} \subset \R^2 \times \R$. Now, one can construct a manifold $X$ with an $S^1$ action and a map $\pi:X \rightarrow Y$ satisfying the properties (a)-(c) above. Consider the trivial $T^2$ fibration $\bar f:Y\rightarrow \R^3$ given by projection. The composition $f=\bar f\circ\pi$ is $3$-torus fibration. For $b\in\Delta$ the fibre $X_b$ is singular along $\bar f^{-1}(b)\cap \Sigma $. Thus the fibres over $\Delta_i$ are homeomorphic to $I_1\times S^1$, whereas the central fibre, $X_{b_0}$, is singular along the figure eight curve. We can take as 
a basis of $H_1(X_b,\Z )$, $e_1(b), e_2(b),e_3(b)$, where $e_2$ and $e_3$ are the 1-cycles in $\bar f^{-1}(b)=T^2$ as before and $e_1$ is a fibre of the $S^1$-bundle. In this basis, the monodromy matrices associated to suitable loops about the legs $\Delta_i$ are 
\begin{equation}\label{neg_mon}
T_1=\left( \begin{array}{ccc}
                 1 & 1 & 0 \\
                 0 & 1  & 0 \\
                 0 & 0  & 1 
              \end{array} \right),\quad 
T_2=\left( \begin{array}{ccc}
                 1 & 0 & 1 \\
                 0 & 1  & 0 \\
                 0 & 0  & 1 
              \end{array} \right),\quad 
T_3=\left( \begin{array}{ccc}
                 1 & 1 & 1 \\
                 0 & 1  & 0 \\
                 0 & 0  & 1 
\end{array}\right). 
\end{equation}
\medskip
We now describe the second version. It is defined over 
\[ X =  \{ z_1z_2 + z_3 -1 \neq 0 \} \cap \{ z_1z_2 - z_3 \neq 0  \} \]
by the function
\[ f(z_1, z_2, z_3) = (|z_1|^2-|z_2|^2, \log |z_1z_2 + z_3 -1|, \log |z_1z_2 - z_3|). \]
If we consider the $S^1$ action (\ref{circle_act}) and the associated projection $\pi(z_1,z_2,z_3) =(|z_1|^2-|z_2|^2, z_1z_2, z_3)$ onto the quotient space. Then $f = \bar f \circ \pi$, where 
\[ \bar f: (t, u) \mapsto (t, \log |u_1 + u_2 -1|, \log |u_1 - u_2 | ). \]
which is a $T^2$ fibration. Notice that in this case 
\[ \Sigma = \{ t=0, u_1 = 0 \}. \] 
\begin{figure}[!ht] 
\begin{center}
\includegraphics{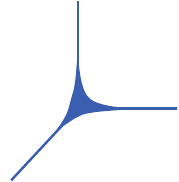}
\caption{}\label{3leg_amoeba}
\end{center}
\end{figure}
The discriminant locus $\Delta$ is $\bar f(\Sigma)$, which can be seen to be a codimension $1$ thickening of the letter ``Y'' graph (i.e. the amoeba of a line). It is not difficult to perturb this fibration so that the ends of the legs of $\Delta$ become pinched to codimension $2$ (see Figure \ref{3leg_amoeba}). With more effort one can make this fibration into a piece-wise smooth Lagrangian one (see \cite{CB-M}, Section 5). We emphasize that from a topological point of view this latter version of negative fibration is equivalent to the first version with $\Delta$ of codimension $2$. In fact one can show that the surface $\Sigma$ in this latter version is isotopic to a surface which maps to a genuine codimension $2$ ``Y'' shaped $\Delta$ constructed as in the first version (see Theorem 4.3 in \cite{ruan_newton}, \cite{mikh_pants} and Section 4 of \cite{Gross_spLagEx}). Also in this example the set of smooth fibres $X_0 = f^{-1} (\R^3 - \Delta)$ has the structure of a $3$-torus bundle $\mathcal E / \Lambda$. 




\subsection{The positive fibration} \label{ex. (1,2)}
In this case we give an explicit fibration which turns out to be also Lagrangian. Here $X =  \{ z_1z_2 z_3 - 1 \neq 0\}$, with standard symplectic form induced from $\C^3$. The fibration $f: X \rightarrow \R^3$ is defined by
\begin{equation}
 f(z_1, z_2, z_3) = (\log |z_1z_2 z_3 - 1|, |z_1|^2-|z_2|^2, |z_1|^2-|z_3|^2). 
\end{equation}
Identify $\R^2$ with $\{ 0 \} \times \R^2 \subset \R^3$. The discriminant locus $\Delta$ is contained in $\R^2$ and 
it is given by $\Delta = \{b_0\} \cup \Delta_1 \cup \Delta_2 \cup \Delta_3$, where $b_0$ and the $\Delta_j$'s are as in (\ref{Yshape}). Notice that $f$ is invariant with respect to the $T^2$ action 
\begin{equation} \label{pos_t2_action} 
\xi \cdot (z_1, z_2, z_3) = (\xi_1^{-1} z_1, \xi_2^{-1} z_2, \xi_1 \xi_2 z_3). 
\end{equation}
The quotient $X/T^2$ can be identified with $\C^* \times \R^2$ via the map 
\begin{equation} \label{+ve_t2pi}
 \pi(z_1, z_2, z_3)=(z_1z_2z_3-1, |z_1|^2-|z_2|^2, |z_1|^2-|z_3|^2) 
\end{equation}
and $f$ is the composition of $\pi$ with  $\bar f: \C^* \times \R^2 \rightarrow \R^3$ given by 
\begin{equation} \label{+ve_t2bf}
 \bar f(u, x_1, x_2) = (\log |u|, x_1, x_2).
\end{equation}
 The $T^2$ orbits are generically isomorphic to $T^2$. Consider the following submanifolds of $X$:
 \[ L_1 = \{ z_1=z_3= 0, z_2 \neq 0 \}, \ \ L_2 = \{ z_1 = z_2 = 0, z_3 \neq 0 \}, \]
 \[  L_3 = \{ z_2 = z_3 = 0, z_1 \neq 0 \} .\]
 Observe that $f(L_j) = \Delta_j$. Moreover, consider the following circles in side $T^2$: $G_1=\{ \xi_2 = 1 \}$, $G_2 = \{ \xi_1 \xi_2 = 1 \}$ and $G_3 = \{ \xi_1 = 1 \}$. Then the stabilizer of points in $L_j$ is the circle $G_j$.   Obviously $(0,0,0)$ is (the unique) fixed point and $f(0,0,0) = b_0$. 
The singular fibre over the vertex of $\Delta$ can be described as $T^2 \times S^1$ after one of the $T^2$'s is collapsed to a point. The fibres over the legs of $\Delta$ are of generic singular type. Given a generic point $b \in \R^3$, choose a basis $e_1, e_2, e_3$ of $H_1(X_b, \Z)$ such that: $e_2$ and $e_3$ are represented, respectively, by the circles $G_1$ and $G_2$ with suitable orientation and $e_1$ is the fibre of $\bar f$ over $b$. Then, with respect to this basis, the monodromy matrices around the legs are the inverse transpose of the matrices in (\ref{neg_mon}).

\begin{thm} \label{section}
Given a $3$-dimensional tropical manifold $(B, \mathcal P, \phi)$, if $B$ is smooth in the sense of \S \ref{GSreconstruction}, then the above local models can be glued to $X_{B_0}$ in order to obtain a symplectic manifold $X_B$ and a Lagrangian fibration $f: X_B \rightarrow B$ such that
\begin{itemize}
 \item[a)] it fits into diagram (\ref{compactify}) provided we replace $B_0$ with a smaller open subset obtained by removing a small neighborhood of the negative vertices; 
 \item[b)] $f$ is continuous and fails to be smooth only over a small neighborhood of the negative vertices, where it is piecewise smooth;
 \item[c)] $f$ has a smooth Lagrangian section $\sigma_0: B \rightarrow X_B$ which extends the zero section on $X_{B_0}$. 
\end{itemize}
\end{thm}
 
We refer to Theorem  8.2 of \cite{CB-M} for the proof of this theorem and for the technical details. We also point out that  this result is based on Gross' topological construction of $X_B$ which is given in Theorem 2.1 of \cite{TMS}.   It is also useful to know that the local models can be modified so that the discriminant locus $\Delta$ is ``curved'',  i.e. the edges of $\Delta$ bend inside the two dimensional monodromy invariant planes that contain them (see Section 4.3 of \cite{CB-M}).

\begin{thm} \label{cekX} Given a $3$-dimensional tropical manifold $(B, \mathcal P, \phi)$ and its dual $(\check B, \check{\mathcal P}, \check{\phi})$ define the torus bundle
\[ \check X_{B_0} = TB_0 / \Lambda. \]
If we identify $(B, \Delta)=(\check B, \check \Delta)$, then $\check X_{B_0}$ and $X_{\check B_0}$ are homeomorphic $T^3$-bundles. In particular, if $B$ is smooth in the sense of \S \ref{GSreconstruction}, $\check X_{B_0}$ can be topologically compactified by gluing a positive fibre over a negative vertex and vice-versa. Denoting the compactification by $\check X_B$, we have that $\check X_{B}$ is homeomorphic to $X_{\check B}$. 
\end{thm}

The first part follows from Proposition 1.50 of \cite{G-Siebert2003} which shows that the monodromy of $\check X_{B_0}$ is the same as the monodromy of $X_{\check B_0}$. The last two sentences follow from Theorem 2.1 of \cite{TMS}.


\section{A review of conifold transitions} \label{conif_review}
\subsection{Local geometry.} Recall that an ordinary double point or {\it node} is an isolated $3$-fold singularity with local equation in $\C^4$ given by 
\begin{equation}\label{eq:std_node}
 z_1z_2-z_3z_4=0
\end{equation}
We call  the ``local conifold'' the $3$-dimensional affine variety $X_0$ defined by this equation. It has two small resolutions. One of them is given by $\pi:X\rightarrow X_0$, where
\begin{equation}\label{eq:std_res}
X=\{(z,[t_1:t_2])\in\C^4\times \PP^1\mid t_1z_1=t_2z_3,\ t_2z_2=t_1z_4\}
\end{equation}
and $\pi$ is the projection onto $\C^4$. The other resolution is obtained by exchanging $z_3$ and $z_4$ in the equations defining $X$. Recall that $X$ is the total space of the bundle $\mathcal O_{\PP^1}(-1) \oplus \mathcal O_{\PP^1}(-1)$.
A smoothing of $X_0$, is given by
\begin{equation}\label{eq:std_smooth}
Y_\epsilon=\{ z_1z_2-z_3z_4=\epsilon\}.
\end{equation}
The symplectic form on $X_0$ and $Y_{\epsilon}$ is the restriction of the standard symplectic form on $\C^4$. 
Notice that $Y_{\epsilon}$ contains a Lagrangian $3$-sphere given by 
\begin{equation} \label{van_s3}
Y_\epsilon\cap\{ z_2=\bar z_1,\ z_3=-\bar z_4\}\cong S^3
\end{equation}
which disappears as $\epsilon \rightarrow 0$. It is called the vanishing cycle of the node. 
The symplectic structure on $X$ is induced by the symplectic structure on $\C^4\times \PP^1$ which, in coordinates $(z,t)$ with $t=t_2/t_1$, is given by
\begin{equation}\label{eq:sym_str}
\frac{i}{2} \sum_{i=1}^4dz_i\wedge d\bar{z}_i + i  \frac{\delta \, dt\wedge d\bar{t}}{2\pi(1+|t|^2)^2}
\end{equation}
where $ \delta$ is the area of $\pi^{-1}(0)\cong \PP^1$.

Observe that after the following change of coordinates
\[ z_1 \mapsto w_1 +i w_2, \ \ z_2 \mapsto w_1 - i w_2, \ \ z_3 \mapsto -w_3 - i w_4, \ \ z_4 \mapsto w_3 - i w_4 \]
we can write $X_0 = \{ \sum_{j=1}^4 w_j^2 = 0 \}$ and $Y_{\epsilon} =  \{ \sum_{j=1}^4 w_j^2 = \epsilon \}$. In these coordinates, the vanishing cycle in $Y_{\epsilon}$ is given by $S^3 = \{ \im w = 0 \} $.
The cotangent bundle of $S^3$ can be written as
\[ T^*S^3 = \{ (u,v) \in \R^4 \times \R^4 \ | \ |u| =1 , \inner{u}{v} = 0 \} \]
with canonical symplectic form $\sum_j dv_j \wedge du_j$.  The important fact is that there is a symplectomorphism 
\begin{equation} \label{simplecto}
\psi: X_0 - \{ 0 \} \rightarrow T^*S^3 - \{ v = 0 \} 
\end{equation}
given explicitly by 
\[ w_j = x_j + i y_j \mapsto \left( \frac{x_j}{|x|}, -2|x|y_j  \right). \]

More generally a complex conifold is a $3$-dimensional projective (or compact K\"ahler) variety $\bar X$ whose singular locus is a finite set of nodes $p_1, \ldots, p_k$.  Given a conifold, one can try to find a small resolution $\pi: X \rightarrow \bar X$, which replaces every node with an exceptional $\PP^1$, or a smoothing $\tilde X$, which replaces a node with a $3$-sphere. Passing from $X$ to $\tilde X$ (or vice-versa) is called a {\it conifold transition}. Notice that the existence of local diffeomorphisms $X-\PP^1 \rightarrow X_0 - \{ 0 \} \rightarrow T^*S^3 - S^3$ imply that small resolutions and smoothings can always be done topologically (by surgery), but there are obstructions if one wishes to preserve either the complex or symplectic (K\"ahler) structure. 

\subsection{Complex smoothings.} A complex small resolution of a set of nodes in a complex conifold always exists, in the sense that a small resolution always has a natural complex structure such that the exceptional $\PP^1$'s are complex submanifolds.  Notice also that there are $2^k$-choices of resolutions, where $k$ is the number of nodes, since for each node we have to choices of resolutions. Although locally the two resolutions differ just by a change of variables, globally we may have topologically distinct resolutions (they are related by a flop).
On the other hand finding a complex analytic smoothing of $\bar X$ is obstructed and the obstructions were studied by Friedman \cite{friedman-trivial} and Tian \cite{Tian}.  

\begin{defi} Given a manifold $X$, we say that a set of homology classes $\eta_j \in H_r(X, \Z)$, $j=1, \ldots, k$ satisfy a {\it good relation} if there exist integers $\lambda_j \neq 0$ for $j=1, \ldots, k$ such that 
\begin{equation} \label{cx_goodrel}
   \sum_{j} \lambda_j \eta_j= 0 
\end{equation}
\end{defi}

Assuming that the resolution $X$ of $\bar X$ satisfies the $\partial \bar \partial$-lemma, Friedman and Tian proved the following. Denote by $C_j \subset X$, $j=1, \ldots, k$ the exceptional $\PP^1$'s of a resolution and by $[C_j]$ their classes in $H_2(X; \Z)$, then a complex smoothing of $\bar X$ exists if and only if for some resolution the $[C_j]$'s satisfy a good relation. 

\subsection{Symplectic resolutions.} In general, even if $\bar X$ is projective, the resolution $X$ does not have a natural K\"ahler form. This suggests that symplectic resolutions are obstructed. The problem was studied by Smith-Thomas-Yau in \cite{STY}. First they show that a symplectic conifold (op. cit. Definitions 2.3 and 2.4) can always be symplectically smoothed (op. cit. Theorem 2.7). Essentially this consists of replacing the nodes $p_1, \ldots, p_k$ with Lagrangian spheres $L_1, \ldots, L_k$ using the symplectomorphism (\ref{simplecto}).   Then they prove the following ``mirror" of the Friedman-Tian result. In a smoothing $\tilde X$ of $\bar X$, the classes $[L_j] \in H_3(\tilde X, \Z)$ satisfy a good relation if and only if there is a symplectic structure on one of the $2^k$ choices of resolutions $X$ such that the resulting exceptional $\PP^1$'s are symplectic (see  op. cit. Theorem 2.9). 

\subsection{Local collars} \label{collars} Let us now make some observations which will be useful in the proof of Theorem \ref{good_rel_pos}. Consider the following $4$-dimensional submanifold with boundary of $T^*S^3$:
\[ N_0 = \{ (u,v) \in T^*S^3 \ | \ v_1 = - \lambda u_2, \ v_2 = \lambda u_1, \ v_3 = - \lambda u_4, \ v_4 = \lambda u_3; \ \lambda \geq 0 \}. \]
Observe that $N_0$ can be regarded as half a real line bundle and we have
\[ \partial N_0 = S^3. \]
Under the symplectomorphism $\psi$ given in (\ref{simplecto}), $N_0-S^3$ is the image of the complex surface
\[ Q_0 = \{ w_1 = i w_2, \ w_3 = iw_4 \} \]
In the $z$ coordinates we have 
\begin{equation} \label{local_collar}
 Q_0 = \{ z_2 = 0, \ z_4 = 0 \}. 
\end{equation}
Clearly we could have also defined $Q_0$ to be one of the following $\{ z_2 = 0, \ z_3 = 0 \}$,  $\{ z_1 = 0, \ z_4 = 0 \}$ or $\{ z_1 = 0, \ z_3 = 0 \}$, then the closure of $\psi(Q_0)$ in $T^*S^3$ would still be a $4$-manifold bounding $S^3$, differing from $N_0$ only by a change in signs in the defining equations. We may think of $N_0$ or $Q_0$ as a ``local collar'' near the vanishing cycle. 
In particular, suppose that inside a symplectic conifold $\bar X$ there exists a $4$-manifold $S$ (without boundary) containing nodes $p_1, \ldots, p_k$, such that in local coordinates $z_1, \ldots, z_4$ around each node, $S$ coincides with $Q_0$. Then in a smoothing $\tilde X$, $S$ lifts to a $4$-dimensional manifold with boundary whose boundary is the union of the vanishing cycles $L_1, \ldots, L_k$. This follows from the proof of Theorem 2.7 of (\cite{STY}), where the node is replaced by a $3$-sphere using $\psi^{-1}$, and the fact that $\psi(Q_0) =  N_0 - S^3$.

We can similarly define a local collar near the exceptional $\PP^1$ in $X$. 
\begin{lem}
 Consider the following subset of $X_0$
\begin{equation} \label{local_collar2}
 P_0 = \{ (z,\bar z, r, s) \in \C \times \C \times \R_{\geq 0} \times \R_{\geq 0} \ | \  rs=|z|^2 \}. 
\end{equation}
Then $\pi^{-1}(P_0)$ is a real $3$-dimensional submanifold with boundary in $X$, homeomorphic to $\PP^1 \times [0,+\infty)$, whose boundary is the exceptional $\PP^1$.
\end{lem}
\begin{proof}
Clearly 
\[ \pi^{-1}(P_0) = \{(z, \bar z, r, s, [t_1:t_2]) \mid t_1z=t_2r,\ t_2 \bar z=t_1s \}. \]
Then, in the coordinate $t=t_2/t_1$, 
\[ \pi^{-1}(P_0) = \{ (rt, r \bar t, r, |t^2| r, t) \mid t \in \C, r \geq 0 \} \cong \C \times \R_{\geq 0}. \]
This proves the lemma.
\end{proof}
 Therefore suppose that inside a complex conifold $\bar X$ we can find a subset $S$, containing nodes $p_1, \ldots, p_k$, such that $S-\{p_1, \ldots, p_k \}$ is a real $3$-dimensional submanifold in $\bar X$ and, in a neighborhood of every node, $S$ coincides with $P_0$ in some local coordinates. Then $\pi^{-1}(S)$ is a real $3$-dimensional submanifold with boundary of a resolution $X$, whose boundary is the union of the exceptional curves. Hence the exceptional curves satisfy a good relation and the nodes can be smoothed. 
\subsection{Topology change.} Suppose that $X$ and $\tilde X$ are related by a conifold transition, i.e. they are respectively a resolution and a smoothing of a set of $k$ nodes $p_1, \ldots, p_k$ of a conifold $\bar X$. Let $d$ be the rank of the subgroup spanned by the homology classes of the exceptional curves  $[C_j] \in H_2(X, \Z)$ and $c$ be the rank of the subgroup spanned by the homology classes of the vanishing cycles $[S_j] \in H_3(\tilde X, \Z)$. Then we have the following formulas relating the Betti numbers of $X$ and $\tilde X$:
\begin{eqnarray}
 k & = & d + c, \nonumber \\
 b_2( \tilde X) & = & b_2(X) - d, \label{contran_top}\\
 b_3( \tilde X) & = & b_3(X) + 2c. \nonumber
\end{eqnarray}
For a proof we refer to the survey \cite{rossi-transitions} and the references therein. 

\subsection{Conifold transitions and mirror symmetry.} In \cite{morrison-look} Morrison conjectures that given $X$ and $\tilde X$ two Calabi-Yau's related by a conifold transition, then their mirror manifolds (if they exist) are also related by a conifold transition, but in the reverse direction, i.e. the mirror of the resolution $X$ is a smoothing $\tilde Y$ and the mirror of the smoothing $\tilde X$ is a resolution $Y$:
\begin{equation*}
\begin{CD}
X @>CT>> \tilde{X}\\
@VMSVV @VVMSV\\
\tilde{Y} @<CT<< Y
\end{CD} \end{equation*}
Morrison also extends the conjecture to more general ``extremal transitions" and supports it with many examples. More examples have appeared later in the literature, e.g. in \cite{conif_ms_grass}.
At the time Morrison wrote the article, the SYZ interpretation of mirror symmetry as dual special Lagrangian fibrations had just been proposed and, in the final remarks he suggests that ``such an understanding could ultimately lead to a proof of the conjecture using the new geometric definition of mirror symmetry''. To achieve this goal he suggests the ``important and challenging problem to understand how such fibrations behave under an extremal transition". The goal of our article is to take up this challenge, at least in the case of conifold transitions, and propose a general strategy using the techniques of the Gross-Siebert program and the properties of dual torus fibrations as studied in \cite{TMS} and \cite{CB-M}. We will show that the strategy works in many special cases but we believe that a more general statement should be in reach of current technologies. We also point out that in \cite{Ruan}, Ruan sketches a Lagrangian fibration on one of the examples of \cite{conif_ms_grass} constructed via a conifold transition.

\section{Explicit fibrations on the local conifold}\label{sec:algebraic}
In this section we discuss explicit torus fibrations on the local conifold $X_0$, on its smoothing $Y_{\epsilon}$ and on its small resolution $X$. These are slightly modified versions of Ruan's fibrations \cite{Ruan} (the maps here are proper, Ruan's fibrations are not) and are special cases of the examples in \cite{Gross_spLagEx}. The fibrations will be of two types: \textit{positive} fibrations, which are $T^2$ invariant, and \textit{negative} ones, which are $S^1$ invariant. We will need these local models to construct torus fibrations on topological conifolds (Corollary \ref{fibr_conifold}) and in the proof of Theorem \ref{good_rel_pos}. In particular it will be important to understand the topology of these models, such as local monodromy (Propositions \ref{mon_+node} and \ref{node_fibr}). We also have two technical paragraphs on the geometry of ``local collars'', these will be essential in the proof of Theorem \ref{good_rel_pos}, but may be skipped on first reading. 

It would be also desirable (although not essential in this paper) to have a symplectic structure on the conifolds constructed in Corollary \ref{fibr_conifold}, so that the fibrations are Lagrangian. We can achieve this, without too much effort, only in the case the conifold does not have negative nodes, since our model of the negative fibration is not Lagrangian. Therefore we also give Lagrangian models of the positive fibration.

\subsection{Positive fibrations}
We assume that the symplectic form on $X$ is given by (\ref{eq:sym_str}). Consider the $T^2$-action on $X$ which, for $\xi=(\xi_1,\xi_2)\in T^2$, is given by
\begin{equation}\label{eq:t2_action}
\xi\cdot (z_1, z_2, z_3, z_4, t) = (\xi_1z_1, \ \xi_1^{-1}z_2,  \ \xi_2z_3, \ \xi_2^{-1}z_4, \ \xi_1\xi^{-1}_2t).
\end{equation}
If we ignore the variable $t$, the same expression also gives a $T^2$ action on the conifold $X_0$ and on the smoothing $Y_{\epsilon}$. 

\begin{ex}[Resolution] \label{ex:res_alg_1} We define the fibration on $X$ and $X_0$. 
The above action is Hamiltonian and the corresponding moment map is $\mu=(\mu_1, \mu_2)$ where:
\begin{equation*}
\begin{split}
&\mu_1=|z_1|^2-|z_2|^2-\frac{\delta}{1+|t|^2}\\
&\mu_2=|z_3|^2-|z_4|^2+\frac{\delta}{1+|t|^2}.
\end{split}
\end{equation*}
Define $f_{\delta}: X \rightarrow \R^3$ by
\begin{equation}\label{eq:exp_res_1}
f_\delta(z,t)=\left(\log |z_1z_2+z_3z_4-1|,\ \mu_1(z,t),\ \mu_2(z,t)\right).
\end{equation}
The quotient of $X/T^2$ can be identified with $\C^* \times \R^2$ via the map 
\begin{equation} \label{+n_t2_proj}
\pi(z,t) = (z_1z_2 + z_3z_3 -1, \mu_1, \mu_2)
\end{equation}
 and $f_{\delta}$ is the composition of $\pi$ with $\bar f: \C^* \times \R^2 \rightarrow \R^3$ given by 
\begin{equation} \label{P_fibr}
 \bar f(u, x_1, x_2) = (\log|u|, x_1, x_2).
\end{equation}
Therefore $\Crit(f_{\delta})$ consists of the set of points where the $T^2$-action has non trivial stabilizer. Hence $\Crit (f_\delta)$ consists of the following components: 
\begin{equation}
\begin{split}
&L_0=\{(z,t)\in X\mid z=0, \ t \neq 0,  \infty \}\cong \C^*  \\
&L_j=\{(z,t)\in X \mid z_i=0, i\neq j \ \ \text{and} \ z_j \neq 0 \}\cong\C^*,\ j=1,\ldots , 4. \\
&b_1= \{ z=0, t=0 \}, \ \ b_2 = \{ z= 0, t = \infty \} 
\end{split}
\end{equation}
Notice that points of $L_2$ and $L_3$ must also satisfy $t= 0$, while points of $L_1$ and $L_4$ must satisfy $t=\infty$. These components are mapped to $\Delta\subset \{0\}\times\R^2$, where $\Delta$ is a trivalent graph  (see Figure \ref{type_I} (a)), consisting of the 5 edges $\Delta_j := f_{\delta}(L_j)$ and $2$ vertices $v_j := f_{\delta} (b_j)$, $j=1,2$, where
\begin{equation} \label{discr_resol}
\begin{split}
& \Delta_0 = \{ (0, -t, t) \, | \, 0 < t <\delta \}, \ \ \ \Delta_1 = \{ (0, t, 0) \, | \, t > 0 \},    \\
& \Delta_2 = \{ (0, -t-\delta, 0) \, | \, t > 0 \}, \ \ \  \Delta_3 = \{ (0, -\delta, t + \delta) \, | \, t > 0 \}, \\
& \Delta_4 = \{ 0, 0, -t ) | \, t > 0 \}, \ v_1 = (0, -\delta, \delta), \ \ v_2 = (0,0,0).
\end{split}
\end{equation}
The closure of $L_0$ is the exceptional $\PP^1$ of the resolution and $\Delta_0$ is the bounded edge of $\Delta$. Observe that when $\delta = 0$, we obtain a torus fibration on the conifold $X_0$, where the bounded edge $f_{\delta}(L_0)$ collapses to a point and $\Delta$ becomes a 4-valent graph. 

\begin{figure}[!ht] 
\begin{center}
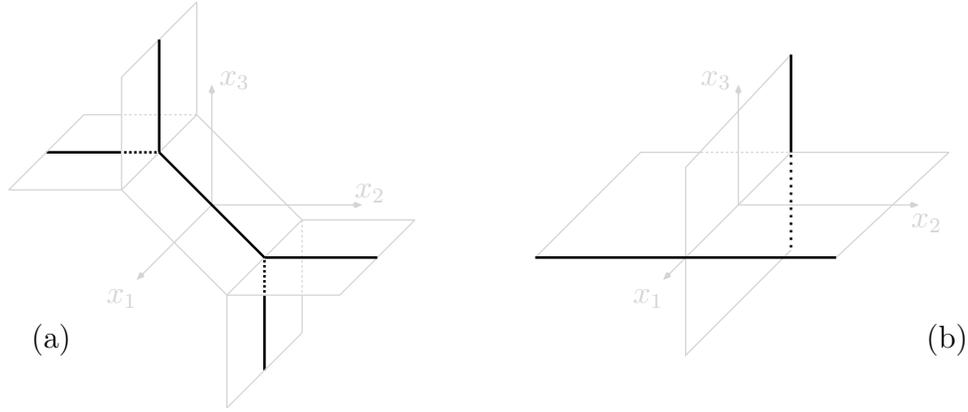
\caption{Positive fibrations: resolution (a), smoothing (b) }\label{type_I}
\end{center}
\end{figure}
\end{ex}

\begin{lem} For all $\delta = 0$, $f_{0}: X_0 \rightarrow \R^3$ has a Lagrangian section $\sigma_0$ defined on a neighborhood of the vertex of $\Delta$. \end{lem}

\begin{proof}
Choose a smooth point $p$ on the fibre over the vertex of $\Delta$. Let $(x,y) = (x_1, x_2, x_2, y_1, y_2, y_3)$ be Darboux coordinates around this point such that $p = (0,0)$ and $f_0$ is given by $(x,y) \mapsto x$. Then $\sigma_0(x) = (x,0)$ is a Lagrangian section.  
\end{proof}

\begin{prop} \label{mon_+node}
For all $\delta \geq 0$, the map $f_{\delta}$ is a  Lagrangian 3-torus fibration. The singular fibres are of generic-singular type over edges $\Delta_j$. When $\delta >0$ the fibres over the two vertices are of positive type. Moreover, given a generic fibre $X_b$, there is a basis $e_1, e_2, e_3$ of $H_1(X_b, \Z)$ and simple closed loops $\gamma_j$ around the edges $\Delta_j$, such that the monodromy $T_j$ around these loops is given by the matrices
\[ T_1 = T_2 = \left( \begin{array}{ccc}
                 1 & 0 & 1 \\
                 0 & 1  & 0 \\
                 0 & 0  & 1 
              \end{array} \right), \]
 \[ T_3 = T_4 = \left( \begin{array}{ccc}
                 1 & 0 & 0 \\
                 0 & 1  & 1 \\
                 0 & 0  & 1 
              \end{array} \right) \]
\[ T_0 = T_2^{-1}T_3^{-1}. \]
\end{prop}
\begin{proof} To show that $f_{\delta}$ is Lagrangian, one can either verify it directly or observe that it is a fibration of the type described in Theorem A of \cite{Goldstein1} (see also Theorem 1.2 of \cite{Gross_spLagEx}). In fact the $T^2$ action (\ref{eq:t2_action}) is Hamiltonian. Also, $f_{\delta}$ is a special case of the special Lagrangian fibrations on open Calabi-Yau toric varieties constructed in Theorem 2.4 of \cite{Gross_spLagEx}. The fact that the fibres over the edges are of generic-singular type and the fibres over trivalent vertices are of positive type is a consequence of Propositions 3.3, 2.9  and Example 2.10 of \cite{TMS} (where positive fibres are called of type $(1,2)$). 

We now compute the monodromy. Consider the following circles in $T^2$: $G_1 = \{ \xi _1= 1 \}$, $G_2 = \{ \xi_2= 1 \}$ and $G_3 = \{ \xi_1\xi_2^{-1} = 1 \}$. Then $G_1$ is the stabilizer of points over $L_1$ and $L_2$, $G_2$ is the stabilizer of points over $L_3$ and $L_4$ and $G_3$ is the stabilizer of points over $L_0$. For a generic fibre $X_b$ choose a basis $e_1, e_2, e_3$ of $H_1(X_b, \Z)$ such that $e_1$  and $e_2$ are represented by $G_1$ and $G_2$ respectively, with a suitable orientation, and $e_3$ is the fibre over $b$ of $\bar f$. Then monodromy around the edges $\Delta_j$ (with suitable choices of loops $\gamma_j$) is given by the matrices $T_j$ (see Example \ref{ex. (1,2)}). Notice that the loops must satisfy $\gamma_1 \gamma_3 = \gamma_4 \gamma_2$, which matches the equation $T_1 T_3= T_4 T_2$.
\end{proof}

\begin{ex}[Smoothing] \label{ex:smooth_1} We now define the fibration on $Y_{\epsilon}$.  Let $f: Y_{\epsilon} \rightarrow \R^3$ be defined by
\begin{equation}\label{eq:exp_smooth_1}
f(z)=\left(\log |z_1z_2+z_3z_4-1|,\ |z_1|^2-|z_2|^2,\ |z_3|^2-|z_4|^2\right),
\end{equation}
where the last two function components give the moment map of the torus action (\ref{eq:t2_action}). The critical locus of $f$ consists of the set of points of $Y_{\epsilon}$ where the torus action has non-trivial stabilizer. This set has two connected components:
\begin{equation*}
\begin{split}
&C_1=\{z_1=z_2=0,\ z_3z_4=-\epsilon\} \cong\C^\ast \\
&C_2=\{z_3=z_4=0,\ z_1z_2=\epsilon\} \cong\C^\ast.
\end{split}
\end{equation*}
Therefore, the discriminant locus of $f$ has two disjoint components:
\begin{equation*}
\begin{split}
&f(C_1)=\{x_1=\log|1+\epsilon|,\ x_2=0\}\cong\R, \\
&f(C_2)=\{x_1=\log|1-\epsilon|,\ x_3=0\}\cong\R,
\end{split}
\end{equation*}
depicted in Figure \ref{type_I} (b).
The singular fibres are all of generic type. We have that $f$ is Lagrangian. This follows from Theorem A of \cite{Goldstein1}, since the $T^2$ action is Hamiltonian. This fibration is also one of the examples discussed as an application of Proposition 3.3 of \cite{Gross_spLagEx}. 

The vanishing cycle (\ref{van_s3}) is mapped by $f$  to the set:
\begin{equation*}
\{x_1=\log|\epsilon-2|z_3|^2-1|,\ x_2=x_3=0, 0\leq |z_3|^2\leq\epsilon\}
\end{equation*}
which is a segment joining the components of the discriminant locus.
\end{ex}

\subsection{Local collars and the positive fibration.} In this section we prove some technical lemmas on the geometry of the ``local collars'' defined in \S \ref{collars} in terms of the positive fibration $f_0: X_0 \rightarrow \R^3$. We will need these results in the proof of Theorem \ref{good_rel_pos}. We have the following

\begin{lem} \label{collar_fib}
Let $f_{0}: X_0 \rightarrow \R^3$ be as in (\ref{eq:exp_res_1}) with $\delta=0$ and let $Q_0 \subset X_0$ be as in (\ref{local_collar}). Let 
\[ S =  \{ x_1 = 0, x_2 \geq 0, x_3 \geq 0 \} \subseteq \R^3. \]
Then $\partial S \subset \Delta$,  $f(Q_0) =S$ and there exists a map $\sigma: S \rightarrow X_0$ such that
\begin{itemize}
\item[i)] $f_0 \circ \sigma = \id_{S}$ and $\sigma(\partial S) \subset \Crit f_0$;
\item[ii)] $Q_0 = T^2 \cdot \sigma(S)$. 
\end{itemize}
\end{lem}

\begin{proof}
The fact that $f(Q_0) =S$ is obvious. Moreover $\partial S = \Delta_1 \cup \Delta_3 \cup \{v_1 \}$. It is clear that the fibres of $f_0|_{Q_0}$ are orbits of the $T^2$ action (\ref{eq:t2_action}). Now define
\begin{equation} \label{collar_sec}
  \sigma(0, x_2, x_3) = (\sqrt{x_2}, 0, \sqrt{x_3}, 0). 
\end{equation}
We have that $\sigma$ satisfies $(i)$ and $(ii)$.
\end{proof}

The following lemma is used to define a perturbation of the local collar $Q_0$. 

\begin{lem} \label{collar_pert1}
Let $f_0: X_0 \rightarrow \R^3$ be as in (\ref{eq:exp_res_1}) with $\delta=0$, $S$ as in Lemma \ref{collar_fib} and $\sigma_0: \R^3 \rightarrow X_0$ a section. For any open neighborhood $U$ of $\Delta$, there exists a smaller neighborhood $V \subseteq U$ of $\Delta$ and a map $\sigma': S \rightarrow X_0$ such that $f_0 \circ \sigma' = \id_S$ and 
\begin{itemize}
\item[i)] $T^2 \cdot \sigma'(S \cap V) = Q_0 \cap f_0^{-1}(V)$;
\item[ii)] $T^2 \cdot \sigma'(S \cap (\R^3 - U)) = T^2 \cdot \sigma_0(S \cap (\R^3 - U))$
\end{itemize}
\end{lem}
\begin{proof}
We work over the quotient with respect to the $T^2$ action. We have $X_0 / T^2$ is isomorphic to $(\C)^* \times \R^2$ with projection $\pi$ given in (\ref{+n_t2_proj}). Then $f = \bar f \circ \pi$, where $\bar f$ is defined in (\ref{P_fibr}). Notice that $\pi (\Crit f_0) = \{ -1 \} \times \Delta$. Let $\sigma: S \rightarrow X_0$ be the map found in Lemma \ref{collar_fib} (formula (\ref{collar_sec})), then the quotient of $\sigma$ is $\bar \sigma = \pi \circ \sigma$ given by
\[
   \bar \sigma (0, x_2, x_3) = (-1, x_2, x_3).   
\]
On the other hand the quotient of $\sigma_0$ will be a section $\bar \sigma_0 = \pi \circ \sigma_0$ of $\bar f$, which restricted to $S$ has the form 
 \[ \bar \sigma_0|_{S}(0, x_2, x_3) = (e^{2\pi i \theta(x_2, x_3)}, x_2, x_3),\]
for some smooth real function $\theta$.
Now let $V \subset U$ be an open neighborhood of $\Delta$ such that there exists a smooth function $\rho: S \rightarrow [0,1]$ satisfying $\rho |_{S \cap V} =1$ and $\rho |_{S \cap (\R^3 - U)} = 0$. Define the following function on $S$:
\[ \tilde \theta = - \rho + (1- \rho) \theta. \]
Then we interpolate $\bar \sigma$ and $\bar \sigma_0$ by defining
\[ \bar \sigma'(0, x_2, x_3) = (e^{2\pi i \tilde \theta(x_2, x_3)}, x_2, x_3). \]
So that $\bar \sigma'|_{S \cap V} = \bar \sigma|_{S \cap V}$ and $\bar \sigma'|_{S \cap (\R^3-U)} = \bar \sigma_0|_{S \cap (\R^3-U)}$. Any lift $\sigma': S \rightarrow X_0$ of $\bar \sigma'$ will satisfy the lemma. 
\end{proof}

We also have a similar result for the other local collar $P_0$ defined in (\ref{local_collar2}). Since we are only interested in $P_0$ near the singular point of $X_0$, it will be convenient to redefine $P_0$ with an additional condition as follows:
\begin{equation} \label{local_collar3}
 P_0 = \{ (z,\bar z, r, s) \in \C \times \C \times \R_{\geq 0} \times \R_{\geq 0} \ | \  rs=|z|^2, \  rs < 1/2 \}.
\end{equation}
We consider the $S^1$ action on $X_0$ given by 
\begin{equation} \label{collar_s1}
 \xi \cdot (z_1, z_2, z_3, z_4) = (\xi z_1, \xi^{-1}z_2, z_3, z_4).
\end{equation}
 Observe that the fixed points of this action are the components $L_3$ and $L_4$ of $\Crit f_0$ mapped over $\Delta_3$ and $\Delta_4$.

\begin{lem} \label{collar2_fib}
Let $f_{0}: X_0 \rightarrow \R^3$ be as in (\ref{eq:exp_res_1}) with $\delta=0$ and let $P_0 \subset X_0$ be the 
local collar defined in (\ref{local_collar3}). Let 
\[ S =  \{ x_1 \leq 0, x_2 = 0 \} \subseteq \R^3. \]
Then $\partial S \subset \Delta$,  $f(P_0) =S$ and there exists a map $\check \sigma: S \rightarrow X_0$ such that
\begin{itemize}
\item[i)] $f_0 \circ \check \sigma = \id_{S}$ and $\check \sigma(\partial S) \subset \Crit f_0$;
\item[ii)] $P_0 = S^1 \cdot \check \sigma(S)$. 
\end{itemize}
where $S^1$ acts by (\ref{collar_s1}). 
\end{lem}

\begin{proof}  Clearly $\partial S = \Delta_3 \cup \Delta_4 \cup \{ v_1 \}$. Moreover
\[f_{0}|_{P_0 }(z,\bar z, r,s) = (\log|2rs -1|, 0, r^2-s^2). \]
A calculation shows that $f_0(P_0) = S$ and that the fibres of $f_0|_{P_0}$  are orbits of the $S^1$-action. Notice also that $P_0 \cap \Crit f_0$ is mapped one to one to $\partial S$. This implies the existence of $\check \sigma$ satisfying the lemma. 
\end{proof}

Also in this case we will need the following:
\begin{lem} \label{collar_pert2}
Let $f_0: X_0 \rightarrow \R^3$ be as in (\ref{eq:exp_res_1}) with $\delta=0$, $S$ as in Lemma \ref{collar2_fib} and $\sigma_0: \R^3 \rightarrow X_0$ a section. For any open neighborhood $U$ of $\Delta$, there exists a smaller neighbourhood $V \subseteq U$ of $\Delta$ and a map $\check \sigma': S \rightarrow X_0$ such that $f_0 \circ \check \sigma' = \id_S$ and 
\begin{itemize}
\item[i)] $S^1\cdot \check \sigma'(S \cap V) = P_0 \cap f_0^{-1}(V)$;
\item[ii)] $S^1 \cdot \check \sigma'(S \cap (\R^3 - U)) = S^1 \cdot \sigma_0(S \cap (\R^3 - U)).$
\end{itemize}
\end{lem}

\begin{proof}
The idea is similar to the proof of Lemma \ref{collar_pert1}. We consider the quotient $X_0 / S^1$ with projection $\pi: X_0 \rightarrow  X_0 / S^1$ and quotient fibration $\bar f: X_0/S^1 \rightarrow \R^3$ such that $f_0 = \bar f \circ \pi$. Let $S_0 = S - \partial S$ and consider the restriction of $\bar f$ to $S_0$, i.e. the map $\bar f|_{S_0}: \bar f^{-1}(S_0) \rightarrow S_0$. This defines a trivial  $T^2$-bundle over $S_0$, i.e. $\bar f^{-1}(S_0) \cong S_0 \times T^2$ and $\bar f|_{S_0}$ is the projection. We have the quotient maps $\check{\bar \sigma} = \pi \circ \check \sigma$ and $\bar \sigma_0 = \pi \circ \sigma_0$. For all $x \in S_0$, we may write 
\[ \check{\bar \sigma}(x) = (x, [ \theta(x)]) \ \ \text{and} \ \  \bar \sigma_0(x) = (x, [ \theta_0(x)]) \] 
where $\theta$ and $\theta_0$ are smooth maps with values in $\R^2$ and $[ \cdot ]$ denotes the class in $T^2 = \R^2 / \Z^2$. We may find an open neighborhood $V \subseteq U$ of $\Delta$ and a smooth real valued bump function $\rho: S \rightarrow \R$ as in the proof of Lemma \ref{collar_pert1} and define
\[ \tilde \theta = \rho \theta + (1- \rho) \theta_0. \]
Then we set $\check {\bar \sigma}'(x) = (x, [\tilde \theta (x)])$. Any lift of $\check{\sigma}': S \rightarrow X_0$ of 
$\check {\bar \sigma}'$ will satisfy the lemma.\end{proof}

\subsection{Negative fibrations}  Consider the $S^1$-action on $X$ defined as follows:
\begin{equation}\label{eq:s1_action}
\xi\cdot (z_1, z_2, z_3, z_4, t ) =( \xi z_1, \ \xi^{-1}z_2, \ z_3, \  z_4,  \ \xi t)
\end{equation}
for $\xi\in S^1$. Ignoring $t$, the same formula also gives an $S^1$ action on $X_0$ and $Y_{\epsilon}$.  

\begin{ex}[Resolution]\label{ex:fibr_res_2} We give another fibration on $X$ and $X_0$. The moment map of the action (\ref{eq:s1_action}) on $X$ is
\begin{equation*}
\mu=|z_1|^2-|z_2|^2+\frac{\delta}{1+|t|^2}.
\end{equation*}
The map $f_\delta: X \rightarrow \R^3$ defined by 
\begin{equation}\label{eq:exp_res_2}
f_\delta(z,t)=( \mu (z,t),\ \log |z_3-1|,\ \log |z_4-1|)
\end{equation}
gives a 3-torus fibration. Notice that $X/S^1$ can be identified with $\R \times (\C^*)^2$ via the map 
\begin{equation} \label{s1proj_neg}
\pi(z, t) = (\mu, z_3-1, z_4-1)
\end{equation}
and $f_{\delta}$ is the composition of $\pi$ with the regular $T^2$ fibration $\bar f: \R \times (\C^*)^2 \rightarrow \R^3$ given by 
\begin{equation} \label{qfib_neg}
    \bar f(u_1, u_2, x) = (x, \log |u_1|, \log |u_2|). 
\end{equation}
Therefore $\Crit (f_\delta)$ is the set of fixed points of the $S^1$ action. It consists of the following two components:
\begin{equation}\label{eq:crit_loc_2}
D_i=\{z_j=0, j\neq i\}, \quad i=3,\ 4 
\end{equation}
Notice that points of $D_3$ must also satisfy $t=0$ and points of $D_4$ must satisfy $t= \infty$. 
Thus $D_3$ and $D_4$ are mapped by $f_\delta$ onto the discriminant locus $\Delta$, which is the union of two lines as in Figure \ref{type_II} (a):
\begin{equation}\label{eq:discr_res_2}
\begin{split}
& \Delta_3 := f_{\delta}(D_3) = \{ x_1=\delta,\ x_3=0\}\cong\R\\
& \Delta_4: = f_{\delta}(D_4) = \{x_1=0, \ x_2=0\}\cong\R.
\end{split}
\end{equation}
\begin{figure}[!ht] 
\begin{center}
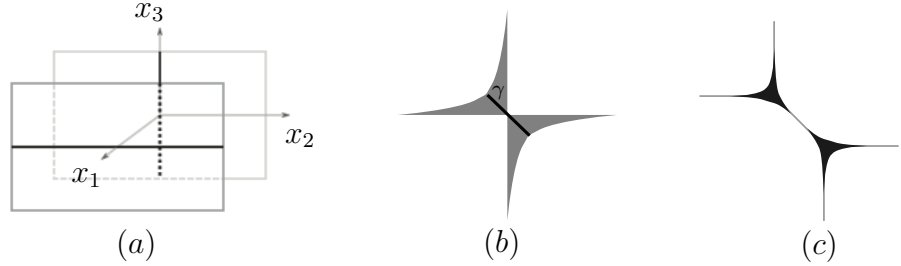
\caption{Negative fibrations: resolution (a), smoothing (b) and (c)}\label{type_II}
\end{center}
\end{figure}
Notice that when $\delta = 0$, $f_{0}$ is well defined on the local conifold $X_0$. In this case $D_3$ and $D_4$ intersect at the origin. Moreover the lines (\ref{eq:discr_res_2}), come together to form a 4-valent vertex. 
A direct calculation shows that the exceptional $\PP^1$ is mapped to the segment $\{ x_2=x_3=0, 0\leq x_1\leq \delta \}$ joining the components of (\ref{eq:discr_res_2}).  The fibres over $p$ in the interior of this segment intersect the exceptional curve along an $S^1$, which collapses to a point as $p$ approaches either component of the discriminant. Observe that this fibration is not Lagrangian with respect to the symplectic form (\ref{eq:sym_str}). We will discuss how to obtain Lagrangian fibrations of this model in \S \ref{tnodes_lagfib}. 
\end{ex}

\begin{lem} For all $\delta \geq 0$, $f_{\delta}$ has a smooth section $\sigma_0$. \end{lem}

\begin{proof}
Let $\Sigma_j =  \pi(D_j)$, then 
\[ \Sigma_3 = \{ (\delta, u, -1) \in \R \times (\C^*)^2 \}, \ \ \Sigma_4 = \{ (0,-1, u) \in  \R \times (\C^*)^2 \}. \]
Let $\bar \sigma_0: \R^3 \rightarrow \R \times (\C^*)^2$ be a section of $\bar f$ which avoids $\Sigma_3 \cup \Sigma_4$, e.g. $\sigma_0 (x_1, x_2, x_3) = (x_1, e^{x_2}, e^{x_3})$. Then a lift $\sigma_0: \R^3 \rightarrow X$ of $\bar \sigma_0$ exists since  the $S^1$-action on $X$ restricts to a trivial $S^1$ bundle over $\bar \sigma_0(\R^3)$. 
\end{proof}

Notice that $f_{\delta}^{-1}(\R^3 - \Delta)$ has the structure of a $T^3$-fibre bundle $\mathcal E / \Lambda$ which is compatible with the structure of $T^2$-fibre bundle $\bar f: \R \times (\C^*)^2 \rightarrow \R^3$ on the quotient with respect to the $S^1$ action, i.e. $\pi$ restricted to a fibre of $f_{\delta}$ is a linear map onto the fibre of $\bar f$. The section $\sigma_0$ corresponds to the zero section.

\begin{prop} \label{node_fibr}
For all $\delta \geq 0$, all singular fibres  of $f_{\delta}$ are of generic type, except the one over the vertex of $\Delta$ when $\delta =0$. Given a generic point $b \in \R^3$, there exists a basis $e_1, e_2, e_3$ of $H_1(X_b, \Z)$ with respect to which the monodromy $T_j$ around 
$\Delta_j$ is given by the following matrices
\begin{equation}\label{node_mon}
T_3=\left( \begin{array}{ccc}
                 1 & 0 & 1 \\
                 0 & 1  & 0 \\
                 0 & 0  & 1 
              \end{array} \right),
\
T_4=\left( \begin{array}{ccc}
                 1 & 1 & 0 \\
                 0 & 1  & 0 \\
                 0 & 0  & 1 
              \end{array} \right). 
\end{equation}
In particular, when $\delta = 0$, two opposite legs of $\Delta$ emanating from its vertex have the same monodromy.
\end{prop}

\begin{proof}
Let $\Sigma_3$ and $\Sigma_4$ be as in the previous lemma and $\Sigma = \Sigma_3 \cup \Sigma_4$. Then $\Sigma$ is a smooth surface, except when $\delta =0$, in which case $\Sigma_3$ and $\Sigma_4$ intersect at the singular point of $X_0$. 
When $\delta>0$, $f_{\delta}$ satisfies properties $(a)-(d)$ given at the beginning of \S \ref{local_fib}, with $Y = \R \times (\C^*)^2$. In the case $\delta=0$, this remains true if we remove the singular point of $X_0$. One can see that the $S^1$ action satisfies property $(c)$ as follows. On the open set $\{ t \neq \infty \} \subset X$, one has $z_1 = tz_3$ and $z_4 = tz_2$. Therefore, on this open set, $(z_2, z_3, t)$ define coordinates with respect to which the $S^1$ action (\ref{eq:s1_action}) can be written as $\xi \cdot (z_2, z_3, t) = ( \xi^{-1} z_2, z_3, \xi t)$. Hence it satisfies property $(c)$. Similarly on the open set $\{ t \neq 0 \}$.

Observe that a $T^2$ fibre of $\bar f$ over a point $b \in \Delta_4$, intersects $\Sigma_4$ along the circle $\{ (0, -1, u) \, | \, |u| = \text{const} \}$. Denote by $e_2$ the class of this circle inside $H_1 (\bar f^{-1}(b), \Z)$. Similarly $\bar f^{-1}(b)$, with $b \in \Delta_3$, intersects $\Sigma_3$ in the circle $\{ (\delta, u, -1) \, | \, |u| = \text{const} \}$. Denote by $e_3$ the class of this circle in $H_1 (\bar f^{-1}(b), \Z)$. Since $\bar f$ is a trivial $T^2$ fibration, we can identify $e_2, e_3$ with a basis of $H_1 (\bar f^{-1}(b), \Z)$ for any $b \in \R^3$. This shows that $f_{\delta}$, in a neighborhood of $\Delta_j$, has the structure described in Example \ref{ex. (2,2)}. Thus all fibres are of generic-singular type. This holds also when $\delta =0$, except for the fibre over the vertex of $\Delta$. 

Moreover, for a generic point $b \in \R^3$, we can choose as basis of $H_1(X_b,\Z )$, the cycles $e_1, e_2,e_3$, where $e_2$ and $e_3$ are the chosen cycles in $\bar f^{-1}(b)$ and $e_1$ is a fibre of the $S^1$-bundle. Thus, it follows from Example \ref{ex. (2,2)}, that with respect to this basis, the monodromies around $\Delta_3$ and $\Delta_4$ are as stated.
\end{proof}

\begin{ex}[Smoothing] \label{ex:fibr_smooth_2} We describe a fibration on $Y_{\epsilon}$. The moment map for the action (\ref{eq:s1_action}) on $Y_{\epsilon}$ is $\mu=|z_1|^2-|z_2|^2$.  The map $f:Y_\epsilon \rightarrow\R^3$ is 
\begin{equation}\label{eq:fibr_smooth_2}
f(z)=(|z_1|^2-|z_2|^2,\ \log |z_3-1|,\ \log |z_4-1|).
\end{equation}
It is a smooth torus fibration but it is not Lagrangian with respect to the standard symplectic form. 
The critical locus is 
\[
\Crit (f)=\{z_1=z_2=0, z_3z_4= - \epsilon\}
\]
and the discriminant $\mathcal A=f ( \Crit (f))$
has the shape of a 4-legged amoeba contained in the plane $\{x_1=0\}\subset\R^3$ as in Figure \ref{type_II} (b). Specifically, if $\Log(u,v) = (\log |u|, \log |v|)$ and 
\begin{equation}\label{eq:hypers}
V=\{(u,v)\in(\C^\ast)^2\mid (u+1)(v+1)+\epsilon=0\}, 
\end{equation}
then  $\mathcal A=\Log (V)$. A construction of a Lagrangian fibration on the smoothing will be explained in \S \ref{tnodes_lagfib}. 

A topologically equivalent description of this fibration, when $\epsilon >0$, can be constructed as follows. 
Let $\Delta = \Delta_0 \cup \ldots  \cup \Delta_4 \cup \{v_1 \} \cup \{v_2 \}$, where the $\Delta_j$'s and the $v_k$'s are as in (\ref{discr_resol}). Then, as we did in Example \ref{ex. (2,1)}, we can construct a surface $\Sigma$ in $Y = T^2 \times \R^3$ such that $\bar f(\Sigma) = \Delta$, in such a way that for a point $b \in \Delta_j$ we have that $\bar f^{-1}(b) \cap \Sigma$ is a circle in $T^2$ representing the class $e_2$ if $j=3$ or $4$, the class $e_3$ if $j=1$ or $2$ and the class $-e_2 -e_3$ if $j=0$. Then, one constructs $X$ with an $S^1$ action satisfying properties $(a)-(c)$ at the beginning of \S \ref{local_fib} (using Proposition 2.5 of \cite{TMS}). The fibration  $f: X \rightarrow \R^3$ is defined as in property $(d)$. If $e_1$ denotes the cycle of the $S^1$ action, with respect to the basis $e_1, e_2, e_3$ of $H_1(X_b, \Z)$, monodromy around the legs $\Delta_j$ is given by the following matrices
\[ T_1 = T_2 = \left( \begin{array}{ccc}
                 1 & 0 & 1 \\
                 0 & 1  & 0 \\
                 0 & 0  & 1 
              \end{array} \right), \ \ \  T_3 = T_4 = \left( \begin{array}{ccc}
                 1 & 1 & 0 \\
                 0 & 1  & 0 \\
                 0 & 0  & 1 
              \end{array} \right), \ \ \  T_0 = T_2^{-1}T_3^{-1}. \]
 In this case $X$ is homeomorphic to a dense open subset of the smoothing  $Y_{\epsilon}$. This is discussed in Section 4 of \cite{Gross_spLagEx}. Notice that in this case $\Delta$ has two vertices over which the fibration is of negative type. 
\end{ex}

\begin{rem} Observe that the torus fibrations on the resolution given in Example \ref{ex:res_alg_1} and the one on the smoothing given in Example \ref{ex:fibr_smooth_2} are topologically ``mirror dual'' in the sense of \cite{TMS}, since monodromies are dual to each other. The same is true for the torus fibrations on the smoothing given in Example \ref{ex:smooth_1} and on the resolution given in Example \ref{ex:fibr_res_2}. This was already observed by Ruan \cite{Ruan} and Gross in Section 4 of \cite{Gross_spLagEx}. In the next section we will show that this is true at the level of ``tropical manifolds'', where mirror duality is meant in the sense of discrete Legendre transform. 
\end{rem}

\subsection{Local collars and the negative fibration.}
We now prove some technical lemmas which describe the local collars of \S \ref{collars} in terms of the negative fibration on $X_0$. These will be used in the proof of Theorem \ref{good_rel_pos}. 

\begin{lem} \label{collar1_-fib} Let $f_{0}: X_0 \rightarrow \R^3$ be as in (\ref{eq:exp_res_2}) and let $Q_0 \subseteq X_0$ be the local collar defined in (\ref{local_collar}). Let 
\[ S = \{ x_1 \geq 0, x_3 = 0 \}. \]
Then $\partial S = \Delta_3$, where $\Delta_3$ is as in (\ref{eq:discr_res_2}) and $f_0(Q_0) = S$. Moreover, for every $b \in S - \partial S$, $f_0^{-1}(b) \cap Q_0$ is an affine $2$-dimensional subtorus of $f_0^{-1}(b)$, parallel to the monodromy invariant $T^2$ with respect to monodromy around $\Delta_3$. 
\end{lem}

\begin{proof}
Obviously $\partial S = \Delta_3$ and a simple calculation shows that $f_0(Q_0) = S$. Observe that $Q_0$ is invariant with respect to the $S^1$ action (\ref{eq:s1_action}). Let $\bar Q_0 = \pi(Q_0)$, i.e. the quotient of $Q_0$ by $S^1$, then
\[ \bar Q_0 = \{ (t, u, -1) \, | \, t \geq 0, \ u \in \C^* \} \]
Clearly, for all $b=(b_1, b_2, 0) \in S$, we have  
\[ \bar f^{-1}(b) \cap \bar Q_0 = \{ (b_1, e^{b_2 + i \theta}, -1) \, | \, \theta \in \R \} \]
which is an affine sub-circle of $\bar f^{-1}(b)$. As we have seen in the proof of Proposition \ref{node_fibr}, the homology class of this circle together with the homology class of the orbit of the $S^1$ action spans the monodromy invariant $T^2$ with respect to monodromy around $\Delta_3$. This proves the lemma. 
\end{proof}

We now treat the case of the local collar $P_0$ defined in (\ref{local_collar2}). Since we only need to understand $P_0$ in a neighborhood of the singular point of $X_0$  it is convenient to redefine $P_0$ as follows
\begin{equation} \label{local_collar4}
 P_0 = \{ (z,\bar z, r, s) \in \C \times \C \times \R_{\geq 0} \times \R_{\geq 0} \ | \  rs=|z|^2, \  r<1, \ s<1 \}.
\end{equation}

Then we have

\begin{lem} \label{collar2_-fib}
Let $f_{0}: X_0 \rightarrow \R^3$ be as in (\ref{eq:exp_res_2}) with $\delta=0$ and let $P_0 \subset X_0$ be the 
local collar defined in (\ref{local_collar4}). Let 
\[ S =\{ x_1=0, x_2 \leq 0, x_3\leq 0 \}. \]
Then $\partial S \subset \Delta$,  $f(P_0) =S$ and there exists a map $\check \sigma: S \rightarrow X_0$ such that
\begin{itemize}
\item[i)] $f_0 \circ \check \sigma = \id_{S}$ and $\check \sigma(\partial S) \subset \Crit f_0$;
\item[ii)] $P_0 = S^1 \cdot \check \sigma(S)$. 
\end{itemize}
where $S^1$ acts by (\ref{eq:s1_action}). 
\end{lem}
\begin{proof}
Clearly $\partial S = (\Delta_3 \cap \{ x_2 \leq 0 \}) \cap (\Delta_4 \cap \{ x_3 \leq 0 \})$, where $\Delta_3$ and $\Delta_4$ are as in (\ref{eq:discr_res_2}). A calculation show that $f(P_0) =S$.  Moreover the fibres of $f_0|_{P_0}$ are the orbits of the $S^1$ action (\ref{eq:s1_action}). 
Observe also that $P_0 \cap \Crit f_0$ maps one to one onto $\partial S$. Define
\[ \check \sigma(0,x_2, x_3) = \left( \sqrt{(1-e^{x_2})(1-e^{x_3})},  \sqrt{(1-e^{x_2})(1-e^{x_3})}, 1-e^{x_2}, 1-e^{x_3} \right). \]
Then $\check \sigma$ satisfies the stated properties.
\end{proof}

We also have the following analogue of Lemmas \ref{collar_pert1} and \ref{collar_pert2}. 
\begin{lem} \label{collar_pert3}
Let $f_0: X_0 \rightarrow \R^3$ be as in (\ref{eq:exp_res_2}) with $\delta=0$, $S$ as in Lemma \ref{collar2_-fib} and $\sigma_0: \R^3 \rightarrow X_0$ a section. For any open neighborhood $U$ of $\Delta$, there exists a smaller neighborhood $V \subseteq U$ of $\Delta$ and a map $\check \sigma': S \rightarrow X_0$ such that $f_0 \circ \check \sigma' = \id_S$ and 
\begin{itemize}
\item[i)] $S^1\cdot \check \sigma'(S \cap V) = P_0 \cap f_0^{-1}(V)$;
\item[ii)] $S^1 \cdot \check \sigma'(S \cap (\R^3 - U)) = S^1 \cdot \sigma_0(S \cap (\R^3 - U)).$
\end{itemize}
\end{lem}
The proof is identical to the proofs of Lemmas \ref{collar_pert1} and \ref{collar_pert2}, so we leave the details to the reader. 

\section{Affine geometry and the local tropical conifold}\label{sec:symplectic_models}
Here we construct two \emph{tropical} models of the conifold $X_0$, which we call positive and negative nodes. We show that these two models are mirror to each other in the sense of the discrete Legendre transform.  Moreover they give rise to torus fibrations on the conifold $X_0$ which are topologically equivalent to the ones in Examples \ref{ex:res_alg_1} and \ref{ex:fibr_res_2} respectively. Then we introduce tropical resolutions and smoothings and show that discrete Legendre transform exchanges smoothings with resolutions. Finally we explain how to obtain Lagrangian 3-torus fibrations from these tropical manifolds.

\subsection{Tropical nodes.}

\begin{ex} (Negative node) \label{ex:neg_node}
Let 
\begin{equation} \label{prism}
T = \conv \{ (0,0,0), (1,0,0), (0,1,0), (0,0,1), (1,0,1), (0,1,1) \}   
\end{equation}
be a triangular prism in $\R^3$. Construct a tropical manifold $(\check{B}, \check{ \mathcal P}, \check{\phi})$ as follows. Take two copies of $T$ and choose in each copy a square face, then label the vertices of these square faces by $v_1, v_2, v_3$ and $v_4$ as in Figure~\ref{node_type2}. Now glue the two copies of $T$ along the chosen faces by the unique affine linear transformation which matches the vertices with the same label. This is $\check B$. We denote by $\check e$ the common face of the two copies of $T$. 
   In $\R^3$, let $e_1, e_2, e_3$ denote the standard basis and let $(x,y,z)$ be the coordinates. Consider the fan $\Sigma$ in $\R^3$ whose $3$-dimensional cones are two adjacent octants, i.e.  $\cone(e_1, e_2, e_3)$, and  $\cone( e_1, - e_2, e_3 )$. At each vertex $v_j$ choose a fan structure which maps the tangent wedge at $v_j$ of the first copy of $T$ to the first octant and the tangent wedge of the second copy of $T$ to the second octant (see Figure \ref{node_type2}). Then the discriminant locus $\Delta$ is the union of the segments joining the barycenter of $\check e$ to the barycenters of its edges. 
\begin{figure}[!ht] 
\begin{center}
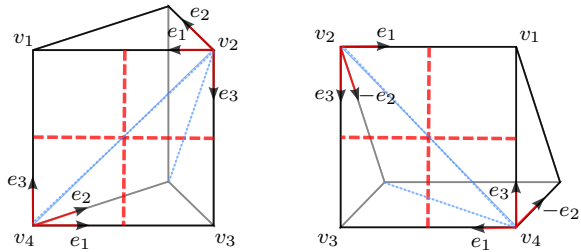
\caption{Match the vertices with the same labels. The arrows show the fan structure at the vertices. The (red) dashed lines denote $\Delta$.} \label{node_type2}
\end{center}
\end{figure}
We can easily compute the monodromy of $T\check{B}_0$. In fact, start at the vertex $v_4$ and choose the vectors $\{e_1, e_2, e_3 \}$ mentioned above as a basis for $T_{v_4}\check{B}_0$. Then consider a path which goes into the first prism, passes through $v_3$ and then comes back to $v_4$ going into the second prism.  Monodromy along this path is given by the matrix 
\begin{equation*} 
\left( \begin{array}{ccc}
                 1 & 1 & 0 \\
                 0 & 1  & 0 \\
                 0 & 0  & 1 
              \end{array} \right)
\end{equation*}
Similarly, consider a path which goes into the first prism, passes through $v_1$ and then comes back to $v_4$ going into the second prism. Monodromy along this path is given by the matrix 
\begin{equation*} 
\left( \begin{array}{ccc}
                 1 & 0 & 0 \\
                 0 & 1  & 0 \\
                 0 & 1  & 1 
              \end{array} \right).
\end{equation*}
Observe that $T_{v_4}\check{B}_0$ has a $2$-dimensional subspace, spanned by $e_1$ and $e_3$ which is invariant with respect to both monodromy transformations. 
Now define the strictly convex MPL-function $\check{\phi}$ so that on the fan $\Sigma$ it is given by
\[ \check \phi(x,y,z) = \begin{cases}
                   y \quad y \geq 0, \\
                   0 \quad y \leq 0.
                  \end{cases} \] 
\end{ex}

If we apply the discrete Legendre transform to $(\check B, \check{\mathcal P}, \check{\phi})$, we obtain the second tropical model for the conifold.

\begin{ex} \label{ex:pos_node} (Positive node) 
The triple $(B, \mathcal P, \phi)$, mirror to $(\check B, \check{\mathcal P}, \check{\phi})$ above, can be described as follows. The manifold $B$ can be identified with $\R^2 \times [0,1]$. If we denote by $Q_j \subseteq \R^2$, $j=1,2,3,4$ the four closed quadrants of $\R^2$, i.e. $\cone((1,0),(0,1))$, $\cone((-1,0),(0,1))$, $\cone((-1,0),(0,-1))$ and $\cone((1,0),(0,-1))$ respectively, then the $3$-dimensional polytopes of $\mathcal P$ are $L_j = Q_j \times [0,1]$. In fact $L_j$ is dual to the vertex $v_j$ in $\check {\mathcal P}$. Here we ignore the polytopes dual to vertices not contained in $\check e$.  The two vertices $p_0 = (0,0,0)$ and $p_1 = (0,0,1)$ are dual to the two triangular prisms. Consider the fan $\Sigma$ whose $3$-dimensional cones are:
\[   \begin{split}
       & \cone(e_1, -e_1 - e_2, e_3), \  \cone( e_1, e_2, e_3), \\
       & \cone(e_1, e_2, - e_3),   \      \cone(e_1, -e_1 - e_2, - e_3). 
      \end{split}     \]
The fan structure at $p_0$ maps tangent wedges of the polytopes $L_1, L_2, L_3$ and $L_4$ respectively to the first, second, third and fourth cone above. Similarly at $p_1$, the tangent wedges to $L_1, L_2, L_3$ and $L_4$ are mapped respectively to the first, fourth, third and second cone above, see Figure~\ref{node_type1}.
\begin{figure}[!ht] 
\begin{center}
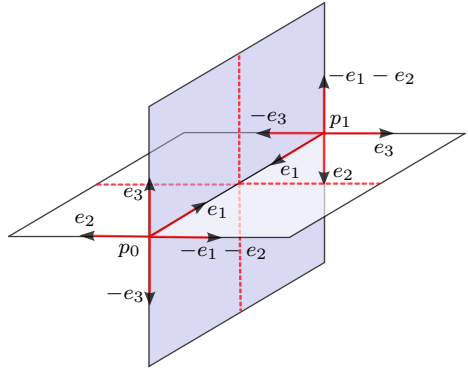
\caption{The tropical positive node. The arrows indicate the fan structure.} \label{node_type1}
\end{center}
\end{figure}

One easily checks that the discriminant locus is given by 
\[ \Delta = \{ (t,0, 1/2), \ t \in \R \} \cup \{ (0,t, 1/2), \ t \in \R \}, \]
with a $4$-valent vertex in $(0,0,1/2)$ . We can compute monodromy at $p_0$, where we choose the basis $\{e_1, e_2, e_3 \}$ of $T_{p_0}B_0$. Consider a path which goes from $p_0$ into $L_2$, reaches $p_1$ and then comes back to $p_0$ passing into $L_1$.  Monodromy along this path is given by the matrix 
\begin{equation*} 
\left( \begin{array}{ccc}
                 1 & 1 & 0 \\
                 0 & 1  & 0 \\
                 0 & 0  & 1 
              \end{array} \right)
\end{equation*}
Similarly, consider a path which goes into $L_2$, reaches $p_1$ and comes back to $p_0$ passing into $L_3$. Monodromy along this path is given by the matrix 
\begin{equation*} 
\left( \begin{array}{ccc}
                 1 & 0 & 1 \\
                 0 & 1  & 0 \\
                 0 & 0  & 1 
              \end{array} \right).
\end{equation*}
Observe that $T_{p_0}B_0$ has a $1$-dimensional subspace, spanned by $e_1$ which is invariant with respect to both monodromy transformations. 
\end{ex}

We can now give a more general notion of ``tropical conifold", which is a tropical manifold where $\Delta$ may have $4$-valent vertices (called nodes) modeled on the previous examples. 

\begin{defi} \label{trop_conif} 
We say that a $3$-dimensional tropical manifold $(B, \mathcal P, \phi)$ is a { \it tropical conifold} if $\Delta$ has vertices of valency $3$ or $4$. The $3$-valent vertices are either of positive or negative type (see Examples \ref{+v} and \ref{-v}).  Every $4$-valent vertex has a neighborhood which is integral affine isomorphic to a neighborhood of the vertex of $\Delta$ either in Example \ref{ex:neg_node} or Example \ref{ex:pos_node}. In the former case the vertex is called a negative node, in the latter a positive node.
\end{defi}

Of course, the discrete Legendre transform of a tropical conifold is also a tropical conifold. 

\subsection{Local tropical resolutions and smoothings}
We will describe two procedures which we claim should be the tropical analogue to resolving or smoothing a node. In fact each procedure uses discrete Legendre transform so that while it resolves the positive node (resp. negative) it smooths the mirror negative one (resp. positive).

Let us describe the local resolution of a positive node, which simultaneously smoothes the negative one.  A positive node is contained in the interior of an edge $e$ of $\mathcal P$ which belongs to four $3$-dimensional faces. In the mirror $(\check B, \check{\mathcal P}, \check \phi)$, the face dual to $e$ is the square face $\check e$. To ``smooth'' the negative node we proceed as follows.  Subdivide $\check e$ in two triangles by adding a diagonal.  Then we need to find a suitable polyhedral decomposition $\check{\mathcal P}'$ which is a refinement of $\check{\mathcal P}$ and  which induces the chosen subdivision of $\check e$. Next, we need an MPL-function $\check{\phi}'$ which is strictly convex with respect $\check{\mathcal P}'$. The given subdivision of $\check e$ also implies a change in the discriminant locus (see Figure \ref{transition}), where the $4$-valent vertex splits in two $3$-valent ones. This is a smoothing of the negative node. The resolution of the positive node is the discrete Legendre transform of the smoothing of the node. The process is illustrated in Figure~\ref{transition}.
\begin{figure}[!ht] 
\begin{center}
\includegraphics{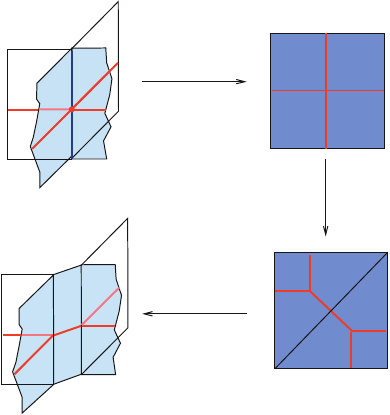}
\caption{The resolution of a positive node: the horizontal arrows are the discrete Legendre transform, the vertical one is the smoothing of the negative node.} \label{transition}
\end{center}
\end{figure}

Similarly we can describe the resolution of a negative node with the simultaneous smoothing of its mirror positive one (Figure \ref{transition2}). To smooth a positive node, first subdivide the edge $\ell$ where it lies in two new edges. Then we need a refinement $\mathcal P'$ of the decomposition $\mathcal P$ which induces the given subdivision of the edge. This decomposition has an extra vertex, hence we need to define a suitable fan structure at this vertex. Next, we find a new MPL function $\phi'$ which is strictly convex with respect to $\mathcal P'$. The choice of the fan structure at the new vertex must have the effect of separating the two lines of discriminant locus.  This is a smoothing of a positive node. The discrete Legendre transform gives us the resolution of the negative node. Note that the two lines of the discriminant locus in the negative node are separated by adding in between a new $3$-dimensional polytope, mirror to the new vertex. 

In the following paragraphs we apply this idea in detail. 
\begin{figure}[!ht] 
\begin{center}
\includegraphics{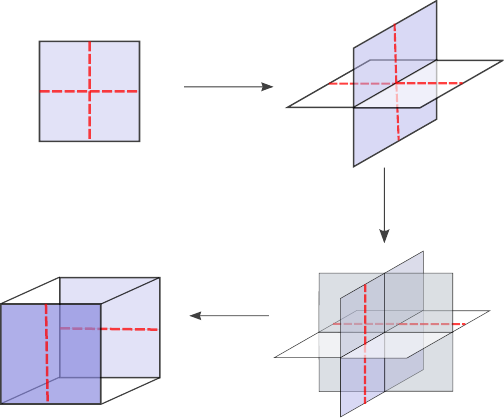}
\caption{The resolution of a negative node.} \label{transition2}
\end{center}
\end{figure}
\subsection{Resolving a positive node and smoothing its mirror.} \label{trop_res_+ve}
In Example \ref{ex:neg_node}, let us subdivide $\check e$ by taking the diagonal from $v_2$ to $v_4$ (see Figure \ref{node_type2}). Assume that $v_2$ and $v_4$ correspond respectively to the vertices $(0, 1, 1)$ and $(1,0,0)$ in the first copy of $T$ and to $(1,0,0)$ and $(0, 1, 1)$ in the second copy. Now subdivide $T$ in the two polytopes $\conv \{ (0,0,0), (1,0,0), (0,1,0), (0,1,1) \}$ and $\conv \{(0,0,0), (0,0,1), (1,0,1), (1,0,0), (0,1,1) \}$, and consider this subdivision for each copy of $T$ in $\check {\mathcal P}$. This gives the new decomposition $\check{\mathcal P}'$. Let us now find the new function $\check \phi$. The decomposition induces also a decomposition of the fans at the vertices. For instance, at $v_2$ the new fan is $\Sigma'_2$, whose $3$-dimensional cones are:
\[
 \cone( e_1, e_2, e_2 + e_3, e_1 + e_3), \cone(e_2+ e_3, e_1 + e_3, e_3),\]
\[ \cone ( e_1, -e_2, e_1 + e_3), \cone( e_3, - e_2, e_1 + e_3). \]
The first two subdivide the first octant and the other two subdivide the second one.
Similarly at $v_4$, the fan is $\Sigma'_4$ whose $3$-dimensional cones are:
\[ \cone(e_1, e_2, e_1 + e_3), \cone(e_1 + e_3, e_2, e_3), \]
\[ \cone(e_1, e_1-e_2, e_1 + e_3), \cone(e_3, - e_2, e_1 - e_2,  e_1 + e_3). \]
Now define an MPL-function $\tilde \phi$ as follows. It is the zero function on the fan at $v_1$. At $v_2$ it is the unique piecewise linear function which is $1$ on  $e_3$ and zero on all other generators of one dimensional cones. Similarly at $v_4$ define $\tilde \phi$ to be $1$ on $e_1$ and zero on all other generators. At $v_3$, $\tilde \phi$ is $-1$ on $e_2$ and zero on all other generators. Now define the new function $\check \phi'$ by
\[ \check \phi' = 2 \check \phi + \tilde \phi. \]
We can verify that $\check{\phi}'$ is well defined and strictly convex. 
The new discriminant locus $\Delta$ has two negative vertices. Applying the discrete Legendre transform to $(\check B, \check{ \mathcal P}', \check {\phi}')$ gives a new tropical manifold $(B', \mathcal P', \phi')$ which we define to be the resolution of a positive node (Figure~\ref{resolution_I}). 
\begin{figure}[!ht] 
\begin{center}
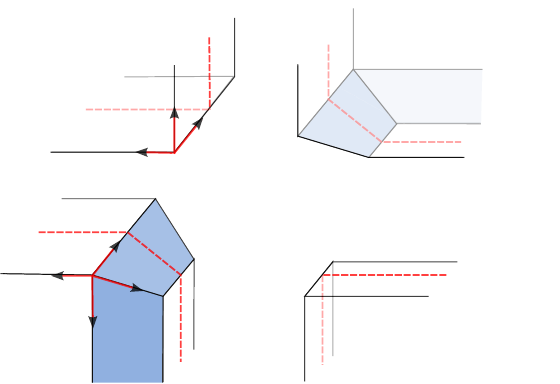
\caption{} \label{resolution_I}
\end{center}
\end{figure}
There are $4$ polytopes in $\mathcal P'$ dual to the vertices $v_1, \ldots, v_4$ of $\check {\mathcal P}'$, which we denote respectively by $\sigma_1, \ldots, \sigma_4$. We also denote by $\ell_{jk}$ the $2$-dimensional face in $\mathcal P'$ dual to the edge from $v_j$ to $v_k$, if such an edge exists. The polytopes $\sigma_2$ and $\sigma_4$ are integral affine isomorphic to the subset in $\R^3$ given by the following inequalities
\[ \begin{cases}
     -2 \leq y \leq 0, \\
      x \geq -1,  \\
      z \geq 0,  \\
      x+z \geq 0, \\
      x-y \geq 0. 
   \end{cases} \]
They intersect along the face $\ell_{24}$. We label the vertices of $\ell_{24}$ by $q_1, q_2, q_3$ and $q_4$ as in Figure~\ref{resolution_I}. The edge from $q_1$ to $q_4$ is the intersection between $\sigma_1$, $\sigma_2$ and $\sigma_4$. The polytopes $\sigma_1$ and $\sigma_3$ are as in the picture. Now consider the fan $\Sigma_1$ in $\R^3$ whose $3$-dimensional cones are:
\[ \cone(e_1, e_2, e_3), \cone(e_1, e_3, -e_2-e_3), \cone(e_1, e_2, -e_1-e_3, -e_2 - e_3). \]
Notice that $\Sigma_1$ is part of the normal fan of the two polytopes in $\check{ \mathcal P}'$ containing $v_1, v_2$ and $v_4$. Thus the fan structure at $q_1$ maps the tangent wedge to $\sigma_1$, $\sigma_2$ and $\sigma_4$ to the first, second and third cone respectively. Here $e_1$ and $e_2$ span a $2$-dimensional cone corresponding to $\ell_{14}$, $e_1$ and $e_3$ span a cone corresponding to $\ell_{12}$ and $e_1$ and $-e_2 - e_3$ span a cone corresponding to $\ell_{24}$. Similarly, at the vertex $q_4$ the fan structure maps the tangent wedge to $\sigma_1$, $\sigma_4$ and $\sigma_2$ to the first, second and third cones of $\Sigma_1$ respectively. The tangent wedges to $\ell_{14}$, $\ell_{12}$ and $\ell_{24}$ correspond respectively to $\cone( e_1, e_3 )$, $\cone(e_1, e_2 )$ and $\cone( e_1, -e_2 - e_3 )$. 
Let $\Sigma_2$ be the fan whose $3$-dimensional cones are:
\[ \cone(e_1, e_2, e_3), \cone( e_1, e_3, -e_1-e_2-e_3), \cone(e_1, e_2, -e_1-e_2-e_3). \]
Notice that $\Sigma_2$ is part of the normal fan of the two polytopes in $\check{ \mathcal P}'$ containing $v_3, v_2$ and $v_4$. Hence at the vertices $q_2$ and $q_3$ the fan structure maps the tangent wedges to $\sigma_3$, $\sigma_4$ and $\sigma_2$ respectively to the first, second and third cone of $\Sigma_2$. The tangent wedges to $\ell_{23}$, $\ell_{34}$ and $\ell_{24}$ correspond respectively to $\cone( e_1, e_2 )$, $\cone(e_1, e_3)$ and $\cone( e_1, -e_1 - e_2 - e_3)$. 

It can be verified that $\Delta$ has two $3$-valent positive vertices (see Figure~\ref{resolution_I}), one on the barycenter of the edge from $q_1$ to $q_4$ and the other on the barycenter of the edge from $q_2$ to $q_3$.

\subsection{Resolving a negative node and smoothing its mirror.} \label{trop_res_-ve}
To smooth the positive node in Example \ref{ex:pos_node}, consider the following refinement $\mathcal P'$ of $\mathcal P$. First rescale every polytope $L_j$ by a factor of two, so that $L_j = Q_j \times [0,2]$. Now subdivide each $L_j$ in the polytopes $L_j^- = Q_j \times [0,1]$ and $L_j^+ = Q_j \times [1,2]$. We have thus added a new vertex  which we denote by $q$. We now define the fan structure at $q$. Let $\Sigma_q$ be the fan in $\R^3$ whose maximal cones are the eight octants, namely $\cone(\pm e_1, \pm e_2, \pm e_3)$. The fan structure at $q$ then identifies the eight tangent wedges of $L_{j}^{\pm}$, $j = 1, \ldots, 4$ with these octants in the obvious way. The fan structures at $p_0$ and $p_1$ is unchanged. The new discriminant locus consists of two disjoint lines of generic-singularities: $\Delta^- = \{ (0,t,1/2),  \ t \in \R \}$ and $\Delta^{+} =  \{ (t,0,3/2),  \ t \in \R \}$. Assume that in the fan $\Sigma_q$, the vector $e_1$ is tangent to the edge from $q$ to $p_1$. On $\Sigma_q$ consider the piecewise linear function $\tilde \phi$ which takes the value $1$ on $e_1$ and the value $0$ on all other generators of $1$-dimensional cones of $\Sigma_q$. We have that 
\[ \phi' = \phi + \tilde \phi \]
is well defined and strictly convex. This structure defines the tropical smoothing of the positive node. A resolution of the negative node is given by the discrete Legendre transform $(\check B, \check{\mathcal P}')$. This can be described as follows. 

\begin{figure}[!ht] 
\begin{center}
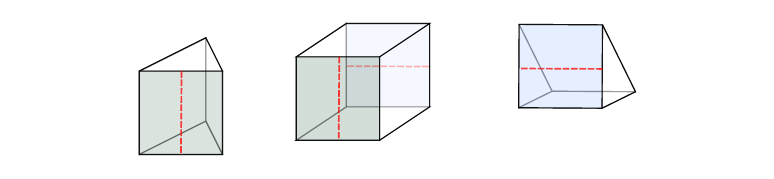
\caption{} \label{neg_resol}
\end{center}
\end{figure}

Take two copies of the triangular prism $T$, defined in (\ref{prism}) and a cube $Q=[0,1]^3 \subseteq \R^3$. Then glue the two copies of $T$ onto two opposite faces of $Q$ as in Figure \ref{neg_resol}, by matching vertices which have the same label. This is $B$ with its polyhedral decomposition $\check{\mathcal P}'$. The cube $Q$ is obviously mirror to the new vertex $q$ in $\mathcal P'$. The fan structure at a vertex $v_j^{\pm}$ is quite simple. In fact consider the fan $\Sigma$ whose $3$-dimensional cones are  $\cone(e_1, e_2, e_3)$ and $\cone(e_1, -e_2, e_3)$. Then the fan structure identifies the tangent wedges of $T$ and $Q$ at $v_j^{\pm}$ with these two octants. It is easy to see that $\Delta$ consists of two disjoint lines as in Figure \ref{neg_resol}. This is the tropical resolution of a negative node. 

There is also an alternative version of the smoothing of a positive node which proceeds as follows. Rescale all polytopes $L_j$ by a factor of $4$ and subdivide each $L_j$ in the polytopes $L_j^- = Q_j \times [0,1]$ and $L_j^+ = Q_j \times [1,4]$. Consider the fan whose three dimensional cones are
\[ \cone(e_1, e_1-e_2, e_3-e_1), \cone(e_1, e_2, e_3-e_1), \cone(e_1, e_2, -e_3), \]
\[ \cone(e_1, e_1-e_2, -e_3), \cone(-e_1, e_1-e_2, e_3-e_1), \cone(-e_1, e_2, e_3-e_1), \]
\[ \cone(-e_1, e_2, -e_3), \cone(-e_1, e_1-e_2, -e_3). \]
Now define the fan structure at the new vertex $q=(0,0,1)$ in such a way that the tangent wedges at $q$ to the polytopes $L_1^-$, $L_2^-$, $L_3^-$ and $L_4^-$ are mapped respectively to the first, second, third and fourth cone, while the tangent wedges to $L_1^+$, $L_2^+$, $L_3^+$ and $L_4^+$ are mapped respectively to the fifth, sixth, seventh and eighth cone. It can be checked that with such a choice of subdivision and fan structure at $q$, the discriminant consists of the two components $\Delta^- = \{(t,0, 1/2) \ t \in \R \}$ and $\Delta^+ = \{ (0, t, 5/2), \ t \in \R \}$. In particular, compared with the previous choices, the two lines of $\Delta$ have now moved in the opposite way. We leave it to the reader to determine a suitable strictly convex MPL function. 

The two possibilities for the smoothing should correspond, in the mirror, to the two choices of small resolution. 

\subsection{Tropical nodes and Lagrangian fibrations} \label{tnodes_lagfib}
In Examples \ref{ex:res_alg_1} and \ref{ex:smooth_1} we have described Lagrangian fibrations over (dense open subsets of) the local conifold $X_0$, of a small resolution $X$ and of a smoothing $Y_{\epsilon}$. It is known that Lagrangian fibrations induce an affine structure on the base of the fibration via action coordinates.

\begin{prop} \label{aff_str_+node} The affine structure induced on $\R^3 - \Delta$ by the Lagrangian fibrations on the conifold $X_0$, on its small resolution $X$ and on its smoothing $Y_{\epsilon}$ given respectively in Examples \ref{ex:res_alg_1} and \ref{ex:smooth_1}, is affine isomorphic (locally around the bounded edge of $\Delta$) respectively to the affine structures given in Examples \ref{ex:pos_node} for the tropical positive node, in \S \ref{trop_res_+ve} for its resolution and in \S \ref{trop_res_-ve} for its smoothing. 
\end{prop}

\begin{proof} We only do the case of the fibration on $X_0$, the other cases are similar.
First of all observe that the inverse transpose of the monodromy matrices in Examples \ref{ex:pos_node} are conjugate to the corresponding monodromy matrices of the fibration over $X_0$ found in Proposition \ref{mon_+node}. Therefore the fibration over $\R^3- \Delta$ gives a torus bundle which is topologically isomorphic to $X_{B_0} = T^*{B_0} / \Lambda^*$. 

In Proposition 4.11 of \cite{CB-M} we showed that the affine structure induced by a Lagrangian fibration over a positive vertex (Example \ref{ex. (1,2)}) is affine isomorphic to the affine structure of Example \ref{+v}.  The same ideas works here, so we only sketch the proof and refer to the above result for details. We work on the fibration over the conifold $f_0: X_0 \rightarrow \R^3$, the other cases are analogous. Observe that the symplectic form $\omega$ on $X_0$ is exact, so let $\eta$ be a primitive of $\omega$. We think the base $\R^3$ as $\R \times \R^2$ and we view $\Delta$ as inside the second factor (see (\ref{discr_resol}) with $\delta = 0$). Let 
\[ U = \R^3 - (\R_{\geq 0} \times \Delta). \]
Since $f_0$ is a topologically trivial $3$-torus bundle over $U$, $H_1(f_0^{-1}(U), \Z) \cong \Z^3$. Fix a basis $e_1, e_2, e_3$ of $H_1(f_0^{-1}(U), \Z)$ with respect to which monodromy around $\Delta$ is given by the matrices in Proposition \ref{mon_+node}. Action coordinates $A: U \rightarrow \R^3$ are defined by 
\[ A(b) = \left( - \int_{e_1} \eta|_{f_0^{-1}(b)},  \int_{e_2} \eta|_{f_0^{-1}(b)},  \int_{e_3} \eta|_{f_0^{-1}(b)} \right). \]
This is well defined since $\eta$ restricted to $f_0^{-1}(b)$ is closed. 

Now let $B$ be the affine manifold with singularities in Examples \ref{ex:pos_node} and denote by $\Delta_B$ its discriminant locus. The proof of the proposition consists in showing that $A$ defines a homeomorphism of pairs between $(\R^3, \Delta)$ and $(B, \Delta_B)$, or at least between suitable neighborhoods of the vertices of $\Delta$ and $\Delta_B$. Moreover $A$ gives an isomorphism of affine structures between $\R^3- \Delta$ and $B_0 = B - \Delta_B$. 
First of all, we need to show that $A$ extends continuously to $\R^3$. Let $A=(A_1, A_2, A_3)$. Since the last two components of $f_0$ are moment maps of the $T^2$ action and $e_2$ and $e_3$ have been chosen as the cycles generated by this action, it is easy to show that $A_2(b) = b_2$ and $A_3(b) = b_3$. So the last two components of $A$ extend to $\R^3$. Let us show that also $A_1$ extends. One can give another description of $A_1$ as follows. Fix $\bar b \in U$ and assume that the primitive $\eta$ has been chosen so that $A_1(\bar b) = 0$. 
Now choose some smooth path $\Gamma: [0,1] \rightarrow U$ between $\bar b$ and $b$ and construct a cylinder $S$ inside $f_0^{-1}(U)$ such that $f_0(S) = \Gamma([0,1])$ and $S \cap f_0^{-1}(\Gamma(t))$ is a circle representing the class $e_1$. Then we have
\[ A_1(b) = \int_{S} \omega. \]
Using the monodromy of the fibration one can show that $A_1$ extends to $\R^3- \Delta$ and then it is also easy to see that it extends to $\R^3$ (see op. cit. for details). To show that $A$ is a homeomorphism it is enough to show that for fixed values of $b_2$ and $b_3$ the map $t \mapsto A_1(t, b_2, b_3)$ is strictly monotone. Observe that since $(b_2, b_3)$ are values of the moment map $\mu= (\mu_1, \mu_2)$, we can form the reduced symplectic manifold $X_{(b_2, b_3)}$. One can see that $X_{(b_2, b_3)} \cong \C^*$ with some symplectic form $\omega_{\text{red}}$. One can compute $\omega_{\text{red}}$ explicitly and check that, although it may have poles at irregular points of the action, it will always be positive definite with respect to the standard orientation on $\C^*$. The Lagrangian fibration induced on $\C^*$ by $f_0$ is $\bar f: u \mapsto \log |u|$. Given $t_2 > t_1$, we have
\[ A(t_2, b_2, b_3) - A(t_1, b_2, b_3) = \int_{\bar f^{-1}[t_1, t_2]} \omega_{\text{red}} > 0. \]
This shows that $A$ is a homeomorphism onto its image. It remains to show that $A$ maps $\Delta$ to $\Delta_{B}$. In fact this may not be true, but as discussed at length in op. cit., one may slightly deform $\Delta_{B}$ inside the monodromy invariant planes, e.g. one may redefine $\Delta_B$ to be $A(\Delta)$. The fact that $A$ is an isomorphism of affine structures between $\R^3- \Delta$ and $B_0$ is a calculation which we leave to the reader. 
\end{proof}

  This shows that we can think of $X_0$, $X$ and $Y_{\epsilon}$ as a (partial) symplectic compactification $X_B$ of the symplectic manifold $X_{B_0} = T^*{B_0} / \Lambda^*$ constructed from the relevant tropical manifold $(B, \mathcal P, \Delta)$ (see diagram (\ref{compactify})). 

In the case of Examples \ref{ex:fibr_res_2} and \ref{ex:fibr_smooth_2}, the given fibrations on $X_0$, $X$ and $Y_{\epsilon}$ are not Lagrangian. In any case we have the following
\begin{prop} \label{-node_glue}
Let $B$ be a neighborhood of the vertex of the discriminant locus in Example \ref{ex:neg_node}, together with the induced affine structure with singularities and denote by $\Delta_B$ the discriminant locus. 
Let $f_0: X_0 \rightarrow \R^3$ be the $3$-torus fibration on the conifold given in Example \ref{ex:fibr_res_2}. Then there exists an homeomorphism of pairs $\iota: (B, \Delta_B) \rightarrow (\R^3, \Delta)$ and a commuting diagram
\begin{equation} \label{-n_diagr}
\begin{array}{ccc}
X_{B_0} & \hookrightarrow & X_0 \\ 
\downarrow & \  & \downarrow \\ 
B_0& \stackrel{\iota}{\hookrightarrow} & \R^3
\end{array}
\end{equation}
where $X_{B_0} = T^*{B_0} / \Lambda^*$.  The upper horizontal map is an isomorphism of $T^3$-torus bundles onto its image and the vertical maps are the torus fibrations. 
\end{prop}
\begin{proof}
Observe that the monodromy of the torus bundle $X_{B_0} \rightarrow B_0$ is given by inverse transpose of the monodromy matrices described in Example \ref{ex:neg_node}. These are easily seen to be conjugate to the monodromy matrices found in Proposition \ref{node_fibr}. Hence the proposition follows, since monodromy is the only topological invariant of these torus bundles. 
\end{proof}
Similarly one has that the smooth fibres of the torus fibration on the resolution $X$  (rep. smoothing $Y_{\epsilon}$) 
given in Example \ref{ex:fibr_res_2} (resp. Example \ref{ex:fibr_smooth_2}) form a torus bundle isomorphic 
to $X_{B_0} \rightarrow B_0$ obtained from the resolution of the negative node of \S \ref{trop_res_-ve} (resp. the smoothing of the negative node of \S \ref{trop_res_+ve}). Therefore, at a topological level, we can consider $X_0$, $X$ and $Y_{\epsilon}$ as the (partial) compactification of $X_{B_0}$. 

\begin{rem} \label{sympl_comp_tres} Notice that the resolution and smoothing of the tropical negative node have symplectic compactifications, which exist by the result of \cite{CB-M}, but we do not know if what we obtain is symplectomorphic to $X$ or $Y_{\epsilon}$, although we strongly believe this is true.
\end{rem}  

\begin{cor} \label{fibr_conifold}
Given a tropical conifold $(B, \mathcal P)$ as in Definition \ref{trop_conif}, there exists a topological conifold $X_B$ 
with a conifold singularity for every node of $B$ and a $3$-torus fibration $f: X_B \rightarrow B$ satisfying the commuting diagram (\ref{compactify}), where the upper horizontal arrow is an open (topological) embedding. Moreover $f$ has a section $\sigma_0: B \rightarrow X$.
\end{cor}
\begin{proof}
Using the homeomorphisms in the last two propositions we can glue local models applying the same arguments as in Gross' topological compactification (Theorem 2.1 of \cite{TMS}). To be more precise, first we glue local models over edges and over positive and negative trivalent vertices. This is done topologically as in Gross, op. cit., or symplectically by matching the affine structures induced by the Lagrangian fibrations (as in \cite{CB-M}).  Next, over positive and negative nodes, we glue the positive and negative fibrations over the conifold. One has to be careful that on overlaps, the gluing over the edges emanating from the node matches with the gluing of the local models over the edges.
Positive fibrations on the local conifold can be glued symplectically over positive nodes of $B$ by matching the affine structures as in Proposition \ref{aff_str_+node}.  On the overlaps, the fibration over the edges in the positive fibration can be matched (symplectically) to the local models over the edges following \S 4.4 of \cite{CB-M}. The negative fibration over a negative node can be glued topologically using Proposition \ref{-node_glue}. In this case, on the overlaps, the fibration over the egdes in the negative fibration can be matched to the local models over the edges using the same argument in the proof of Theorem 2.1 of \cite{TMS}). The section $\sigma_0$ is obtained by matching the zero section of $X_{B_0}$ with some fixed sections on the local models. 
\end{proof}

We also believe one can put a symplectic structure on $X_B$, extending the one on $X_{B_0}$, which makes $X_B$ into a symplectic conifold in the sense of \cite{STY}. In fact this is true if $B$ does not have negative nodes. We will call $X_B$ the conifold associated to $B$. Notice also that, as in the smooth case (see Theorem \ref{cekX}), we could consider $\check X_{B_0} = TB_0/ \Lambda$. Then $\check X_{B_0}$  can be compactified to form $\check X_{B}$. Over positive (resp. negative) nodes we glue the negative (resp. positive) fibration of the local conifold and similarly with positive and negative vertices. We still have that $\check X_{B}$ is homeomorphic to $X_{\check B}$. 

It is also reasonable to expect that the Gross-Siebert theorem (\cite{GrSi_re_aff_cx}), which associates to a tropical manifold a toric degeneration of Calabi-Yau manifolds, can be extended to tropical conifolds. Namely, one can expect that given a tropical conifold, we can reconstruct a toric degeneration whose fibres are Calabi-Yau manifolds with nodes.

\section{Good relations} \label{trop_relation}
We now start addressing the question of when a given set of nodes in a tropical conifold $B$ can be simultaneously resolved/smoothed.  We introduce the notion of a tropical $2$-cycle and we prove that if a set of nodes is contained in a tropical $2$-cycle, then both the vanishing cycles inside a smoothing of $X_{B}$ and the exceptional curves inside a small resolution of the mirror $X_{\check B}$ satisfy a good relation. Hence, morally, the obstructions to the symplectic resolution of $X_{B}$ and to the complex smoothing of the mirror $X_{\check B}$ vanish simultaneously. In this section we assume that the tropical conifold $B$ is oriented. In particular this implies that monodromy has values in $\Sl(\Z, n)$ and that the fibres of $X_{B}$ have a canonical orientation. 

\subsection{Tropical $2$-cycles} \label{tr_2_cy}
Recall that the tropical hyperplane $V^{n-1}$ in $\R^n$ is the set of points in $\R^n$ where the following piecewise linear function fails to be smooth
\[ f(x_1, \ldots, x_n) = \max \{x_1, \ldots, x_n,0 \}. \]
The point $(0, \ldots, 0)$ is called the vertex of the tropical hyperplane. We call $V^1$ and $V^2$ the tropical line and plane respectively.  

Our definition of a tropical $2$-cycle resembles the definition of a tropical surface such as in \cite{mikhalk_lectures}, but it has a more topological flavor. A tropical $2$-cycle will be a map from a certain space $S$ to $B$ plus some other data.

\begin{defi} \label{tropical_domain}
A \textit{tropical domain} $S$ is a compact Hausdorff topological space such that for every $p \in S$ there is a neighborhood $U$ of $p$ and a homeomorphism $\phi: (U,p) \rightarrow (W,q)$, where $(W,q)$ can be one of the following pairs (see Figure \ref{tdomain}):
\begin{itemize}
\item[a)] $q \in \R^2$ and $W$ is a neighborhood of $q$ ($p$ is called a {\it smooth point} of $S$);
\item[b)] $q$ is the vertex of a tropical plane $V^2$ and $W$ is a neighborhood of $q$ in $V^2$ ($p$ is called an {\it interior vertex} of $S$);
\item[c)] $q = (0,0,0) \in V^1 \times \R$ and $W$ is a neighborhood of $q$ in $V^1 \times \R$ ($p$ is called an {\it interior edge point});
\item[d)] $W$ is a neighborhood of $q=(0,0)$ in the closed half plane $\{ x \geq 0 \} \subset \R^2$ ($p$ is called {\it a smooth boundary point});
\item[e)] $q = (0,0,0) \in V^1 \times \R_{\geq 0}$ and $W$ is a neighborhood of $q$ ($p$ is called a {\it boundary vertex});
\end{itemize} 
We call points of type $a), b), c)$ interior points. Points of type $d)$ and $e)$ form the {\it boundary} of $S$, which we denote by $\partial S$. We also denote the set of smooth interior points by $S_{sm}$. We will also assume that all connected components of $S_{sm}$ are orientable and we fix an orientation. The union of all interior edge points forms a $1$-dimensional manifold, whose connected components we call the {\it interior edges} of $S$.
\end{defi}

\begin{figure}[!ht] 
\begin{center}
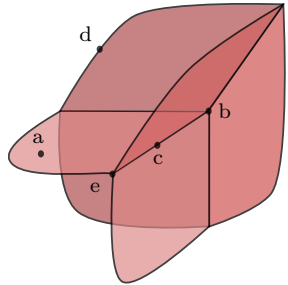
\caption{A tropical domain with smooth points (a), interior vertex (b), interior edge points (c), smooth boundary points (d) and boundary vertices (e).} \label{tdomain}
\end{center}
\end{figure}

We can now give the definition of a tropical $2$-cycle:
\begin{defi}  \label{tropical_cycle}
Let $(B, \mathcal{P}, \phi)$ be a tropical conifold. A \textit{tropical $2$-cycle} in $B$ is the data $(S, j, v)$ where 
\begin{itemize}
\item[a)] $S$ is a tropical domain and $j: (S, \partial S) \rightarrow (B, \Delta)$ is an embedding;
\item[b)] $j^{-1}(\Delta) = \partial S \cup \{q_{1}, \ldots, q_r \}$, where $q_1, \dots, q_r$ are smooth points and $j(q_k)$ is an edge point of $\Delta$ for all $k=1, \ldots,r$. We denote
\[ S_0 = S_{sm} - \{q_{1}, \ldots, q_r \}; \]
\item[c)] $v$ is a primitive, integral, parallel vector field defined along $j(S_0)$;
\item[d)] $j(p)$ is a negative vertex of $\Delta$ if and only if  $p$ is a boundary vertex of $S$;
\end{itemize} 
Notice that $v$ induces a rank $2$ subvector bundle $\mathcal F$ of $T^*B_0$ over $j(S_0)$, where: 
\[ \mathcal F_q = \ker v(q) = \{ \alpha \in T^*_qB_0 \ | \ \alpha(v(q)) = 0 \}. \]
for every $q \in j(S_0)$. The above properties imply that if $j(p)$ is a node or a positive vertex, then $p$ is a smooth boundary point. We may consider $\partial S$ as a graph where boundary vertices are trivalent vertices and $p$ is a bivalent vertex if and only if $j(p)$ is a node or a positive vertex. Then edges of $\partial S$ are mapped to a subset of the edges of $\Delta$.
We require in addition the following properties (see also Figure \ref{tcycle}): 
\begin{itemize}
\item[e)] if $j(p)$ is a negative node then the two edges of $\partial S$ emanating from $p$ are mapped to edges of $\Delta$ as in Figure \ref{tcycle}, picture (4);
\item[f)] if $j(p)$ is a positive node then the two edges of $\partial S$ emanating from $p$ are mapped to edges of $\Delta$ as in Figure \ref{tcycle}, picture (3);
\item[g)] let $p \in S$ be an interior edge point lying on the edge $e$. Choose an orientation of $e$. Given a small connected neighborhood $U$ of $p$,   $j(U \cap S_{sm})$ has $3$ connected components. The orientation of $e$ and the orientation of $B$ induce a cyclic ordering of these components. Denote by $v_1, v_2, v_3$ the vector field $v$ restricted to these components, indexed according to the ordering. Then the $v_k$'s span a rank two subspace of $T_pB$ and they satisfy the following {\it balancing condition}
\[
 \epsilon_1 v_1 + \epsilon_2 v_2 + \epsilon_3 v_3 = 0, 
\]
where $\epsilon_k = 1$ if the chosen orientation of $e$ coincides with the orientation induced from the orientation of the $k$-th component of $j(U \cap S_{sm})$, otherwise $\epsilon_k =-1$. 
\item[h)] if $e$ is an edge of $\partial S$ and $U$ a small neighborhood of $j(e)$, then for all points $q \in j(S_{sm}) \cap U$, $\mathcal F_q$ coincides with the monodromy invariant subspace  with respect to monodromy around $j(e)$;
\item[i)] if $q_k$ is one of the smooth points such that $j(q_k)$ is an edge point of $\Delta$, then monodromy of $\mathcal F$ around $j(q_k)$ is conjugate to the matrix 
\[ \left(
\begin{array}{cc}
  1 &  0  \\
  1 &   1
\end{array}
\right); \]
\item[j)] if $p$ is an interior vertex point, and $U$ is a small connected neighborhood of $p$,  then $j(U \cap S_{sm})$ has $6$ connected components. Let $v_k$, $k=1, \ldots,6$ denote the restrictions of $v$ to these components. Then the $v_k$'s span the whole of $T_pB$. 
\end{itemize}
\end{defi}

\begin{figure}[!ht] 
\begin{center}
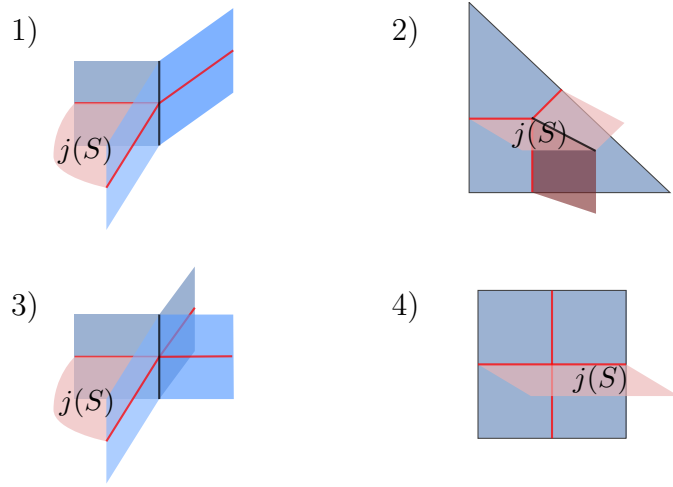
\caption{How $j(S)$ interacts with $\Delta$: 1) at a positive vertex; 2) at a negative vertex; 3) at a positive node; 4) at a negative node. } \label{tcycle}
\end{center}
\end{figure}

Notice that the embedding $j$ is topological, i.e. the components of $S_{sm}$ do not have to be mapped as affine subspaces of $B$.  In particular the boundary of $S$ can also be mapped to a ``curved"  $\Delta$ (see \S \ref{GSreconstruction} and the comments following Theorem \ref{section}). 

Given a tropical conifold $(B, \mathcal{P}, \phi)$, in Corollary \ref{fibr_conifold} we constructed the topological conifold $X_B$. Strictly speaking, we do not have general symplectic or complex reconstruction theorems for conifolds, nevertheless we have topological smoothings and resolutions of $X_B$ and therefore we can speak about vanishing cycles and exceptional curves and it makes sense to ask whether these satisfy good relations.  We will prove the following:

\begin{thm} \label{good_rel_pos} Let $(B, \mathcal{P}, \phi)$ be an oriented tropical conifold and $(\check B, \check{\mathcal{P}}, \check{\phi})$ its Legendre dual. Let $(S, j, v)$ be a tropical $2$-cycle and $p_1, \ldots, p_k$ be the points of $S$ which are mapped to nodes of $B$. Then the corresponding vanishing cycles $L_1, \ldots, L_k$ in a smoothing of $X_B$ and exceptional curves $C_1, \ldots, C_k$ in a small resolution of $X_{\check B}$ satisfy a good relation. 
\end{thm}

We want to construct a (topological) $4$-dimensional submanifold with boundary $\tilde S$ inside a smoothing of $X_B$ and a $3$-dimensional (topological) submanifold with boundary $\tilde S^*$ inside a resolution of $X_{\check B}$ such that $\partial \tilde S = \cup_{i=1}^{k} L_i$ and $\partial \tilde S^* = \cup_{i=1}^{k} C_i$. 
Equivalently (see \S \ref{collars}) we may think of  $\tilde S$ as a $4$-dimensional submanifold without boundary inside $X_B$, containing the nodes, such that in local coordinates near the nodes, $\tilde S$ coincides with the local collar $Q_0$ as in (\ref{local_collar}). 
Similarly $\tilde S^*$ may be thought as a subset of $X_{\check B}$, which, away from the nodes, is a submanifold and in local coordinates near the nodes, it coincides with the local collar $P_0$ as in (\ref{local_collar2}).
The idea is as follows. Consider the rank two bundle $\mathcal F$ defined by $v$. Since $v$ is integral, we have that $\mathcal F / (\mathcal F \cap \Lambda^*)$ defines a $2$-torus bundle over $j(S_{0})$.  For every $\alpha \in \mathcal F_q$, denote by $[\alpha]$ its class inside  $\mathcal F_q / (\mathcal F_q \cap \Lambda^*)$. Given a section $\sigma: B_0 \rightarrow X_{B_0}$, define 
\begin{equation} \label{lift_S:sm}
 \tilde S_0 = \{ (q, \sigma(q) + [\alpha]) \ | \ q \in j(S_{0}) \ \ \text{and} \ \ \alpha \in \mathcal F_q \}.
\end{equation} 
Essentially $\tilde S_0$ is  $\mathcal F / (\mathcal F \cap \Lambda^*)$ translated by $\sigma$. We want to prove that for a suitable choice of $\sigma$, $\tilde S_0$ can be compactified to form the submanifold $\tilde S$ such that $f(\tilde S) = j(S)$. 

The construction of $\tilde S^*$ is similar. First of all, as observed in Theorem \ref{cekX}, instead of working in $X_{\check B}$ we can equivalently work in $\check X_{B}$, which is formed using the tangent bundle of $B_0$. 
We denote by $\check f: \check X_{B} \rightarrow B$ the fibration. The vector field $v$ generates a $1$-dimensional subbundle of $TB_0$ along $j(S_{0})$, which we call $\mathcal V$. 
Then $\mathcal V/ \mathcal V \cap \Lambda$ is an $S^1$ bundle over $j(S_{0})$. We form $\tilde S^*_{0}$ by translating $\mathcal V/ \mathcal V \cap \Lambda$ by a suitable section $\check \sigma: B \rightarrow \check X_{B}$. The set $\tilde S^*$ is constructed as a suitable compactification of $\tilde S_0^*$. 

As we will see in the construction, $\sigma$ and $\check \sigma$ should not be genuine sections. 
They should be maps from $j(S)$ to $X_B$ (resp. $\check X_B$) which are sections restricted to $j(S_{0})$ but such that $\sigma(j(\partial S)) \subset \Crit f $ (resp. $\check{\sigma}(j(\partial S)) \subset \Crit \check{f} $). 
We will first define local models in some detail. We point out that in the local constructions of $\sigma$ and $\check \sigma$ there is a certain amount of  flexibility. For instance we can use partitions of unity  to glue together various pieces of $\tilde S$ and $\tilde S^*$ so that they match up when we construct $X_B$ as in Corollary \ref{fibr_conifold}. In particular it will be convenient to construct $\sigma$ and $\check \sigma$ so that they coincide with the zero section $\sigma_0$ away from $j(\partial S)$.

\subsection{Local models at smooth boundary points....}
We discuss local models for $\tilde S$ and $\tilde S^*$ near an edge $e \subseteq j(\partial S)$ of $\Delta$. In this case, the local model for both torus fibrations $f$ and $\check f$ is the generic singular fibration of  Example \ref{ex. (2,2)}.  Condition $(h)$ of Definition \ref{tropical_cycle} implies that $v$ and $\mathcal F$ are uniquely determined in a neighborhood of $e$, so that the lifts $\tilde S_0$ and $\tilde S^*_0$ only depend on the choice of sections.

Notice that $v$ is the unique (up to sign) primitive integral vector such that the monodromy $T$ around the edge $e$ satisfies $T(w)-w=m v$ for all $w \in T_bB_0$, where $m \in \Z$ depends on $w$. It then follows from the first part of Example \ref{ex. (2,2)}, that for every $b \in B_0$ near $e$, the circle $\mathcal V_b / \mathcal V_b \cap \Lambda$ must be an orbit of the $S^1$ action (the cycle called $e_1$). Suppose that $U$ is a small neighborhood of $e$ and let $\check \sigma: j(S) \cap U \rightarrow \check X_B$ be a section such that 
\begin{equation} \label{csigma_bound}
  \check \sigma(e) \subseteq \Crit \check f.
\end{equation}
Then we define 
\begin{equation} \label{tss_edge}
 \tilde S^* \cap \check f^{-1}(U) := S^{1} \cdot \check \sigma(j(S) \cap U). 
\end{equation}
By construction $\tilde S^*- (\tilde S^* \cap \Crit \check f)$ coincides (inside $\check f^{-1}(U)$) with $\tilde S^*_0$.  
We can assume that $U \cap j(S) \cong e \times [0,1)$. Then, since the $S^1$ orbits collapse to points on $\Crit \check f$,  it is easy to see that $\tilde S^* \cap \check f^{-1}(U) \cong e \times D$, where $D \subset \C$ is the open unit disc. Thus $\tilde S^* \cap \check f^{-1}(U)$ is a $3$-manifold. 

In a similar way we can construct $\tilde S \cap f^{-1}(U)$. The orbits of the $T^2$ action (\ref{torus_act}) give the monodromy invariant $T^2$ with respect to monodromy around $e$ (see (\ref{eq matrix g})). Thus $\mathcal F_b / \mathcal F_b \cap \Lambda^*$ must coincide with such an orbit. In particular, given a section $\sigma: j(S) \cap U \rightarrow X_B$ such that 
\begin{equation} \label{sigma_bound}
 \sigma(e) \subseteq \Crit f, 
\end{equation}
then we define 
\begin{equation} \label{ts_edge}
 \tilde S \cap f^{-1}(U) :=  T^2 \cdot \sigma(j(S) \cap U). 
\end{equation}
By construction $\tilde S - \Crit f$ coincides (inside  $f^{-1}(U)$) with $\tilde S_0$.  
It can be seen that $\tilde S \cap f^{-1}(U) \cong e \times S^1 \times D$, and thus it is a $4$-manifold. 

Notice that the flexibility in the construction of $\sigma$ and $\check \sigma$ stands in the fact that we only require them to satisfy (\ref{sigma_bound}) and  (\ref{csigma_bound}) respectively. We have the following

\begin{lem} 
\label{lift_edge}
Let $f: X \rightarrow U$ be a generic-singular fibration as constructed in \S \ref{ex. (2,2)}, where $U = D \times (0,1)$ and $\Delta = \{ 0 \} \times (0,1) \subset U$ is the discriminant locus. Let $S = [0,1) \times (0,1)$ and let $j: S \rightarrow U$ be an embedding such that $j(\partial S)= \Delta$. Given any section $\sigma_0: U \rightarrow X$ of $f$ and a neighborhood  $V \subset U$ of $\Delta$ , there exists a section $\sigma: j(S) \rightarrow X$ such that 
\begin{itemize}
\item[i)] $\sigma(j(\partial S)) \subseteq \Crit f $; 
\item[ii)]  $S^1 \cdot \sigma|_{j(S) \cap (U-V)} = S^1 \cdot \sigma_0|_{j(S) \cap (U-V)}$;
\item[iii)] $T^2 \cdot \sigma|_{j(S) \cap (U-V)} = T^2 \cdot \sigma_0|_{j(S) \cap (U-V)}$.
\end{itemize}
\end{lem}
\begin{proof} 
We can work on the quotient by the $S^1$ action $Y= X / S^1 = U \times T^2$. We let $\pi: X \rightarrow Y$ be the quotient map and $\bar f: Y \rightarrow U$ the projection.  We defined $f = \bar f \circ \pi$. Let $\bar \sigma_0 = \pi \circ \sigma$ be the quotient section of $\bar f$. By the construction of Example \ref{ex. (2,2)}, we have that $\Sigma = \pi (\Crit f)$ is a ``cylinder" such that for all $b \in \Delta$, $\bar f^{-1}(b) \cap \Sigma \subset T^2$ is a circle representing some fixed homology class $e_3$. Decompose $T^2$ as $S^1 \times S^1$ in such a way that $S^1 \times \{ 1 \}$ represents the class $e_3$. We can choose coordinates on $S^1 \times S^1$ in such a way that $\sigma_0 (b) = (1, 1) \in S^1 \times S^1$ for all $b \in U$. We can also describe $\Sigma$ as follows. There exists a function $\tau: \Delta \rightarrow S^1$ such that 
\begin{equation} \label{critf_eq}
 \Sigma = \{  \{ b \} \times S^1 \times \{ \tau (b) \} \in \Delta \times S^1 \times S^1 \, | \, b \in \Delta \}. 
\end{equation}
Notice that since the image of $\sigma_0$ must be disjoint from $\Crit f$, we must have that $\tau$ maps into $S^1 - \{ 1 \}$. In particular we can write 
\begin{equation} \label{tau_sig}
\tau(0,t) = e^{2 \pi i \theta(t)}
\end{equation}
where $b= (0,t) \in \{0\} \times (0,1) = \Delta$ and $\theta$ is a smooth function with values in $(0,1)$. 
Inside $S$, we can find an open neighborhood $W \subset j^{-1}(V)$ of $\partial S$ and a real valued, smooth function $\rho: S \rightarrow [0,1]$ such that $\rho|_{W} =1$ and $\rho|_{j^{-1}(U-V)} = 0$. Now let $(s,t) \in [0,1) \times (0,1)$ be the coordinates on $S$ and define
\[ \tilde \theta (s,t) = \rho(s,t) \theta(j(0,t)), \]
where $\theta$ is as in \ref{tau_sig}.
Now define
\[ \bar \sigma(j(s,t))=(j(s,t), 1, e^{2\pi i \tilde \theta(s,t)}) \in j(S) \times S^1 \times S^1. \]
By construction we have that $\bar \sigma|_{j(S) \cap (U-V)} = \bar \sigma_0|_{j(S) \cap (U-V)}$. Therefore any lift $\sigma: j(S) \rightarrow X$ of $\bar \sigma$ will satisfy (i) and (ii). Since $S^1$ orbits are contained in the $T^2$ orbits, also (iii) must hold. 
\end{proof}
Using $\sigma$ and $\check \sigma$ as in the above lemma ensures that if we define $\tilde S^*$ and $\tilde S$ as in (\ref{tss_edge}) and (\ref{ts_edge}), then, away from a neighborhood $V$ of the edge $e$, they are defined using a fixed genuine section $\sigma_0$. We will need the following result in the proof of Theorem \ref{good_rel_pos}.
\begin{lem} 
\label{match_sigma}
Let $f: X \rightarrow U$, $U$, $\Delta$, $S$ and $j: S \rightarrow U$ be as in Lemma \ref{lift_edge}. Let $\sigma: j(S) \rightarrow X$ be a section such that $\sigma(j(\partial S)) \subseteq \Crit f $. Given $U' = D \times (0, \epsilon) \subset U$ with $\epsilon >0$ and another section $\sigma': j(S) \cap U' \rightarrow X$ such that $\sigma'(j(\partial S) \cap U') \subseteq \Crit f $, then, for any $0 < \epsilon' <\epsilon$, there exists a section $\sigma'': j(S) \rightarrow X$, satisfying $\sigma''(j(\partial S)) \subseteq \Crit f$ such that 
\[
\begin{split}
 & \sigma''|_{j(S) \cap (D \times (0, \epsilon'))} = \sigma'|_{j(S) \cap (D \times (0, \epsilon'))}  \\
 & \sigma''|_{j(S) \cap (U-U')} = \sigma|_{j(S) \cap (U-U')}
 \end{split}
 \]
\end{lem}
\begin{proof}
We use the same set-up of Lemma \ref{lift_edge}, in particular we work on the $S^1$-quotient $Y = U \times T^2$ where $T^2 = S^1 \times S^1$ and $\pi(\Crit f) = \Sigma$ is described as in (\ref{critf_eq}). In the coordinates $(s,t) \in [0,1) \times (0,1) = S$, the quotient sections $\bar \sigma$ and $\bar \sigma'$ may be described as follows
\[
\begin{split}
& \bar \sigma(j(s,t))=(j(s,t), e^{2 \pi i  \theta_1(s,t)} , e^{2 \pi i  \theta_2(s,t)}) \\
& \bar \sigma'(j(s,t))=(j(s,t), e^{2 \pi i  \theta'_1(s,t)} , e^{2 \pi i  \theta'_2(s,t)})
\end{split}
\]
where $\theta_j$ and $\theta_j'$ are smooth functions. Now let $\rho: S \rightarrow [0,1]$ be a smooth function such that $\rho|_{S \cap j^{-1}(D \times (0, \epsilon'))} = 1$ and $\rho|_{S \cap j^{-1}(U-U')} = 0$ and define
\[ \tilde \theta_j = \rho \theta'_j +  (1-\rho) \theta_j \ \ \ j=1,2 \]
We still have $\tilde \theta_2(0,t) = \tau(j(0,t))$. Define
\[ \bar \sigma''(j(s,t))=(j(s,t), e^{2 \pi i  \tilde \theta_1(s,t)} , e^{2 \pi i  \tilde \theta_2(s,t)}). \]
By construction, any lift $\sigma'': j(S) \rightarrow X$ of $\bar \sigma''$ will satisfy the lemma. 
\end{proof}

\subsection{.... at nodes.} Let us now discuss the local models for $\tilde S$ and $\tilde S^*$ at a positive node of $B$. The local model for the fibration $f: X_B \rightarrow B$ over a positive node is given in equation (\ref{eq:exp_res_1}) (with $\delta = 0$) over the local conifold $X_0$. Lemma \ref{collar_fib} shows that the local collar $Q_0$ defined in  (\ref{local_collar}), is a compactification $\tilde S$ of $\tilde S_0$ for some suitable choice of section $\sigma$. In fact, given $S$ as in the lemma and $j$ the inclusion, then $j(S)$ is compatible with condition (f) of Definition \ref{tropical_cycle}.
Moreover the conclusions $(i)$ and $(ii)$ of the lemma show that near the edges of $\Delta$, $Q_0$ is constructed as $\tilde S$ in (\ref{ts_edge}), with $\sigma$ satisfying (\ref{sigma_bound}).
Observe also that the assumption that in a neighborhood of a positive node $S$ (or $j(S)$) is as in the lemma is not restrictive. In fact, given that condition (f) of Definition \ref{tropical_cycle} must hold, the symmetry of the local model ensures that we can assume that $\partial S$ coincides with the two edges as in the lemma, moreover we can always locally isotope $S$ so that it coincides precisely with the definition of $S$ given in the lemma.  Finally, Lemma \ref{collar_pert1} says that one can find $\sigma$, satisfying (\ref{sigma_bound}), such that $\tilde S$ coincides with $Q_0$ over a neighborhood $V$ of $\Delta$ (point (i)), but is defined using a fixed section $\sigma_0$ outside a larger neighborhood $U$ (point (ii)).

Let us now construct $\tilde S^*$. Since we are working with $\check X_B$ (i.e. on the quotient of the tangent bundle of $B$ by $\Lambda$), the local model for the fibration $\check f$ in a neighborhood of a positive node of $B$  is the negative fibration on $X_0$ defined in Example \ref{ex:fibr_res_2}, with $\delta = 0$. Lemma \ref{collar2_-fib} says that the local collar $P_0$ defined in (\ref{local_collar4}) is a compactification $\tilde S^*$ of $\tilde S_0^*$ defined using a suitable section $\check \sigma$. In fact, given $S$ as in the lemma (and $j$ the inclusion) $j(S)$ is compatible with point $(f)$ of Definition \ref{tropical_cycle}. Moreover the conclusions $(i)$ and $(ii)$ of the lemma show that near the edges of $\Delta$, $P_0$ is constructed as $\tilde S^*$ in (\ref{tss_edge}).
The assumption that $S$ is as in the lemma is not restrictive, in fact we can always locally isotope $S$ into that position. Lemma \ref{collar_pert3} ensures that $\check \sigma$ can be chosen so that $\tilde S^*$ coincides with $P_0$ over a neighborhood $V$ of $\Delta$ (point (i)) and is defined using a fixed section $\sigma_0$ outside a neighborhood $U$ (point (ii)). 

We now define $\tilde S$ and $\tilde S^*$ near a negative node of $B$. To construct $\tilde S$ we work with the negative fibration $f_{\delta}$ (with $\delta = 0$) over $X_0$ given in formula (\ref{eq:exp_res_2}). 
Lemma \ref{collar1_-fib} ensures that the local collar $Q_0$ is a compactification $\tilde S$ of $\tilde S_0$ defined using a suitable section $\sigma$. First of all, the definition of $S$ in the lemma satisfies condition $(e)$ of Definition \ref{tropical_cycle}. Moreover the lemma says that $Q_0 - (Q_0 \cap \Crit f)$ coincides with $\tilde S_0$ defined in (\ref{lift_S:sm}), for some suitable section. In fact it concludes that the intersection of $Q_0$ with a smooth fibre coincides with the translation by a section of the monodromy invariant $T^2$ around the edges which bound $S$, i.e. it coincides with a translation of $\mathcal F_q$ as prescribed by condition (h) of Definition \ref{tropical_cycle}. 
Again, assuming that $S$ is locally as in the lemma is not restrictive. An argument analogous to Lemmas \ref{collar_pert1} and \ref{collar_pert2} can be used to construct $\tilde S$ so that it coincides with $Q_0$ over a neighborhood $V$ of $\Delta$ and is defined by a fixed section $\sigma_0$ outside a neighborhood $U$. 

Let us now discuss $\tilde S^*$. We need to use the positive fibration (\ref{eq:exp_res_1}). In this case Lemma \ref{collar2_fib} tells us that the local collar $P_0$ (as defined in (\ref{local_collar3})) is a compactification $\tilde S^*$ of $\tilde S^*_0$ defined by a suitable section $\check \sigma$. First of all 
$S$, as defined in the lemma, and the inclusion $j$ satisfy point $(e)$ of Definition \ref{tropical_cycle}. Moreover point (ii) of the lemma says that $P_0$ is defined as $\tilde S^*$ in (\ref{tss_edge}). Again the assumption that locally $S$ is as in the lemma is not restrictive. Lemma \ref{collar_pert2} tells us that we can construct $\tilde S$ so that it coincides with $P_0$ over a neighborhood $V$ of $\Delta$ and is defined by a fixed section $\sigma_0$ outside a neighborhood $U$. 

\subsection{... at positive vertices.} The models of $\tilde S$ and $\tilde S^*$ at a positive vertex of $B$ are similar to previous ones. 
For the definition of $\tilde S$ we use, as the local model for the fibration $f: X \rightarrow \R^3$, the positive fibration given explicitly in \S \ref{ex. (1,2)}. 
In this case, we remarked that $f$ is invariant with respect to the $T^2$ action (\ref{pos_t2_action}) and the $T^2$ orbits coincide with the monodromy invariant $T^2$.  

Let us describe the local models for $\tilde S$ and $\tilde S^*$ explicitly for this case.

\begin{lem} Let $f: X\rightarrow \R^3$ be the positive fibration described in \S \ref{ex. (1,2)}. Let 
\[ \tilde S = \{ z_1 = 0 \} \subseteq X \]
and 
\[ S = \{ x_1 = 0, x_2 \leq 0, x_3 \leq 0 \} \subset \R^3. \]
Then,  $\partial S \subseteq \Delta$, $f(\tilde S) = S$ and there exists a map $\sigma: S \rightarrow X$ such that
\begin{itemize}
\item[i)] $f \circ \sigma = \id_{S}$ and $\sigma(\partial S) \subset \Crit f$;
\item[ii)] $\tilde S = T^2 \cdot \sigma(S)$. 
\end{itemize}
\end{lem}

\begin{proof}
Clearly $\partial S$ coincides with two edges of $\Delta$ emanating from the vertex. A simple computation shows $f(\tilde S) = S$. The fibres of $f_{|\tilde S}$ over $S$ are precisely orbits of the $T^2$ action. Consider the quotient $X/T^2 \cong \C^* \times \R \times \R$, with projection $\pi$ given in (\ref{+ve_t2pi}) and the quotient fibration $\bar f$ given in (\ref{+ve_t2bf}), such that $f = \bar f \circ \pi$. Then $\pi(\tilde S)$ coincides with the image of the section $\bar \sigma: S \rightarrow \C^* \times \R \times \R$ of $\bar f$ given by $\bar \sigma (0, x_2, x_3) = (-1, x_2, x_3)$. Any lift $\sigma: S \rightarrow X$ of $\bar \sigma$ will satisfy $(i)$ and $(ii)$. 
\end{proof}

This lemma explicitly describes, at a positive vertex, a compactification $\tilde S$ of $\tilde S_0$ defined by a suitable section $\sigma$. Clearly $\tilde S$ is a $4$-dimensional submanifold. Notice also that the definition of $S$ is compatible with Definition \ref{tropical_cycle}. Also, the assumption that $S$ is as in the lemma is not restrictive. 

Let us now describe $\tilde S^*$ inside $\check X_B$. In this case we must use the negative fibration of \S \ref{ex. (2,1)} (see Theorem \ref{cekX}). It is convenient to use the topological description given at the beginning of \S \ref{ex. (2,1)}. Recall that we have $\bar Y =T^2 \times \R^2$  and that we have the ``pair of pants" $\Sigma \subset \bar Y$ which is mapped onto $\Delta$
by the projection onto $\R^2$. Let $(x_2, x_3)$ be the coordinates on $\R^2$. Then we consider $Y = \bar Y \times \R$ and we denote by $x_1$ the extra $\R$-coordinate. We identify $\bar Y$ with $\{ x_1= 0 \}$. We then consider the space $X$ with the $S^1$ action, whose quotient is $Y$ and whose fixed point set is $\Sigma$. Now let $S = \{ x_1=0, x_2 \leq 0, x_3 \leq 0 \}$. Then $\partial S \subset \Delta$ and this definition is compatible with Definition \ref{tropical_cycle}. 
Now consider a section $\check \sigma: S \rightarrow Y$ such that $\check \sigma (\partial S) \subset \Sigma$. We define $\tilde S^* = \pi^{-1}(\check \sigma( S))$, where $\pi: X \rightarrow Y$ is the quotient by the $S^1$ action. It can be easily seen that in this case $\tilde S^*$ is homeomorphic to $\C \times \R$, since the orbits of the $S^1$ action collapse to points over $\Sigma$.

Again, $\tilde S$ and $\tilde S^*$ can be perturbed so that, away from a neighborhood of $\partial S$,  they coincide with $\tilde S_0$ and $\tilde S_0^*$ constructed from a fixed section (see Lemmas \ref{collar_pert1} and \ref{lift_edge}). Moreover the specific definitions of $S$ are not restrictive, since one can assume that up to local isotopy, near a positive vertex, $S$ is as given.

\subsection{... at negative vertices.} \label{trop_cycl_negv} We now describe models for $\tilde S$ and $\tilde S^*$ near a boundary vertex $p \in S$ (Definition \ref{tropical_domain}, condition $(e)$). Then, $j(p)$ is a negative vertex of $B$ (point $(d)$ of Definition \ref{tropical_cycle}). To construct $\tilde S$ inside $X_B$, we use as the local model for the fibration, the negative fibration of \S \ref{ex. (2,1)}. To construct $\tilde S^*$ inside $\check X_B$ we use the positive fibration of \S \ref{ex. (1,2)}. 

Let us start describing $\tilde S$. Let us consider the topological description of the negative fibration $f: X \rightarrow \R^3$ given at the beginning of \S \ref{ex. (2,1)}. Then let
\[ j(S) = \Delta \times \R_{\geq 0 } \subseteq \R^3. \]
Since $\Delta$ is homeomorphic to a tropical line $V^1$, we have that $(0,0,0) \in j(S)$ is the image of a boundary vertex (here we assume that $(0,0)$ is the vertex of $\Delta$). Now recall that $\Sigma \subset \bar Y$ is the ``pair of pants'' which is mapped by $\bar f$ to $\Delta$ and that $\bar f^{-1}(\Delta_i)$ is the cylinder $S^1 \times \Delta_i$,  where $S^1 \subset T^2$ represents the class $-e_3$, $-e_2$ and $e_2 + e_3$ when $i = 1,2,3$ respectively. If $\pi: X \rightarrow Y$ is the $S^1$ quotient defining the space $X$, then we define the lift $\tilde S \subset X$ of $j(S)$ to be 
\[ \tilde S = \pi^{-1}( \Sigma \times \R_{\geq 0}). \]
We have that $\tilde S$ is homeomorphic to $\Sigma \times D$, where $D$ is a disc in $\R^2$. Now consider 
$f_{| \tilde S}: \tilde S \rightarrow j(S)$ and let $q \in \Delta_1 \times \R_{>0} \subset j(S)$. Notice that $f_{| \tilde S}^{-1}(q) = S^1 \times S^1$, where the first factor is a circle representing $-e_3$ and the second factor is a circle representing $e_1$, the class of an orbit of the $S^1$-action. These two cycles are precisely the monodromy invariant ones with respect to monodromy around $\Delta_1$. Similarly we can argue about $f_{| \tilde S}^{-1}(q) = S^1 \times S^1$, with $q \in  \Delta_j \times \R_{>0}$ and $j=2,3$. 
Thus $f_{|\tilde S}^{-1}(\Delta_i \times \R_{\geq 0})$ coincides with the construction of $\tilde S$ done in the case of the smooth boundary points in $\Delta_i$. Notice moreover that for all $q \in \Delta_i \times \R_{>0}$, which are sufficiently away from $\Delta_i \times \{ 0 \}$, we can assume that $f_{| \tilde S}^{-1}(q) \cong T^2$ is parallel to the unique linear subspace of $f^{-1}(q)$ which is monodromy invariant with respect to monodromy around $\Delta_i$. We can also perturb $\tilde S$, so that away from $j(\partial S)$ it coincides with a lift $\tilde S_0$ constructed from a fixed section. Observe that the points on $(0,0) \times \R_{>0} \subset j(S)$ are images of interior edge points of $S$. A more detailed description of $\tilde S$ over such points will be given in \S \ref{grel_proof}.

Let us now construct $\tilde S^*$, where we use the positive fibration $f: X \rightarrow \R^3$ as the model fibration. In this case, $f$ is $T^2$ invariant with respect to the action (\ref{pos_t2_action}). 
We let $j(S) = \R_{\geq 0} \times \Delta$ and let $\check \sigma: j(S) \rightarrow X$ be a section such that $j(\partial S) \subset \Crit f$.
 Notice that $j(\partial S) = \Delta$ and $\Crit f$ consists of the union of the sets $L_j= \{ z_i=0, i \neq j \}$, with $j=1,2,3$. Each of these is mapped to one of the legs of $\Delta$.
 Denote by $\Delta_j$  the leg of $\Delta$ which is the image of $L_j$. For each $j=1,2,3$, denote by $G_j \cong S^1$ the stabilizer of the points in $L_j$ with respect to the $T^2$ action. For instance, $G_3 = \{ \xi_1 \xi_2 = 1 \}$.
Now define
\[ \tilde S^*_j = \bigcup_{p \in \R_{\geq 0} \times \Delta_j} G_j \cdot \check \sigma(p). \]
It can be checked that, since the orbits of $G_j$ collapse to single points on $L_j$, we have $\tilde S^*_j \cong D \times \R_{\geq 0}$. Where $D \times \{ 0 \}$ corresponds to the union of the $G_j$-orbits of points $\check \sigma(p)$ for all $p \in \R_{\geq 0} \times \{(0,0)\}$. 
Observe that there is a suitable choice of orientations of the $G_j$'s such that, in $H_1(T^2, \Z)$, the classes $[G_j]$ satisfy $[G_1]+[G_2]+[G_3] = 0$. Let $Q \subset T^2$ be a $2$-chain such that $\partial Q =[G_1]+[G_2]+[G_3]$ and consider the following closed subset of $X$:
\[ C = \bigcup_{p \in \R_{\geq 0} \times \{(0,0) \} } Q \cdot \check \sigma(p) \]
Define
\[ \tilde S^* = C \cup \tilde S^*_1 \cup \tilde S^*_2  \cup \tilde S^*_3. \]
Observe that $Q$ can be chosen so that $\tilde S^*$ is a submanifold homeomorphic to $\R^3$. This can be seen as follows. One can choose $Q$ homeomorphic to a standard $2$-simplex with the vertices identified (see Figure \ref{algae} (a)). Then, since the $T^2$ orbit of $\check \sigma(0,0,0)$ collapses to a point, we have that $C$ is homeomorphic to the cone of $Q$, i.e. to $\R_{\geq 0} \times Q$ after the subset $\{0\} \times Q$ is collapsed to a single point. One can then see that after attaching to this space the sets $\tilde S^*_j$
we obtain something homeomorphic to $\R^3$. 

\subsection{Pair of pants.} Before engaging in the proof of Theorem \ref{good_rel_pos} we generalize the latter argument and give a slightly more general ``pair of pants'' construction.  Let $v_1, v_2, v_3 \in \R^2$ be primitive integral vectors, spanning $\R^2$, such that
\begin{equation} \label{bal}
 v_1 + v_2 +v_3 = 0. 
\end{equation}
Given $T^2 = \R^2 / \Z^2$, for each $j=1,2,3$ the subspace $\R v_j$ covers a subgroup $G_j \cong S^1 \subseteq T^2$, with orientation induced by $v_j$.   Notice that since all $v_j$'s are primitive, under the canonical isomorphism $\Z^2 \cong H_1(T^2, \Z)$, the subgroups $G_j$ represent the class $v_j$. Let $V^{1} \subseteq \R^2$ be the tropical line as defined at the beginning of \S \ref{tr_2_cy} and denote by $p = (0,0)$ its vertex and by $D_1$, $D_2$ and $D_3$ the three edges of $V^1$ emanating from $p$, indexed in anti-clockwise order.  Let $X = \R^2 \times T^2$ and consider the following subset
of $X$: \[  \Sigma_0 = \bigcup_{j=1}^{3} D_j \times G_j. \]
We have the following
\begin{lem} \label{pair_pants}
There exists $Q \subset \{ p \} \times T^2$ such that
\[ \Sigma = \Sigma_0 \cup  Q  \]
is an embedded piecewise smooth submanifold of $X$.
\end{lem}

\begin{figure}[!ht] 
\begin{center}
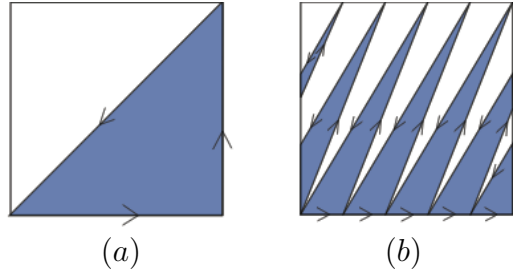
\caption{The region $Q$ (dark): (a) when $\{v_1, v_2 \}$ is a $\Z^2$-basis, (b) when it is not. } \label{algae}
\end{center}
\end{figure}
\begin{proof}
Notice that (\ref{bal}) implies that there exists a chain complex $Q$ such that $\partial Q = \cup_j G_j$. We want to show that $Q$ can be chosen so that $\Sigma$ is a (topological) manifold. 
If $\{v_1, v_2 \}$ forms a  $\Z^2$-basis, choose $Q$ to be the triangle pictured in Figure \ref{algae} $(a)$. Then  $\Sigma$ is a piecewise smooth submanifold of $X$. Notice that in principle one could also choose a different triangle, i.e. the other half of the square, but in this case its orientation should be opposite to the one induced by the choice of ordering of the $v_j$'s. In particular the ordering allows a canonical choice of $Q$. 

Let us now discuss the case where $\{ v_1, v_2 \}$ does not form a $\Z^2$-basis. In Figure \ref{algae} we have pictured the choice of $Q$ for $v_1 = (1,0)$ and $v_2 = (2,5)$. Notice that in this case $v_3= (-3, -5)$ is also primitive, as required in the hypothesis. Suppose that  $v_1$ and $v_2$ generate a sublattice of $\Z^2$ of index $m$. Let $N$ be the lattice generated by $\hat{v}_1 = m^{-1}v_1$ and $\hat{v}_2 = m^{-1}v_2$ and let $T'= \R^2/N$. Since $\Z^2 \subset N$, the identity of $\R^2$ induces a covering $\pi: T^2 \rightarrow T'$ of degree $m$. Let $\hat{v}_3 = - \hat{v}_1 - \hat{v}_2$ and let $G'_j \subseteq T'$ be the group covered by $\R \hat v_j$. Then, since $\{ \hat v_1, \hat v_2 \}$ forms a basis of $N$, we can choose the triangle $Q'$ such that $\partial Q' = \cup_j G'_j$ as above. Now let $Q = \pi^{-1}(Q')$. It can be checked that $\partial Q =  \cup_j G_j$. Moreover $\Sigma$, defined with this choice of $Q$, is a piecewise smooth submanifold of $X$.   
\end{proof}
 
\subsection{Proof of Theorem \ref{good_rel_pos}} \label{grel_proof} In the previous discussion we have shown how the zero sections of $f$ and $\check f$ can be perturbed in a neighborhood of $j(\partial S)$  so that $\tilde S$ and $\tilde S^*$ can be constructed over points on $j(\partial S)$. Now we need to show how $\tilde S$ and $\tilde S^*$ can be defined over the remaining points of $j(S)$.
We have the following cases: $p$ is one of the smooth points $q_1, \ldots, q_r$ in condition $(b)$ of Definition \ref{tropical_cycle}; $p$ is an interior vertex of $S$; $p$ is an interior edge point of $S$. 
In all these cases we will define a suitable closed subset $Q_p$ (resp. $Q^*_p$) of the fibre $f^{-1}(j(p))$ (resp. $\check f^{-1}(j(p))$) such that 
\[ \tilde S_0 \cup \left( \cup_{p} Q_p \right)  \ \ \ ( \text{resp.} \  \tilde S^*_0 \cup \left( \cup_{p} Q^*_p \right) ) \]
is a submanifold. 

\medskip

{\it Step 1: constructing $\tilde S^*$ over interior edge points.} Let us take an interior edge $e \subset S$ and a point $p \in e$. We apply point $(g)$ of Definition \ref{tropical_cycle}. Given some orientation on $e$ and a small connected neighborhood $U$ of $p$ in $S$, $j(U \cap S_{sm})$ has three connected components indexed with a unique cyclic order.  Moreover, we have the vector fields $v_1, v_2$ and $v_3$ satisfying the balancing equation. In the following argument we assume that all $\epsilon_k$'s in the balancing equation are $1$. In the general case replace $v_j$ with $-v_j$ whenever $\epsilon_j =-1$. 

Let us first define $Q^*_p$. Inside $\check f^{-1}(j(p)) = T_{j(p)}B / \Lambda_{j(p)}$, the span of the vectors $v_1, v_2$ generates a two torus $T$ and we may consider its subgroups $G_k$, $k=1,2,3,$ generated by the $v_k$'s. 
Notice that if we parallel transport $G_k$ to a point $q$ of the $k$-th component of $j(U \cap S_{sm})$, then $G_k = \mathcal V_q/ \mathcal V_q \cap \Lambda_q$. 
The balancing condition is equivalent to saying that in $H_1(T, \Z^2)$ we have $[G_1] + [G_2] + [G_3] = 0$, where the cycles are oriented by the $v_k$'s. Then we may define $Q^*_p \subset T$ as in the proof of Lemma \ref{pair_pants}. 
Notice that the choice of $Q_p^*$ is canonical, given by the ordering of the $v_k$'s. Notice also that it is independent of the choice of orientation of the edge $e$.
In fact, changing the orientation of $e$ inverts the cyclic ordering but also changes all the signs $\epsilon_k$ appearing in the balancing condition and hence the orientations of the $G_k$'s. 
It is easy to see that attaching $Q^*_p$ to $\tilde S^*_0$ for all interior edge points $p$, produces a submanifold.  

\medskip 

{\it Step 2: construct $\tilde S$ over interior edge points.} Now define $Q_p$. Inside  $f^{-1}(j(p))$ we have a triple of $2$-tori $T_1$, $T_2$ and $T_3$ defined by $v_1$, $v_2$ and $v_3$ via the bundle $\mathcal F$.
 Recall that both fibres $\check f^{-1}(j(p))$ and $f^{-1}(j(p))$ are canonically oriented. Let $w \in T^*_pB$ be primitive and integral, such that $w(v_1) = w(v_2) = 0$. 
The sign of $w$ can be chosen by imposing that $\inner{w}{v} > 0$ for all $v \in T_{j(p)}B$ such that $\{v_1, v_2, v \}$ forms an oriented $\R$-basis of $T_{j(p)}B$. Notice that $\R w / \Z w = T_1 \cap T_2 \cap T_3$. 
The vector $v_k$ and the orientation on $f^{-1}(j(p))$ induce an orientation on $T_k$. Now let $T$ be the two torus obtained as the quotient of $f^{-1}(j(p))$ by  $\R w / \Z w$. 
Then $T_k$ is mapped to a subgroup $G_k \subset T$, with an orientation. Again, the balancing condition implies $[G_1] + [G_2] + [G_3] = 0$ in $H_1(T, \Z^2)$. 
Then, as in the proof of Lemma \ref{pair_pants}, we can construct $Q'_p$ such that $\partial Q'_p = \cup_k G_k$. Define $Q_p \subset f^{-1}(j(p))$ to be the pre-image of $Q'_p$ via the quotient map $f^{-1}(j(p)) \rightarrow T$.  Then clearly $\partial Q_p = \cup_k T_k$.
 It is not difficult to show that attaching $Q_p$ to $\tilde S_0$ for all interior edge points $p$ we obtain a submanifold. 

\medskip

{\it Step 3: the constructions over interior edge points and over boundary vertices match on overlaps.} In case the interior edge ends on a boundary vertex, we should show that the constructions in Steps 1 and 2 match with the construction in a neighborhood of a boundary vertex given in \S \ref{trop_cycl_negv}. In the case of $\tilde S^*$, by parallel transport, the vector fields $v_1$ and $v_2$ can be defined on a neighborhood $U$ of the edge $e$ and they generate a $T^2$ action on $\check f^{-1}(U)$.  The torus $T$ mentioned in Step 1 is an orbit of this $T^2$ action. If $e$ ends on a boundary vertex it can be seen that this $T^2$ action coincides with the one used to construct $\tilde S^*$ in \S \ref{trop_cycl_negv}. Moreover, also the subgroups $G_k$, defined above, coincide with subgroups $G_k$ in \S \ref{trop_cycl_negv}. Therefore the two constructions coincide.

In the case of $\tilde S$, the two constructions do not strictly coincide, but they are isotopic. Therefore they can be made to coincide by using this isotopy along the edge $e$. 
First of all observe that the three tori $T_1, T_2$ and  $T_3$ near a boundary vertex must coincide with the monodromy invariant ones around each leg of $\Delta$. 
In \S \ref{trop_cycl_negv} we had $j(S) = \Delta \times \R_{\geq 0}$ and interior edge points of $j(S)$ are of type $q =(0,0,t)$ where $(0,0)$ is the vertex of $\Delta$ and $t > 0$. 
Also we had $\tilde S = \pi^{-1}( \Sigma \times \R_{\geq 0})$, where $\Sigma$ is the ``pair of pants" constructed in \S \ref{ex. (2,1)}. 
The main difference between the construction of $\tilde S$ given above and the one given in \S \ref{trop_cycl_negv} is precisely in the way we defined $\Sigma$. 
In fact, in \S \ref{ex. (2,1)} the legs of $\Sigma$ are glued together so that, if $b_0$ is the vertex of $\Delta$, then $\Sigma \cap (T^2 \times \{ b_0 \})$ is a ``figure eight''. 
We could make a different choice, e.g. we could assume that the three circles, forming the legs of $\Sigma$, come together at $T^2 \times \{b_0 \}$ as linear subspaces. 
Then we could assume that  $\Sigma \cap (T^2 \times \{ b_0 \})$ is some closed set $Q'$, such as a triangle or a pair of triangles. The main point is that any such choice would be equivalent up to isotopy and one can choose $Q'$ so that we obtain the same construction of $\tilde S$ as described above. Notice also that the $S^1$ action coincides with the action of $\R w / \Z w$. In particular one can use the isotopy to interpolate the two constructions as we move away from a boundary vertex.

\medskip

{\it Step 4: constructing $\tilde S^*$ over an interior vertex.} Now let $p \in S$ be interior vertices. Let us define $Q^*_p$. There are four interior edges of $S$ meeting at $p$. Moreover, given small connected neighborhood $U$ of $p$ in $S$, $j(U \cap S_{sm})$ has $6$ connected components. Let us orient each interior edge emanating from $p$ in the direction moving away from $p$.  Let $v_1, \ldots, v_6$ denote the vector field $v$ restricted to the six components of $j(U \cap S_{sm})$. We can assume that the indices have been chosen in such a way that the following ordered sets $\{ v_1, v_2, v_3 \}$, $\{ v_3, v_4, v_5 \}$, $\{ v_1, v_5, v_6 \}$ and $\{v_2, v_6, v_4 \}$ correspond to the triples meeting at each one of the four interior edges, ordered according to the cyclic ordering imposed by their orientations. Without loss of generality, one can see that the sign rule in the balancing condition gives the following equations
\[ v_1 + v_2 +v_3 = 0, \ \ -v_3 -v_5 - v_4 = 0, \ \ -v_1+v_6+v_5 = 0, \ \ -v_2 + v_4 - v_6 = 0. \]
Now let $Q^*_1$, $Q^*_2$, $Q^*_3$ and $Q^*_4$ denote the four subsets of $\check f^{-1}(j(p))$ constructed as above from these triples of vectors. These are obviously the limits of the sets $Q^*_{p'}$ constructed along the four edges as $p'$ approaches the vertex $p$.
Using point $(j)$ of Definition \ref{tropical_cycle} we can assume that $v_1, v_2$ and  $v_4$ are linearly independent.
 If $v_1$, $v_2$ and $v_4$ are a $\Z$-basis for $\Lambda_p$ then it can be calculated that $\cup_k Q^*_k$ is the boundary of a $3$-simplex immersed in $\check f^{-1}(j(p))$. In this case we define $Q^*_p$ to be this $3$-simplex. More generally one can use an argument similar to the one used in the proof of Lemma \ref{pair_pants} to find a suitable $Q^*_p$ such that  $\partial Q^*_p = \cup_k Q^*_k$.   
One can show that attaching to $\tilde S_0^*$ this $Q^*_p$ and the sets $Q^*_{p'}$ as above for all interior edge points,  we obtain a submanifold. 

\medskip

{\it Step 5: constructing $\tilde S$ over interior vertices.}
Now let us define $Q_p$. For each point on the four edges emanating from $p$, the previous construction gave a  $3$-chain whose boundary is the union of the three $2$-tori corresponding to that edge. 
These four $3$-chains come together at the point $p$ as subsets of $f^{-1}(j(p))$. Denote them by $Q_j$, $j=1, \ldots, 4$ and define $Q_p = \cup_{j=1}^{4} Q_j$. It can be verified that attaching to $\tilde S_0$ this $Q_p$ and the sets $Q_{p'}$ as above for all interior edge points $p'$,  we obtain a submanifold. 

\medskip

{\it Step 6: constructing $\tilde S$ over the points $q_1, \ldots, q_r$.} Let $p$ be one of the points $\{ q_1, \ldots, q_r \}$ of point $(b)$ of Definition \ref{tropical_cycle}. Let us define $Q_p$.  Given the condition on the monodromy of $\mathcal F$ around $p$ (condition ($i$) of Definition \ref{tropical_cycle}) it is natural to guess that $Q_p$ must be an $I_1$ fibre. This can be shown as follows. Let us use the construction in \S \ref{ex. (2,2)} of the generic-singular fibration over an edge of $\Delta$. Then the base of the fibration is $U=D \times (0,1)$ and the quotient of $X$ by the $S^1$ action is $Y = U \times T^2$. Let $e_2, e_3 \in H_1(T^2, \Z)$ and $\Sigma$ be defined as in \S \ref{ex. (2,2)}. If $e_1$ is the orbit of the $S^1$ action, then monodromy of $f$ is given by (\ref{eq matrix g}) in the basis $e_1, e_2, e_3$ of $H_1(X_b, \Z)$. Now let $\{ h_1,h_2 \}$ be an integral basis of $\mathcal F$ with respect to which monodromy of $\mathcal F$ around $p$ has the given form. Then it can be shown that $h_2$ represents the class $e_1$ and $h_1$ represents the class $e_2 + a e_3$ for some $a \in \Z$. 
Then by change of basis of $H_1 (T^2, \Z)$, w.l.o.g. we may as well assume that $h_1$ represents $e_2$. We may also assume that $j(S) \cap U = D \times \{1/2 \}$. If we consider a circle $S^1 \subseteq T^2$ representing the class $e_2$, then we may define 
$\tilde S = \pi^{-1}(D \times \{1/2 \} \times S^1)$, where $\pi: X \rightarrow Y$ is the projection with respect to the $S^1$ action. Then $\tilde S$ has an $I_1$ fibre over $p = (0, 1/2) \in D \times (0,1) = U$ and it is a submanifold. 

\medskip

{\it Step 7: constructing $\tilde S^*$ over points $q_1, \ldots, q_r$.} To define $Q_p^*$, we use the same model for the generic singular fibration, but since we are working on the tangent bundle one can see that condition $(i)$ of Definition \ref{tropical_cycle} implies that the vector field $v$ corresponds to the class $e_3$. In particular $v$ is monodromy invariant. This implies that $v$ induces an $S^1$ action on $\check f^{-1}(U)$ where $U$ is a neighborhood of $j(p)$. Clearly, away from $p$, the circles $\mathcal V_q / \mathcal V_q \cap \Lambda_q$ defining $\tilde S_0^*$ are orbits of this $S^1$ action. We define $Q^*_p$ to be the orbit of $\sigma_0(j(p))$ with respect to this $S^1$ action, where $\sigma_0$ is the zero section. Clearly attaching  $Q^*_p$ to $\tilde S_0^*$ gives a submanifold. 

\medskip

{\it Step 8: matching up the constructions.} We complete the argument with a remark on how to match the constructions of $\tilde S$ and $\tilde S^*$ on the overlaps between the various open sets. 
The local fibrations and the zero section $\sigma_0$ match by construction (see Corollary \ref{fibr_conifold}). 
In all local constructions, for all points $b$ on the overlaps,  $f^{-1}(b) \cap \tilde S$ (resp. $\check f^{-1}(b) \cap \tilde S^*$) is an affine subspace of the fibre uniquely prescribed by the vertor field $v$ and by a choice of section
$\sigma$ (resp $\check \sigma$). Therefore we only need to match the sections $\sigma$ (resp. $\check \sigma$). 
In all local constructions, $\sigma$ and $\check \sigma$ where chosen to coincide with the zero section away from 
a small neighborhood $V$ of $\Delta$. So on $B-V \subset B_0$ we define $\tilde S$ and $\tilde S^*$ using the zero section. The remaining case to check is on the overlap between a neighbourhood of an edge and a neighborhood of a vertex or a node. Let $f: X \rightarrow U$ be the local model for the fibration along an edge of $\Delta$. In Lemma \ref{lift_edge} we have constructed  $\sigma: S \cap U \rightarrow X$, defining $\tilde S$ (or equivalently $\tilde S^*$) along this edge. Now, $U$ may overlap with some other open set where we have a fibration over a vertex or node (or maybe another edge).  
We have $U = D \times (0,1)$ and assume that, for some $\epsilon > 0$, the open set $U' = D \times (0, \epsilon)$ is the portion of $U$ which overlaps 
with an edge emanating from a node or vertex. Over $U'$ we have another section $\sigma': U' \cap S \rightarrow X$ which defined $\tilde S$ over the vertex or node. By construction $\sigma'$ satisfies $\sigma'(U' \cap \partial S) \subset \Crit f \cap f^{-1}(U')$. In Lemma \ref{match_sigma} we have constructed a section $\sigma'': S \rightarrow X$ which interpolates $\sigma$ and $\sigma'$. Replace $\sigma$ with $\sigma''$, so that now the constructions match. Since both $\sigma'$ and $\sigma$ coincide with $\sigma_0$ outside some neighborhood $V$ of $\Delta \cap U$, by construction the same will be true of $\sigma''$.
\section{Simultaneous resolutions/smoothings of nodes} \label{simult_res_sm}
We discuss the problem of simultaneously resolving a set of nodes in a tropical conifold $(B, \mathcal P, \phi)$. More precisely, given a set of nodes in $B$ and hence in $X_B$, the topological conifold associated to $B$, we want to construct a new tropical manifold $(B', \mathcal P', \phi')$ whose topological compactification $X_{B'}$ is homeomorphic to the resolution of the corresponding nodes in $X_B$. To achieve this we need to change the affine structure on $B$ so that a neighborhood of the node is replaced by a neighborhood which is (locally) affine isomorphic to its tropical resolution described in \S \ref{trop_res_+ve} and \S \ref{trop_res_-ve}. This is not a local problem. In general, we should expect global obstructions. Intuitively, resolving a node implies the insertion of an extra $2$-dimensional face (in the positive case) or an extra $3$-dimensional polytope (in the negative case), causing a modification of nearby polytopes which propagates away from the node. 

As we have discussed in \S \ref{trop_res_+ve} and \S \ref{trop_res_-ve}, the local resolution of a node can be achieved by smoothing the mirror node and then applying discrete Legendre transform. Therefore, the problem of tropically resolving a set of nodes is equivalent to the one of tropically smoothing the mirror ones. In this sense the phrase ``simultaneously resolving/smoothing a set of nodes'' also means simultaneously resolving the nodes on $B$ \emph{and} smoothing the mirror nodes on $\check B$. In Theorem \ref{good_rel_pos} we have proved that the existence of tropical $2$-cycles in $B$ containing the nodes guarantees that the obstructions to resolve the nodes in $X_{B}$ and to smooth the mirror ones in $X_{\check B}$ vanish simultaneously. This suggests the idea that the existence of tropical $2$-cycles could imply the vanishing of the obstructions to the tropical resolution of nodes.

So, let  $p_1, \ldots, p_{k+s}$ be a set of nodes of $\check B$, where $p_1, \ldots, p_k$ are negative and $p_{k+1}, \ldots, p_{k+s}$ are positive.  The negative ones are the barycenters of square $2$-dimensional faces $e_1, \ldots, e_k$ and the positive ones are barycenters of $1$-di\-men\-sio\-nal edges $\ell_{k+1}, \ldots, \ell_{k+s}$. Then, to simultaneously smooth these nodes we want to do the following:
\begin{itemize}
\item[a)] find a refinement $\check{\mathcal P}'$ of $\check{\mathcal P}$, which is a toric subdivision, inducing a diagonal subdivision of $e_1, \ldots, e_k$ and the barycentric subdivision of $\ell_{k+1}, \ldots, \ell_{k+s}$;
\item[b)] define a suitable fan structure, compatible with $\check{\mathcal P}'$, at the barycenters of the edges $\ell_{k+1}, \ldots, \ell_{k+s}$ which has the effect of changing the discriminant locus as in \S \ref{trop_res_-ve};
\item[c)] construct an MPL-function $\check \phi'$, strictly convex with respect to $\check {\mathcal P}'$.
\end{itemize} 
This has to be done so that we obtain a new tropical manifold (or conifold, if there are other nodes) $(\check B', \check {\mathcal P}', \check \phi ')$ and the corresponding manifold (or conifold) $X_{\check B'}$ is homeomorphic to the smoothing of $X_{\check B}$ at the nodes corresponding to $p_1, \ldots, p_{k+s}$. 
If this can be done, then the mirror $(B', \mathcal P', \phi')$ of $(\check B', \check {\mathcal P}', \check \phi ')$ gives the simultaneous resolution of the mirror nodes. We remark that in point $(a)$ the subdivision of the edges $\ell_{k+1}, \ldots, \ell_{k+s}$ does not have to be necessarily barycentric, it may simply be a subdivision given by adding one vertex at an interior integral point.

\subsection{Related nodes.}
Let $p$ be a node in a tropical conifold and consider two tropical $2$-cycles $(S_1, j_1, v_1)$ and $(S_2, j_2, v_2)$. To avoid cumbersome notation we identify $S_k$ with its image $j_k(S_k)$. If $p \in S_1 \cap S_2$,  the vector fields $v_1$ and $v_2$ are always monodromy invariant in a neighborhood of $p$, therefore, by parallel transport, they can be compared. In the case $p$ is a positive node, they are either equal or opposite to each other. When $p$ is a negative node, we have two cases. If $	\partial S_1$ and  $\partial S_2$ coincide near $p$, then $v_1$ and $v_2$ are either equal or opposite. If $\partial S_1$ and $\partial S_2$ intersect transversally in $p$, then $v_1$ and $v_2$ form a basis of the monodromy invariant plane which locally contains $\Delta$. In the following the orientation on $\partial S_k$ is induced from the orientation on $S_k$.
\begin{defi} \label{coeff_eps} Let $S_1$ and $S_2$ be tropical $2$-cycles an let $p$ be a node. We can uniquely define a coefficient  $\epsilon_{S_1 S_2}(p)$ with the following properties. If either $S_1$ or $S_2$ does not contain $p$, then $\epsilon_{S_1 S_2}(p)= 0$. Assume $p \in S_1 \cap S_2$. When $p$ is positive, $\epsilon_{S_1 S_2}(p) = 1$ in the following cases
\begin{itemize}
\item[1)] $\partial S_1$, $\partial S_2$ and their orientations coincide in a neighborhood of $p$ and $v_1=v_2$;
\item[2)] $S_1$, $S_2$ and their orientations are as in Figure \ref{canc} (a) and (b) and $v_1=v_2$;
\end{itemize}
Now assume $p$ is negative. If $\partial S_1$ and $\partial S_2$ intersect transversally in $p$, then their orientations and their ordering (according to their index) induces an orientation of the monodromy invariant plane containing $p$. Then $\epsilon_{S_1S_2}(p)= 1$ in the following cases
\begin{itemize}
\item[3)] $\partial S_1$ and $\partial S_2$ intersect transversally in $p$ and the orientation they induce on the monodromy invariant plane is opposite to the orientation induced by  $v_1 \wedge v_2 $;
\item[4)]$\partial S_1$, $\partial S_2$ and their orientations coincide in a neighbourhood of $p$ and $v_1=v_2$.
\end{itemize}
All other cases are uniquely determined by the property that $\epsilon_{S_1S_2}(p)$ changes sign if we either change the orientation of one of the $S_j$'s or change the sign of one of the $v_j$'s. Notice that 
$\epsilon_{S_1S_2}(p) = \epsilon_{S_2S_1}(p)$ and $\epsilon_{S_1S_1}(p) =1$ if and only if $p \in S_1$.
\end{defi}

\begin{figure}[!ht] 
\begin{center}
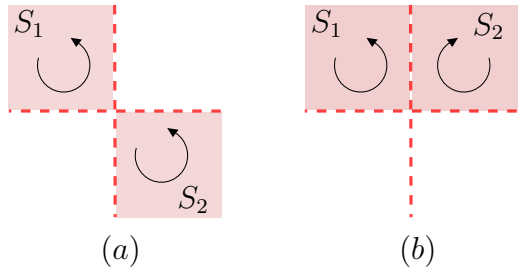
\caption{Tropical $2$-cycles at a positive node. If the orientations are as pictured and $v_1 = v_2$, then $\epsilon_{S_1S_2}(p) = 1$.} \label{canc}
\end{center}
\end{figure}

\begin{defi} \label{pos_nodes_rel}
A given set of nodes $p_1, \ldots, p_k$ of a tropical conifold $(B, \mathcal P, \phi)$ is said to be 
\begin{itemize}
\item[i)]{\it $\omega$-related} if the vanishing cycles associated to the corresponding nodes in $X_{B}$ satisfy a good relation;
\item[ii)] {\it $\C$-related} if the exceptional $\PP^1$'s in some small resolution of the corresponding nodes in $\check X_{B}$ satisfy a good relation;
\item[iii)]{\it related} if there exist tropical $2$-cycles $(S_1, j_1) \ldots, (S_r, j_r)$ such that it coincides with the set of nodes satisfying
\begin{equation} \label{rel_coef}
 \sum_{l=1}^{r} \epsilon_{S_kS_l}(p) \neq 0, \ \ \ \text{for at least one} \ k \in \{1, \ldots, r \}. 
\end{equation}
\end{itemize}
\end{defi}

For the motivation behind Definitions \ref{coeff_eps} and \ref{pos_nodes_rel} see \S \ref{example_good_rel}. 

We conjecture the following 

\begin{con} \label{trop_cycle_conj} Given a set of nodes $p_1, \ldots, p_k$ in a tropical conifold then
\begin{itemize}
\item[i)] the notions of related, $\omega$-related and $\C$-related are equivalent;
\item[ii)] the property that the set can be simultaneously tropically re\-solved/smo\-othed by the above process is equivalent to some property of the set expressible purely in terms of tropical $2$-cycles containing the nodes. 
\end{itemize}
\end{con} 

Although we believe this conjecture to be morally true, we do not exclude that refinements in our definition of tropical $2$-cycle will be necessary. We do not know if $\omega$ or $\C$-related is equivalent to the fact that the set can be simultaneously resolved/smoothed. In Proposition \ref{grel_comb} we prove that $\omega$-related implies related in a class of examples of compact tropical conifolds. These examples and the next paragraph show evidence of (ii). 

\begin{rem} \label{rescalings} In the proof of the following cases, in order to preserve integrality, we allow rescalings of the affine structure on $B$ (or $\check B$). This means that we rescale all polytopes of $\mathcal P$ (or $\check {\mathcal P}$) by a factor of $N \in \N$. 
\end{rem}

\subsection{Some special cases.} Suppose that there is a polytope $P \in \mathcal P$ of the type $P = L \times [0, m]$, where $L$ is a $2$-dimensional convex integral polytope in $\R^2$ and consider $S = L \times  \{ m/2 \}$. We assume that $\partial S \subset \Delta$, that the vertices $p_1, \ldots, p_k$ of $S$ are positive nodes and that $\partial S$ does not contain any other vertices of $\Delta$. These nodes are related (see Figure~\ref{pos_nodes_3}) by the tropical $2$-cycle $(S, j, v)$, where $j$ is the inclusion and $v$ is the vector field parallel to the edges containing the nodes. 

\begin{figure}[!ht] 
\begin{center}
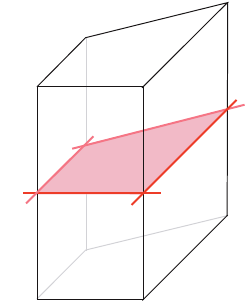
\caption{}
\label{pos_nodes_3}
\end{center}
\end{figure}

\begin{thm} \label{resolve_nodes_1}
If positive nodes $p_1, \ldots, p_k$ in a tropical conifold $(B, \mathcal P, \phi)$ are in a configuration as above, then they can be tropically resolved.
\end{thm}

\begin{proof}
We smooth the mirror nodes. First observe that 
since the positive nodes coincide with the vertices of $S$ and $\partial S \subseteq \Delta$, then $L$ (and therefore $S$) must be a Delzant polytope. This means that the tangent wedge at each vertex of $L$ is generated by primitive integral vectors which form a $\Z^2$ basis. 
Let $q_1, \ldots, q_k$ be the vertices of $L$, ordered so that $q_j$ and $q_{j+1}$ belong to the same edge (and indices are cyclic). Then the edges of $P$ containing the nodes are given by $e_j = q_j \times [0,m]$ and we denote the vertices of these edges by $q_j^+ = (q_j, m)$ and $q_j^- = (q_j,0)$. Let us denote by $Q_j^+$ and $Q_j^-$ the $3$-dimensional polytopes of $\check {\mathcal P}$ which are dual to the vertices $q_j^+$ and $q_j^-$ respectively. It is clear that $Q_j^+$ and $Q_j^-$ will have as common face the square face $\check e_j$ which is dual to the edge $e_j$. 
Now all the faces $\check e_j$ intersect in a common vertex $v_0$ which is dual to the polytope $P$. 
Moreover, $\check e_{j-1}$ and $\check e_{j}$ intersect in an edge, which we denote $\ell_j$. One of the vertices of $\ell_j$ is $v_0$ and we denote the other one by $v_j$ (see Figure~\ref{pos_nodes_1}). The fan structure at $v_0$ is given by the normal fan of $P$, which we denote $\Sigma_P$. If $\Sigma_L \subset \R^2$ denotes the normal fan of $L$, with $2$-dimensional cones $C_1, \ldots, C_k$, then $\Sigma_P$ is clearly given by the $3$-dimensional cones $C_j^+ = C_j \times [ 0, + \infty)$ and $C_j^- = C_j \times (-\infty, 0]$. We have that $C_j^+$ and $C_j^-$ are the tangent wedges at $v_0$ of the polytopes $Q_j^+$ and $Q_j^-$ respectively.  Notice that emanating from $v_0$ we have the edges $\ell_j$, all lying in a plane, plus two more edges transversal to this plane which we denote $\ell^+$ and $\ell^-$, belonging to $Q_j^+$ and $Q_j^-$ respectively. In affine coordinates $\ell^+$ and $\ell^-$ have opposite directions. 

\begin{figure}[!ht] 
\begin{center}
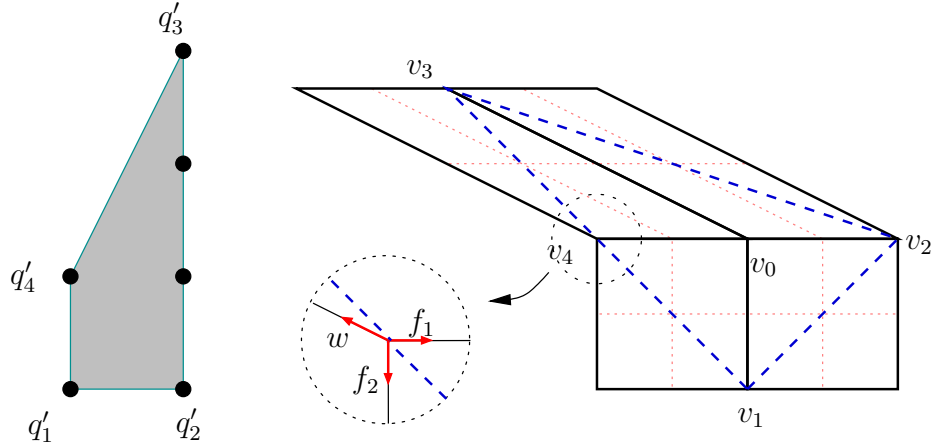
\caption{On the left we see the polytope $L$ and on the right the square faces $\check e_1, \ldots, \check e_4$ in the mirror. The (dark) dashed diagonals give the subdivision and we have also drawn the discriminant locus (light dotted lines).} \label{pos_nodes_1}
\end{center}
\end{figure}

We now describe the refinement $\check {\mathcal P}'$ of $\check {\mathcal P}$. We subdivide the faces $\check e_j$ by taking the diagonal joining $v_j$ to $v_{j+1}$. We extend this subdivision as follows. The union of all the diagonals from $v_j$ to $v_{j+1}$ encloses a $2$-dimensional region containing the vertex $v_0$. We denote this region by $\check S$. Notice that $\check S \cap \check \Delta$ consists only of the (negative) nodes in the barycenters of the faces 
$\check e_j$ and the edges connecting them. Therefore there exists an open neighborhood $U$ of the relative interior of $\check S$ such that $\bar U \cap \check \Delta = \check S \cap \check \Delta$, where $\bar U$ is the closure of $U$. Moreover we can assume that $\bar U \cap Q_j^{\pm}$ is convex, for all $j$. Now rescale $B$ (see Remark \ref{rescalings}) so that the edges $\ell^+$ and $\ell^-$ contain at least one interior integral point which is also contained in $U$.  Let $v^+$ and $v^-$ be the interior integral points of $\ell^+$ and $\ell^-$ respectively, closest to $v_0$.  
Now subdivide each $Q_j^{\pm}$ in two convex polytopes: one being  $V_{j}^{\pm} = \conv (v^{\pm}, v_0, v_j, v_{j+1})$ and the other the closure of its complement. By the convexity of $U \cap Q^{\pm}_j$, we also have $V_{j}^{\pm} \subseteq \bar U$, for all $j$.  It is clear that this gives a well defined refinement $\check {\mathcal P}'$ of $\check {\mathcal P}$ which extends the given diagonal subdivision of the faces. 
It follows from the choice of $U$ and the fact that $V_{j}^{\pm} \subseteq \bar U$ that, among the new facets of codimension at least $1$, the only ones whose relative interiors intersect $\Delta$ are the diagonals of the faces $\check e_j$. This ensures that the decomposition $\check {\mathcal P}'$ alters the discriminant locus only at the nodes contained in the faces $\check e_j$.  Moreover it implies that the subdivision is toric (in the sense of Gross and Siebert, see \S \ref{polyhedral_dec} and \ref{mon_properties}). This follows from Proposition 1.32 of \cite{G-Siebert2003} where it is shown that a polyhedral subdivision is toric if and only if for every facet $\tau$ there exists a neighborhood $U_{\tau}$ of $\inter(\tau)$ such that, if $b \in \tau - \Delta$ and $\gamma \in \pi_1(U_{\tau}- \Delta, b)$, then $\tilde \rho(\gamma)(w) - w$ is tangent to $\tau$ for every $w \in T_bB_0$. This property trivially holds for the new facets whose interiors do not intersect $\Delta$. For the diagonals of the faces $\check e_j$ it can be easily verified (it can be done on the local models). 

We now define $\check \phi'$. First we construct an MPL-function $\tilde \phi$, which we may regard as ``supported in a neighborhood of $\check S$''.  
Notice that $\ell_j$ and $\ell_{j+1}$ are edges of a square face, hence all $\ell_j$'s have the same integral length, which we denote by $n$. On the fan at $v_0$, $\tilde \phi$ maps the integral generator of the one dimensional cone corresponding to $\ell_j$  to $-1$, the integral generator of $\ell^{+}$ to $-2n+1$ and the integral generator of $\ell^{-}$ to  $0$. On the fan at $v_j$, $\tilde \phi$ maps the integral generator of $\ell_j$ to $1$ and all other edges to zero.
On the fan at $v^{\pm}$,  $\tilde \phi$ maps the generator of $\ell^{\pm}$ to $n$ and all other edges to zero.
At all remaining vertices of $\check {\mathcal P}'$, $\tilde \phi$ is defined to be zero. We claim that
$\tilde \phi$ is a well defined MPL-function. Once we have proved this, we define 
\begin{equation} \label{new_function}
 \check \phi' = N \check \phi + \tilde \phi, 
\end{equation}
for some integer $N$ and we claim that for $N$ sufficiently big, $\check \phi'$ is
strictly convex.

We now prove that $\tilde \phi$ is well defined. 
We need to show that, along every edge of $\check {\mathcal P}'$ connecting vertices $v$ and $v'$, the quotient functions, computed at $v$ and $v'$, match.  We only do this for the pair of vertices $v_j$ and $v^+$. The other cases are similar. At $v_j$,  construct a basis $\{ f_1, f_2, f_3 \}$ for $T_{v_j} \check B$ as follows. Let $f_1$ be the integral generator of the edge $\ell_j$ and $f_2$  the integral generator of the other edge of $\check e_j$ adjacent to $v_j$ (hence parallel to $\ell_{j+1}$). Notice that if we denote by $w$ the integral generator of the other edge of $\check e_{j-1}$ adjacent to $v_j$ (hence parallel to $\ell_{j-1}$) then $w = - f_2 + k_j f_1$ for some $k_j$ (see Figure~\ref{pos_nodes_1}). Now at $v_0$, consider the vector $v^+ - v_0$ and parallel transport it to $T_{v_j} \check B$ via a curve from $v_0$ to $v_j$ passing into $V^+_{j}$. We define $f_3$ to be the resulting vector. Notice that, if we parallel transport $f_3$ back along the same curve and then again to $v_j$ along a curve passing into $V_j^-$, then we obtain $f_3 + f_1$. This is the effect of monodromy around the component of the discriminant locus passing trough the edge $\ell_j$. This implies that if we parallel transport $v^- - v_0$ to $v_j$ along a curve passing into $V_j^-$ then we obtain $- f_3 - f_1$, since in affine coordinates at $v_0$ we have $v^- - v_0 = -(v^+ - v_0)$. Notice that the tangent direction to the edge from $v_j$ to $v^+$ at $v_j$ is $n f_1 + f_3$, where $n$ is the integral length of $\ell_j$. Similarly the tangent direction to the edge from $v_j$ to $v^-$ is $(n-1)f_1 - f_3$. 
Finally, in the fan structure at $v_j$, the tangent wedges of the polytopes $V_j^+$, $V_j^-$, $V_{j-1}^{-}$, $V_{j-1}^+$ correspond respectively to the cones
 \[ \cone( f_1, f_1 + f_2, nf_1 + f_3), \ \ \cone( f_1, f_1 + f_2, (n-1)f_1 - f_3), \]
 \[ \cone(f_1, w + f_1, (n-1)f_1 - f_3),  \ \ \cone( f_1, w+ f_1, nf_1 + f_3 ). \]
As defined, $\tilde \phi$ has value $1$ at $f_1$ and zero at all other one dimensional cones of the fan. Clearly, this uniquely defines integral linear maps on each cone.


We now compute the quotient along the edge from $v_j$ to $v^+$, i.e. along $nf_1 + f_3$. 
We choose a basis for the quotient space to be $\{ [f_1], [f_2]\}$.  The quotient fan will have four maximal cones, two of which are $V_{j-1}^+$ and $V_j^+$. Since $[f_1 + f_2] = [f_1] + [f_2]$ and $[w + f_1] = (k_j + 1) [f_1]-[f_2]$, the four cones are $\cone([f_1],[f_1] + [f_2])$, $\cone( [f_1], (k_j +1) [f_1] - [f_2])$, $\cone( -[f_1], [f_1] + [f_2])$ and $\cone( -[f_1], (k_j +1) [f_1] - [f_2])$, where the first two correspond to $V_{j}^+$ and $V_{j-1}^+$ respectively. The quotient function of $\tilde \phi$ maps $[f_1]$ to $1$ and all other generators of one dimensional cones to zero. 


We now come to the vertex $v^+$. Consider the normal fan of $L$ in $\R^2$ and view it as embedded in $\R^3$ by the first two coordinates and denote $v = (0,0,1)$. Then the fan at $v^+$ is given by the cones $K_j^-=\cone( n \bar \ell_j - v, \, n \bar \ell_{j+1} - v, \, -v )$ and $K_j^+= \cone( n \bar \ell_j - v, \, n \bar \ell_{j+1} - v, \, v )$, where $j=1, \ldots, k$. The cones $K_j^-$ correspond to the polytopes $V_{j}^+$, hence $-v$ points in the direction of $v_0$ and $n \bar \ell_j - v$ in the direction of $v_j$. Then $\tilde \phi$ has value $n$ on $-v$ and $0$ at all other generators of one dimensional cones. 


We now check the quotient along $n \bar \ell_j - v$. A basis for the quotient is $\{ [\bar \ell_j], [\bar \ell_{j+1}] \}$. Notice that $[v] = n [ \bar \ell_j ]$, $[n \bar \ell_{j-1} - v ] = -n(k_j +1)[ \bar \ell_j]- n [\bar \ell_{j+1}]$ and $[n \bar \ell_{j+1} - v] = n[\bar \ell_{j+1}] - n [\bar \ell_j]$. Therefore we have the four cones in the quotient:
$\cone( -(k_j +1)[ \bar \ell_j]-  [\bar \ell_{j+1}], -[\bar \ell_j] )$, $\cone( [\bar \ell_{j+1}] - [\bar \ell_j], - [\bar \ell_{j}] )$, $\cone( -(k_j +1)[ \bar \ell_j]-  [\bar \ell_{j+1}], [\bar \ell_j])$, 
$\cone( [\bar \ell_{j+1}] - [\bar \ell_j],  [\bar \ell_{j}] )$, corresponding to the cones $K_{j-1}^-$, 
$K_{j}^-$, $K_{j-1}^+$ and $K_{j}^+$ respectively. Now the quotient function maps $-[\ell_j]$ to $1$ and all other generators of one dimensional cones to zero. This shows that the functions along the edge from $v^+$ to $v_j$ coincide at the vertices $v^+$ and $v_j$. Similarly for all other vertices. 


Now we prove that $\check \phi'$ as defined in (\ref{new_function}) is strictly convex. 
First let us recall some useful facts. Suppose $\Sigma$ is a complete fan in $\R^3$ and $\phi$ a piecewise linear function on $\Sigma$. Let $F$ be a two dimensional cone of $\Sigma$ and $C$, $C'$ two maximal cones such that $F = C \cap C'$. Then $\phi$ restricts to linear maps $m$ and $m'$ on $C$ and $C'$ respectively, such that $(m-m')|_{F} = 0$. Now consider a $\Z^3$-basis $\{ f_1, f_2, f_3 \}$ of $\R^3$, such that $f_1$ and $f_2$ generate the plane containing $F$ and $f_3$ points towards the interior of $C$. Then the integer $h_{\phi}(F) = \inner {m-m'}{f_3}$ is independent of the chosen basis. The function $\phi$ is strictly convex if and only if $h_{\phi}(F) > 0$ for every two dimensional cone $F$ of $\Sigma$.   Clearly $h_{N\phi}(F) = N h_{\phi}(F)$ for every integer $N$ and $h_{\phi + \phi'}(F) = h_{\phi}(F) + h_{\phi'}(F)$. 

The fan at the vertex $v_0$ was unaffected by the subdivision $\check{\mathcal P}'$. Now recall that if we consider the space of not necessarily integral, piecewise linear functions on a fan $\Sigma$, then the strictly convex ones form an open cone inside this space. This implies that $\check \phi'$ is strictly convex at $v_0$ for sufficiently large $N$, since $\check \phi$ is strictly convex. 

Now consider the vertices $v_j$. Here the fan structure has been changed by the subdivision. Let $\Sigma$ be the fan before the subdivision and $\Sigma'$ the new fan. A two dimensional cone $F$ of $\Sigma'$ can be of two types: it either intersects the interior of a maximal cone of $\Sigma$ or it is entirely contained in a two dimensional cone of $\Sigma$.   
In the latter case we have $h_{\check \phi}(F) > 0$ and therefore for sufficiently large $N$, $h_{\check \phi'}(F) = N h_{\check \phi} (F) + h_{\tilde \phi}(F) > 0$. In the former case $h_{\check \phi}(F) = 0$, but it can be easily checked by direct calculation that $h_{\tilde \phi}(F) > 0$ for all such $F$'s. Therefore we have again $h_{\check \phi'}(F) > 0$. 
The case for the vertices $v^+$ and $v^-$ is analogous. This completes the proof.
\end{proof}

\begin{cor} \label{resolve_nodes_cor}
Suppose that the sets $K = \{ p_1, \ldots, p_k \}$ and $K' = \{ p_1', \ldots, p_r' \}$  of positive nodes in $B$ are each in a configuration like in Theorem~\ref{resolve_nodes_1} such that $K$ and $K'$ are the corners of $S$ and $S'$ respectively. Assume moreover that $p_1 = p_1'$ and that $S \cap S' = \{ p_1 \}$. Then the nodes $K \cup K'$ are related and can be simultaneously resolved.  
\end{cor}

\begin{proof}
To show that the nodes $K \cup K'$ are related it is enough to chose orientations on $S$ and $S'$ so that, at the node $p_1$ they are like in Figure \ref{canc} (a) and then chose the sign of the vector fields $v$ and $v'$ so that they coincide in a neighborhood of $p_1$.  

Let $P = L \times [0, m]$ and $P'= L' \times [0,m]$ be the polytopes of $\mathcal P$ such that 
$S = L \times \{ m/2 \}$ and $S' = L' \times \{ m/2 \}$. Then $P$ and $P'$ have only one edge in common, i.e. the one containing $p_1$. Denote it by $e_1$. Notice that if we carry out the construction explained in Theorem~\ref{resolve_nodes_1} for $P$ or for $P'$ we obtain 
the same subdivision of the two-face $\check e_1$, mirror of $e_1$. Denote by $v_1$ and $v_2$ the opposite vertices of $\check e_1$ forming the diagonal in such a subdivision.  Also denote by $v_0$ and $w_0$ the vertices mirror to $P$ and $P'$ respectively. Then the vertices of $\check e_1$ are 
$v_1$, $v_2$, $v_0$ and $w_0$. Moreover, emanating from $v_0$ we have two edges $\ell^+$ and $\ell^-$ and the integral points $v^+$ and $v^-$ on them, as in the proof of the theorem. Let us denote by $\lambda^+$, $\lambda^-$ and $w^+$ and $w^-$ the analogous edges and points emanating from $w_0$. Let $Q_1^+$ and $Q_1^-$ be the polytopes such that $Q_1^+ \cap Q_1^- = \check e_1$.
We subdivide $Q_1^{\pm}$ by the polytopes $\conv(v_0, v_1,v_2, v^{\pm})$, $\conv(w_0, v_1,v_2, w^{\pm})$ and the closure of the complement these two. All other polytopes are subdivided as in Theorem~\ref{resolve_nodes_1} applied to $P$ and $P'$. Moreover, the theorem applied to $P$ and $P'$ gives MPL-functions $\tilde \phi_1$ and $\tilde \phi_2$ on $\check B$ respectively. Then we define
\[ \check \phi' = N \check \phi + \tilde \phi_1 + \tilde \phi_2, \]
which, for sufficiently large $N$, is strictly convex. 
\end{proof}

We can also generalize these results to a configuration of nodes as depicted in Figure~\ref{pos_nodes_4}.

\begin{figure}[!ht] 
\begin{center}
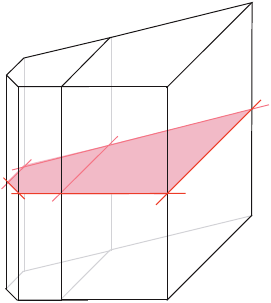
\caption{}
\label{pos_nodes_4}
\end{center}
\end{figure}

\begin{thm} \label{resolve_nodes_2} Suppose that the sets $K = \{ p_1, \ldots, p_k \}$ and $K' = \{ p_1', \ldots, p_r' \}$  of positive nodes in $B$ are each in a configuration like in Theorem~\ref{resolve_nodes_1} such that $K$ and $K'$ are the corners of $S$ and $S'$ respectively. Assume moreover that $S \cap S'$ is an edge of $\Delta$ having as vertices $p_1' = p_1$ and $p_2=p_2'$. Then the sets of nodes $K \cup K'$ and $(K \cup K')- \{ p_1, p_2 \}$ are related and can be simultaneously tropically resolved.  
\end{thm}

\begin{proof}
To show that the set of nodes $K \cup K'$ is related choose orientations on $S$ and $S'$ so that near the nodes $p_1$ and $p_2$ they are like in Figure \ref{canc} (b) and then choose the signs of the vector fields $v$ and $v'$ so that they coincide near $p_1$ and $p_2$. To show that the set of nodes $(K \cup K')- \{ p_1, p_2 \}$ is related, change the orientation of $S$. 
Let $P = L \times [0, m]$ and $P'= L' \times [0,m]$ be the polytopes of $\mathcal P$ such that 
$S = L \times \{ m/2 \}$ and $S' = L' \times \{ m/2 \}$. 
Then, if $q_1, \ldots, q_k$ are the corners of $L$ and $q_1', \ldots, q_r'$ are the corners of $L'$ we assume that the nodes $p_j$ and $p_j'$ correspond to $(q_j , m/2)$ and $(q_j', m/2)$ respectively. 
Since we assume that $p_1 = p_1'$ and $p_2 = p_2'$, we have that $(q_1, m/2)$ and $(q_2, m/2)$ are identified with $(q_1', m/2)$ and $(q_2', m/2)$ via an identification of the corresponding $2$ dimensional faces of $P$ and $P'$ respectively. 

\begin{figure}[!ht] 
\begin{center}
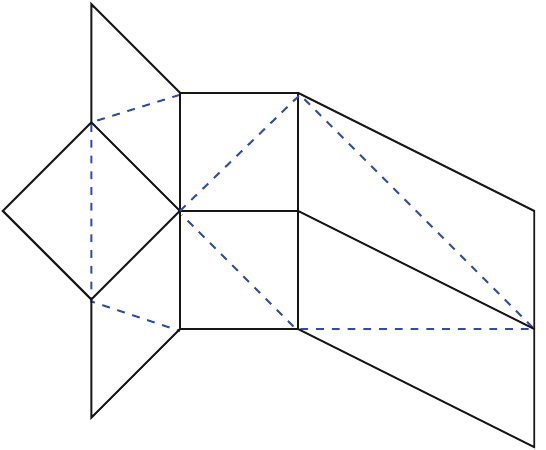
\caption{}
\label{pos_nodes_2}
\end{center}
\end{figure}

Now let us look at the mirror. As in the proof of Theorem~\ref{resolve_nodes_1}, we denote by $Q^+_j$ and $Q^-_j$ the maximal polytopes of $\mathcal P '$ corresponding to the vertices $(q_{j},m)$ and $(q_j,0)$ respectively and by $R^{+}_{j}$ and $R^{-}_{j}$ those corresponding to $(q_{j}',m)$ and $(q_j',0)$. Clearly we have $Q^{\pm}_{j} = R^{ \pm}_{j}$ when $j=1,2$. 
Also denote by $\check e_j$ and $\check d_j$  respectively the square faces $Q_j^+ \cap Q_j^-$ and $R_j^+ \cap R_j^-$. We also denote the edges $\ell_j = \check e_{j-1} \cap \check e_{j}$ and $\lambda_j = \check d_{j-1} \cap \check d_{j}$. Clearly we have $\check e_{j} = \check d_j$ for $j=1,2$ and $\ell_2 = \lambda_2$. We let $v_0$ be the vertex dual to $P$ and $w_0$ the vertex dual to $P'$.  We have that $v_0$ and $w_0$ are connected by the edge $\ell_2$. 
The edges $\ell_j$ and $\lambda_j$ emanate from $v_0$ and $w_0$ respectively (see Figure~\ref{pos_nodes_2}) and we denote by $v_j$ and $w_j$ the other vertices of $\ell_j$ and $\lambda_j$ respectively (here $j \neq 2$). We assume that $v_0, w_0, w_3, v_3$ are the vertices of $\check e_2$ and $v_0, w_0, w_1, v_1$ are the vertices of $\check e_1$. We denote by $\ell^{+}$ and $\ell^-$ (resp. $\lambda^+$ and $\lambda^-$) the other two edges emanating from $v_0$ (resp. $w_0$) which are mirror respectively to the faces $L \times \{ m \}$ and $L \times \{ 0 \}$ (resp. $L' \times \{ m \}$ and $L' \times \{ 0 \}$). We let $v^+$ and $w^+$  (resp. $v^-$ and $w^-$) be the interior integral points on $\ell^+$ and $\lambda^+$ (resp. $\ell^-$ and $\lambda^-$) which are closest to $v_0$ and $w_0$ respectively. 

Let us first smooth the negative nodes lying inside all the faces $\check e_j$ and $\check d_j$ except $j=1,2$. This resolves the set $(K \cup K')- \{ p_1, p_2 \}$. Subdivide the faces $\check e_j$ and $\check d_j$, with $j \neq 1,2$, by adding the diagonals from $v_j$ to $v_{j+1}$ and from $w_j$ to $w_{j+1}$. 
We leave $\check e_1$ and $\check e_2$ as they are for the moment. We extend this subdivision as follows. 
The polytopes $Q^{\pm}_{j}$ and $R^{\pm}_j$ with $j \neq 1,2$ are subdivided just like in the proof of Theorem~\ref{resolve_nodes_1}, e.g. $Q^{+}_{j}$ is subdivided by taking $\conv(v_0, v_{j}, v_{j+1}, v^{+})$ as one polytope and the closure of its complement as the other one. 
We subdivide $Q^{\pm}_1$ (resp $Q^{\pm}_2$) by taking  $\conv(v_0, w_0, v^{\pm}, w^{\pm}, v_1, w_1)$ (resp. $\conv(v_0, w_0, v^{\pm}, w^{\pm}, v_3, w_3)$) as one polytope and the closure of its complement as the other. 
The fact that the latter is convex (and hence that the subdivision is well defined) follows from the observation that the vertices $v_1$, $w_1$, $v^{\pm}$ and $w^{\pm}$ are coplanar in $Q_1^{\pm}$ (similarly the vertices  $v_3$, $w_3$, $v^{\pm}$ and $w^{\pm}$ in $Q_2^{\pm}$). This is a consequence of the fact that the tangent wedges to $Q_1^{\pm}$ at the vertices $v_0$ and $w_0$ are smooth cones (since $L$ and $L'$ are Delzant). Using the same argument as in Theorem \ref{resolve_nodes_1} we can assume that no parts of the discriminant $\check \Delta$ are affected by this subdivision other than the nodes we want to smooth. 

We now define an MPL-function $\tilde{\phi}_1$. On the fan at $v_0$ (resp. $w_0$), $\tilde \phi_1$ takes the value $-1$ on the primitive, integral tangent vectors to the edges $\ell_j$ (resp $\lambda_j$) and $0$ on the primitive, integral tangent vector to $\ell_2$ (reps. $\lambda_2$). On the primitive tangent vector to $\ell^+$ (resp $\lambda^+$), $\tilde \phi_1$ takes the value $-2n + 1$ and on the primitive tangent vector to $\ell^-$ (resp $\lambda^-$) it takes the value $0$. On the fan at $v_j$ (resp. $w_j$), $j \neq 2$, $\tilde \phi_1$ takes the value $1$
on the primitive integral tangent vector to the edge $\ell_j$ (resp. $\lambda_j$) and zero on all other one dimensional cones. On the fan at $v^{\pm}$ (resp $w^{\pm}$), $\tilde \phi_1$ takes the value $n$ on the primitive integral tangent vector to the edge $\ell^{\pm}$ (resp. $\lambda^{\pm}$) and zero on all other one dimensional cones. We let the reader check that $\tilde {\phi}_1$ is a well defined, MPL function with respect to the given decomposition.

To smooth the remaining two nodes, we further subdivide the polytopes $Q^{\pm}_j$ with $j=1,2$. We denote by $\check {\mathcal P}'$ the resulting subdivision. For $k=1$ or $3$, subdivide the polytopes $\conv(v_0,w_0, v_k, w_k, v^{\pm}, w^{\pm})$ as $$ \conv(v_0,w_0, v_k, w_k, v^{\pm}, w^{\pm})= \conv(v_0, w_0, w_k, w^{\pm}) \cup \conv(v_0, v_k, w_k, v^{\pm}, w^{\pm}). $$
This has the effect of subdividing the faces $\check e_j$, $j=1,2$, with the diagonal from $v_0$ to $w_j$ (see Figure~\ref{pos_nodes_2}). Now notice that near $w_0$ the subdivision looks precisely as the one defined in the proof of Theorem~\ref{resolve_nodes_1}, in case we wanted to resolve only the nodes $p_1', \ldots, p_r'$. So we can define a function $\tilde \phi_2$ to be zero outside the polytopes containing $w_0$ and just as in the proof of Theorem~\ref{resolve_nodes_1} on the polytopes which contain $w_0$.  

Finally, we define
\[ \check \phi' = N \check \phi + \tilde \phi_1 + \tilde \phi_2, \]
for some positive integer $N$. As in Theorem~\ref{resolve_nodes_1}, we can check that for sufficiently big $N$, $\check \phi'$ is strictly convex with respect to $\check {\mathcal P}'$. 
This completes the tropical smoothing of the given set of negative nodes. The discrete Legendre transform defines a tropical resolution of the mirror positive nodes. 
\end{proof}

We now prove that also some special configurations of negative nodes can be simultaneously resolved. 
Let $(L, \mathcal P_L)$ be a two-dimensional smooth tropical manifold with boundary, such that $L$ is homeomorphic to a disc, and satisfying the following assumptions.

\begin{ass} \label{negative:assump}
Denote by $v_1, \ldots, v_k$ the boundary vertices of $L$ and by $\lambda_1, \ldots, \lambda_k$ the boundary edges, such that $v_j$ and $v_{j+1}$ are the vertices of $\lambda_j$. We assume the following:
\begin{itemize}
\item[a)] at every vertex $v_j$ the affine structure is such that the edges $\lambda_j$ and $\lambda_{j-1}$ emanating from $v_j$ are colinear.
\item[b)] every vertex $v_j$ belongs to only one other edge of $L$ besides $\lambda_j$ and $\lambda_{j-1}$. Denote it by $\gamma_j$.
\end{itemize}
\end{ass}

Consider $L \times [0,1]$ with the product tropical structure and let $F_j = \lambda_j \times [0,1]$ be the boundary $2$-dimensional faces.  Suppose there is a $3$ dimensional tropical manifold $(B, \mathcal P, \phi)$ and an embedding of tropical manifolds $L \times [0,1] \hookrightarrow B$ such that every two face $F_j$ contains a negative node $p_j$ of $B$. In particular $F_j \cap \Delta$ is the union of the two segments joining the barycenters of opposite pairs of edges of $F_j$. Clearly the negative nodes $p_1, \ldots, p_k$ are related by the tropical $2$-cycle defined by $S= L \times \{ \frac{1}{2} \}$. The vector field $v$ in condition $(c)$ of Definition \ref{tropical_cycle} is parallel to edges of type $\{v_j \} \times [0,1]$, where $v_j$ is a vertex of $L$. Notice that the singular points of $L$ generate edges of $\Delta$ which intersect $S$ transversely in the interior, giving points as in conditions $(b)$ and $(i)$ of Definition \ref{tropical_cycle}. We have the following

\begin{thm} \label{resolve_-ve_nodes}
The negative nodes $p_1, \ldots, p_k$ in a configuration as above can be tropically resolved.
\end{thm}

\begin{proof}
The proof is similar to Theorem \ref{resolve_nodes_1}. Since we want to use mirror symmetry to resolve the nodes we start by describing the mirror. Observe that, if $P \in \mathcal P_L$ is any $j$-dimensional polytope of $L$ ($j=0,1,2$), then the mirror of the $(j+1)$-polytope $P \times [0,1]$ of $B$ will be a $(2-j)$-polytope $\check P$, inside $\check B$. 
All the polytopes $\check P$ are coplanar. We also have the polytopes $P \times \{1 \}$ and $P \times \{ 0\}$ whose mirrors we denote by $\check P^+$ and $\check P^-$ and their dimension is $3-j$. Clearly $\check P^+ \cap \check P^- = \check P$.
To indicate polytopes of $L$ we will use greek letters for edges ($\lambda, \delta, \gamma$...), lower case for vertices ($v, c, p, \ldots$) and upper case ($C, F, N, \ldots$) for $2$-faces. Every boundary edge $\lambda_j$ of $L$ belongs to a $2$-face $C_j$ of $L$. Moreover, it follows from point $(b)$ of Assumption \ref{negative:assump} that $C_j$ intersects $C_{j+1}$ in $\gamma_{j+1}$. Clearly the edge $\check \gamma_{j+1}$ joins the vertex $\check C_j$ to the vertex $\check C_{j+1}$. All edges $\check \gamma_j$ enclose a region, homeomorphic to a disc, which we denote by $\check L$.  Notice that if $P$ is a $2$-face of $L$, then $\check P^+$ and $\check P^-$ are edges emanating from $\check P$ which lie on a common line passing through $\check P$ and transversal to $\check L$. Moreover the edges of type $\check P^+$ (or $\check P^-$), with $\check P$ a vertex of $\check L$, are pairwise parallel with respect to parallel transport inside $\check L$. 

Given a boundary vertex $v_j$ of $L$, $\check v_j$ is a $2$-face which contains $\check \lambda_{j-1}$, $\check \gamma_j $ and $\check \lambda_{j}$ as bounding edges. Moreover $\check \lambda_{j} = \check v_j \cap \check v_{j+1}$. Point $(a)$ of Assumption \ref{negative:assump} implies that $\check \lambda_{j-1}$ and $\check \lambda_{j}$ are parallel edges of $\check v_j$. Notice that $\check \lambda_j$ is mirror to the face $F_j$, therefore the barycenter of $\check \lambda_j$ is a positive node mirror to the node $p_j$. Moreover the line inside $\check v_j$ going from the barycenter of $\check \lambda_{j-1}$ to the barycenter of $\check v_j$ and then to the barycenter of $\check \lambda_{j}$ is part of the discriminant locus $\Delta$. 

In order to resolve the nodes $p_1, \ldots, p_k$, we first define a new decomposition on $(\check B, \check{\mathcal P}, \check \phi)$ as follows. There is an open neighborhood $U$ of the region $\check L$ such that $U \cap \Delta$ consists only of small intervals containing the points $\check L \cap \Delta$. We can rescale $\check B$ so that all edges $\check \lambda_j$ and $\check C^{\pm}$, where $C$ is a $2$-face of $L$, contain at least one interior integral point inside $U$.  On each of the edges $\check \lambda_j$, $\check C_j^+$ and $\check C_j^-$ emanating from $\check C_j$ choose the integral points closest to $\check C_j$ and call them $t_j$, $q_j^+$ and $q_j^-$ respectively. 
Now subdivide $\check v_j^+$ in two convex polytopes, one being $W_j^+= \conv(t_{j-1}, t_{j}, \check C_{j-1}, \check C_{j}, q_{j-1}^+ , q_{j}^+)$ and the other one being the closure of its complement. Similarly subdivide $\check v_j^-$ and let $W_j^- = \conv(t_{j-1}, t_{j}, \check C_{j-1}, \check C_{j}, q_{j-1}^-, q_{j}^-)$. Given any interior vertex $v$ of $L$, we now give a subdivision of the $3$-polytope $\check v^+$. Notice that all vertices of $\check v$ are of the type $\check C$,  where $C$ is $2$-face of $L$ containing $v$. Let $q_C^+$ (resp. $q_C^-$) be the integral point on $\check C^+$ (resp. $\check C^-$) closest to $\check C$. Consider the polytope inside $\check v^+$ given by the convex hull of all points points $\check C$ and $q_C^+$ for all $2$-faces $C$ containing $v$. Denote it by $W_v^+$ and subdivide $\check v^+$ by $W_v^+$ and the closure of its complement. Similarly subdivide $\check v^-$. This gives a well defined decomposition of $\check B$. Notice that $W_v^+$, as a polytope, is just $\check v \times [0,1]$.

Now we need to define a fan structure at the new vertices of the subdivision. This definition must have the effect of smoothing the positive nodes. Although the fan structure at $\check C_j$ is unchanged, it is convenient to describe it.  Let $\{e_1, e_2, e_3 \}$ be an integral basis of $T_{\check C_j} \check B$ such that $e_3$ is tangent to $\check C_j^+$,  $e_2$ is tangent to $\check \lambda_j$ and $e_1$ is contained in the tangent wedge to $\check v_{j+1}$. There are pairs of integers $(a_j, b_j)$ and $(a_{j+1}, b_{j+1})$ such that the tangent wedges to $\check v_j^+$, $\check v_{j+1}^+$, $\check v_j^-$,  and $\check v_{j+1}^-$ correspond respectively to
\begin{eqnarray}
  \label{fan_cj} \cone( -a_je_1+b_je_2, e_2, e_3),  & \cone( a_{j+1}e_1+b_{j+1}e_2, e_2, e_3), &  \\
\nonumber \cone( -a_je_1+b_je_2, e_2, -e_3), & \cone( a_{j+1}e_1+b_{j+1}e_2, e_2, - e_3). &
\end{eqnarray}
Here  $ a_je_1+b_je_2$ and $-a_{j+1}e_1+b_{j+1}e_2$ are respectively the primitive generators of the cones corresponding to $\check \gamma_j$ and $\check \gamma_{j+1}$. Notice that $a_j$ and $a_{j+1}$ will be positive. We will not need the other cones of the fan structure at $\check C_j$.

Let us define the fan structure at $t_j$. There are eight maximal polytopes meeting at $t_j$: $W_j^+$, $W_{j+1}^+$, $W_j^-$, $W_{j+1}^-$ and their complements. Let  $\{f_1, f_2, f_3 \}$ be the standard basis of $\R^3$. We define the eight cones in the fan structure at $t_j$ to be 
\[ \cone( a_{j+1}f_1-b_{j+1}f_2, f_2, f_2 + f_3), \cone(-a_j f_1- b_j f_2,  f_2, f_2 + f_3), \]
 \[ \cone( a_{j+1}f_1-b_{j+1}f_2, f_2, - f_3),  \cone( -a_j f_1- b_j f_2,  f_2, - f_3), \]
\[ \cone( a_{j+1}f_1-b_{j+1}f_2, -f_2, f_2 + f_3), \cone( -a_j f_1- b_j f_2,  -f_2, f_2 + f_3), \]
 \[ \cone(  a_{j+1}f_1-b_{j+1}f_2, - f_2, - f_3),  \cone( -a_j f_1- b_j f_2,  - f_2, - f_3). \]
where the first four correspond respectively to the tangent wedges of $W_{j+1}^+$, $W_{j}^+$, $W_{j+1}^-$ and $W_{j}^-$. The other four to their complements. In the first cone, $f_2$ is tangent to the edge from $t_j$ to $\check C_j$ and therefore, by parallel transport inside $W^+_{j+1}$, $f_2$ is a parallel to $-e_2$. The primitive integral tangent vector to the edge from $t_j$ to $q_j^+$ corresponds, in the first cone, to $f_2 + f_3$. Notice that this implies that $f_3$, by parallel transport inside $W^+_{j+1}$, is parallel to $e_3$. The vector $a_{j+1}f_1-b_{j+1}f_2$ is tangent to the edge from $t_j$ to $t_{j+1}$ and therefore, by construction, it is parallel to $\check \gamma_{j+1}$ with respect to parallel transport inside $W_{j+1}^+$. Similarly $-a_j f_1- b_j f_2$, in the second cone, is tangent to the edge from $t_j$ to $t_{j-1}$ and it is therefore parallel to $\check \gamma_j$, with respect to parallel transport inside $W_{j}^+$. Notice the crucial fact that in the third cone $-f_3$ is tangent to the edge from $t_j$ to $q_j^-$ and therefore, by parallel transport inside $W_{j+1}^-$, it is parallel to $-e_3-e_2$.  This choice of fan structure guarantees that parallel transport along a loop going from $\check C_j$ to $t_j$ passing inside $W_{j+1}^+$ and then back to $\check C_j$ passing into $W_{j}^+$ is the identity. While moving along a loop going from $\check C_j$ to $t_j$ passing inside $W_{j+1}^+$ and then back to $\check C_j$ passing into $W_{j+1}^-$ gives the monodromy matrix
 \[
\left(
\begin{array}{ccc}
  1 & 0  & 0  \\
  0 &  1 &  1 \\
  0 &   0&   1
\end{array}
\right).
\]
computed with respect to the basis $\{ e_1, e_2, e_3 \}$. This corresponds to the smoothing of the positive node (compare with \S \ref{trop_res_-ve}, in particular with the second construction of the smoothing).

We now define the fan structure at the points $q^+_j$. The tangent wedges to $W_{j}^+$, $W_{j+1}^+$ are mapped respectively to the cones: 
\[ \cone( -a_jf_1+b_jf_2, f_2+f_3, f_3), \cone( a_{j+1}f_1+b_{j+1}f_2, f_2+ f_3, f_3). \]
Here $f_3$ is tangent to the edge from $q^+_j$ to $\check C_j$, $f_2+f_3$ is tangent to the edge $q^+_j$ to $t_j$. In the first cone, 
$-a_jf_1+b_jf_2$ is tangent to the edge from  $q^+_j$ to $q^+_{j-1}$. In the second cone $a_{j+1}f_1+b_{j+1}f_2$ is tangent to the edge from  $q^+_j$ to $q^+_{j+1}$. The complements of $W_{j}^+$ and $W_{j+1}^+$ inside $\check v_j^+$ and $\check v_{j+1}^+$ are mapped respectively to,
\[ \cone(-a_jf_1+b_jf_2, f_2+f_3, -f_3), \cone( a_{j+1}f_1+b_{j+1}f_2, f_2+ f_3, -f_3). \]
The other polytopes meeting in $q^+_j$ come from the subdivision of $\check v^+$, where $\check v$ is a $2$-face of $\check L$, containing the vertex $\check C_j$. Namely we have $W_v^+$ and its complement inside $\check v^+$. Suppose that in the fan structure at $\check C_j$, the tangent wedge of $\check v$ at $\check C_j$ is  $\cone(\alpha_1 e_1 + \alpha_2 e_2, \beta_1 e_1 + \beta_2 e_2)$. Then, at $q^+_j$, the tangent wedges of $W_v^+$ and its complement are mapped respectively to the cones 
\[ \cone( \alpha_1 f_1 + \alpha_2 f_2, \beta_1 f_1 + \beta_2 f_2, f_3 ), \cone( \alpha_1 f_1 + \alpha_2 f_2, \beta_1 f_1 + \beta_2 f_2, -f_3). \] This defines the fan structure at $q^+_j$.  Similarly we define the fan structure at $q^-_j$.
The fan structure at points of type $q^+_C$ or $q^-_C$, where $\check C$ is an interior vertex of $\check L$ is defined similarly and we leave its definition to the reader. It can be verified that these fan structures are compatible with the fan structures at all other points which are unchanged. Moreover this construction produces a smoothing of the positive nodes mirror to $p_1, \ldots, p_k$. In order to complete the resolution of $p_1, \ldots, p_k$ we need to find a suitable multivalued strictly convex piecewise linear $\check \phi'$. The strategy, as in the proof of Theorem \ref{resolve_nodes_1}, is to define a suitable $\tilde \phi$ which is ``supported in a neighborhood of $\check L$''. Then define $\check \phi'$ as in (\ref{new_function}) and prove that it is strictly convex for large $N$. Here is how we define $\tilde \phi$. Let $d$ be a positive integer. On the fan at $t_j$, let $\tilde \phi(f_2) = d$ and zero at all other primitive edges.  On the fan at $\check C_j$ let $\tilde \phi(e_2) = \tilde \phi(e_3) = -d$ and zero at all other primitive edges. On the fan at $q_j^+$ and $q_j^-$, let $\tilde \phi(f_3) = d$ and zero at all other primitive edges. We claim that for some choice of $d$, $\tilde \phi$ is a well defined multivalued piecewise linear function. Let us first explain how to find $d$. Notice that if we take $d=1$, then $\tilde \phi$ may not be integral. For instance on the first cone of the list (\ref{fan_cj}), we would get $\tilde \phi(e_1)= -b_j/a_j$ which may not be an integer. Similarly on the second cone $\tilde \phi(e_1)= b_{j+1}/a_{j+1}$. One can check that if we let $d$ be a common multiple of $a_j$ and $a_{j+1}$ for all $j=1, \ldots, k$, then $\tilde \phi$ will be integral. 

Verifying that $\tilde \phi$ is well defined is a tedious calculation similar to the one carried out in Theorem \ref{resolve_nodes_1} and we therefore omit it. Also the argument to show that the function $\check \phi'$ defined in (\ref{new_function}) is strictly convex for large $N$ is the same. 
\end{proof}

\section{Examples} \label{2exampls} 
We discuss some examples of compact tropical conifolds where various sets of nodes can be simultaneously resolved/smoothed. In the first example (and its mirror) we have slightly modified an example discussed in \S 4 of \cite{Gross_Batirev} so that it contains $9$ nodes. Resolving/smoothing subsets of these nodes produces new examples of compact tropical manifolds and their mirrors. We know that one of these is associated to a toric degeneration of a complete intersection in a toric Fano manifold (see discussion in \S \ref{global_res_sm}). We are not sure about the other ones. In the last paragraph we further generalize these examples. 

\subsection{First example...} \label{ex_one}
As the set $\mathcal P$ of polytopes we take $18$ copies of a triangular prism, i.e. of 
\begin{equation} \label{prismT}
 T =  \conv \{ (0,0,0), (4,0,0), (0,4,0), (0,0,4), (4,0,4), (0,4,4) \}. 
\end{equation}
We divide $\mathcal P$ in two families of nine prisms each and label prisms in each family by $\sigma_{jk}$ and $\tau_{jk}$ respectively, where $j,k$ are cyclic indices from $1$ to $3$. The vertices of $\sigma_{j-1,k-1}$ and $\tau_{j-1,k-1}$ are labeled like in Figure~\ref{prism1}.
 
\begin{figure}[!ht] 
\begin{center}
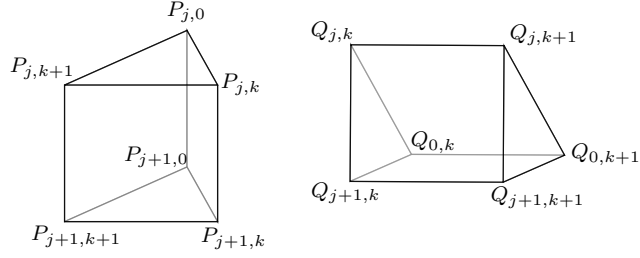
\caption{The two families of prisms: $\sigma_{j-1,k-1}$ on the left and $\tau_{j-1,k-1}$ on the right.}\label{prism1}
\end{center}
\end{figure}

Notice that some polytopes will have vertices with the same label. We glue polytopes along $2$-dimensional faces by matching vertices with the same label. For instance, Figure~\ref{prism3} shows the result of gluing $\tau_{11}$, $\tau_{21}$, $\tau_{31}$. The result of identifying polytopes with this rule gives two (polyhedral) solid tori, one from each family. Now glue the boundaries of these solid tori along $2$-dimensional faces by matching the vertex $Q_{jk}$ with the vertex $P_{jk}$, for every $j,k=1,2,3$. This gives the manifold $B$, homeomorphic to a $3$-sphere, and the polyhedral decomposition $\mathcal P$.
\begin{figure}[!ht] 
\begin{center}
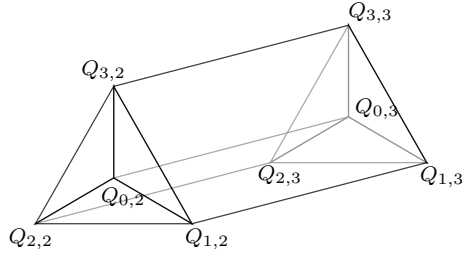
\caption{Assembling $\tau_{11}$, $\tau_{21}$ ,$\tau_{31}$.}\label{prism3}
\end{center}
\end{figure}
 We now describe the fan structure at vertices. There are two types of vertices, those labeled $P_{k,0}$ (or $Q_{0,k}$), lying in the interior of the solid tori, and those labeled $P_{j,k}$ (which are the same as $Q_{j,k}$) lying on the boundary of the solid tori. We take as representatives of each type, the vertices $P_{3,0}$ and $P_{3,1}$ respectively, and describe the fan structure there. Vertices of the same type will have the same fan structures, in the obvious sense. 

For the vertex $P_{3,0}$, take the fan corresponding to $\PP^2 \times \PP^1$, whose three dimensional cones are 
\[ \cone(e_1,e_2, e_3), \ \cone(e_1,-e_1-e_2, e_3), \ \cone(e_2,-e_1-e_2, e_3 ),\]
\[ \cone(e_1,e_2, -e_3), \ \cone (e_1,-e_1-e_2, -e_3), \  \cone(e_2,-e_1-e_2, -e_3). \]
The vertex $P_{3,0}$ belongs to six $3$-dimensional faces of $B$. Three of these are $\sigma_{11}$, $\sigma_{12}$, $\sigma_{13}$, which intersect on the edge from $P_{3,0}$ to $P_{2,0}$. Identify the tangent wedges at $P_{3,0}$ of these three polytopes with the first three cones, in such a way that the tangent direction to the edge from $P_{3,0}$ to $P_{2,0}$ is mapped to $e_3$. 
The other three $3$-dimensional faces containing $P_{3,0}$ are $\sigma_{21}, \sigma_{22}, \sigma_{23}$, which intersect on the edge from $P_{3,0}$ to $P_{1,0}$. Identify the tangent wedges at $P_{3,0}$ of these three polytopes with the last three cones, in such a way that the tangent direction to the edge from $P_{3,0}$ to $P_{1,0}$ is mapped to $-e_3$.

For the vertex $P_{3,1}$, take the fan corresponding to $\PP^1 \times \PP^1 \times \PP^1$, whose $3$-dimensional cones are the octants of $\R^3$.  Notice that there are eight $3$-dimensional faces containing $P_{3,1}$: $\sigma_{23}$, $\sigma_{22}$, $\sigma_{12}$, $\sigma_{13}$, $\tau_{23}$, $\tau_{22}$, $\tau_{12}$ and $\tau_{13}$. The fan structure at  $P_{3,1}$ identifies the tangent wedges of these eight polytopes respectively with: $\cone(-e_1, - e_2,  e_3)$, $\cone( e_1, - e_2,  e_3 )$, \ \ $\cone(e_1, e_2,  e_3)$, \ \ $\cone( -e_1, e_2, e_3 )$, \ \ $\cone(-e_1, - e_2,  -e_3)$, $\cone( e_1,  - e_2,  -e_3 )$,  \ $\cone(e_1, e_2,  -e_3)$ and $\cone( -e_1, e_2, -e_3 )$. This determines the fan structure at $P_{3,1}$ and similarly for each vertex of the same type.

\begin{figure}[!ht] 
\begin{center}
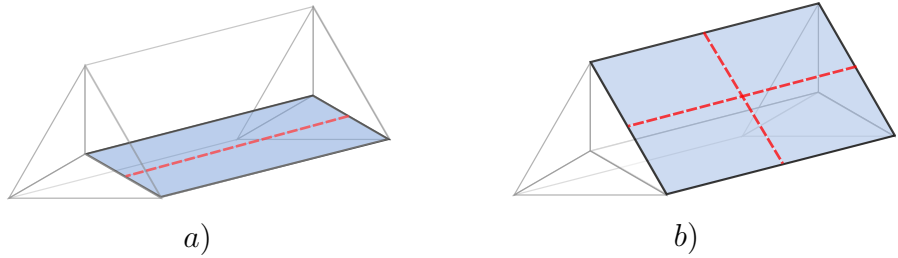
\caption{The discriminant locus $\Delta$ (dashed lines)}\label{monodromy}
\end{center}
\end{figure}

The discriminant locus $\Delta$, also depicted in Figure \ref{monodromy}, can be described as follows. Inside square faces with vertices $P_{j,0}, P_{j+1,0}, P_{j,k}$ and $P_{j+1,k}$ (or $Q_{0,k}, Q_{0,k+1}, Q_{j,k}$ and $Q_{j,k+1}$), $\Delta$ consists only of the segment joining the barycenter of the edge from $P_{j,0}$ to $P_{j,k}$  (resp. from $Q_{0,k}$ to $Q_{j,k}$) to the barycenter of the edge from $P_{j+1,0}$ to $P_{j+1,k}$ (resp. from $Q_{0,k+1}$ to $Q_{j,k+1}$), see Figure \ref{monodromy} (a). Monodromy around this component of $\Delta$ is given by the matrix 
\begin{equation} \label{mon1_ex} 
\left( \begin{array}{ccc}
1 & 3 & 0 \\ 
0 & 1 & 0 \\ 
0 & 0 & 1
\end{array} \right) 
\end{equation} 
Observe that all these segments together give six disjoint circles, three in the interior of each solid torus. Moreover each circle has multiplicity $3$, as can be seen from the monodromy. It can be shown that the structure can be modified slightly so that each circle splits into three, each one with the monodromy of generic-singularities (see for instance \S 4 of \cite{Gross_Batirev}). We will ignore this issue here.

Inside square faces with vertices $P_{j,k}$, $P_{j,k+1}$, $P_{j+1,k+1}$ and $P_{j+1,k}$, $\Delta$ has a quadrivalent vertex (see Figure \ref{monodromy} $(b)$). In fact it can be checked that this vertex is a negative node. Overall there are $9$ negative nodes. 
There are no other components of $\Delta$.

Finally we define a strictly convex piecewise linear function $\phi$ on $(B, \mathcal P)$. It is enough to specify the function on the fans associated to each vertex. At vertices of type $P_{k,0}$ (resp. $Q_{0,k}$) the fan is the fan of $\PP^2 \times \PP^1$. We define $\phi$ to have value $1$ at $-e_1 -e_2$ and at $e_3$, and zero at the remaining ones. At vertices of type $P_{j,k} = Q_{j,k}$, the fan is the fan of $\PP^1 \times \PP^1 \times \PP^1$. Here we take $\phi$ to have value $1$ at $-e_1$, $-e_2$ and $-e_3$ and zero at the remaining ones.  One can check that $(B, \mathcal P, \phi)$ is a well defined tropical conifold with $9$ negative nodes.

\subsection{...and its mirror} The \ discrete \ Legendre \ transform of the \ previous example gives \ its mirror \ \ $(\check B, \  \check{\mathcal P})$.  \ The polytopes \ dual to the vertices \ $P_{j,0}$ or \ $Q_{0,k}$ \ are six triangular prisms (i.e. {\small $\conv \{(0,0,0), (0,1,0), (1,0,0), (0,0,1), (0,1,1), (1,0,1)\}$}), which we label $\sigma_j$ and $\tau_k$ respectively. The polytopes dual to the other $9$ vertices are cubes (i.e. {\small $\conv \{(0,0,0), \- (0,1,0), (1,0,0), \- (1,1,0), \- (0,0,1), \- (1,0,1), \-  (0,1,1),  (1,1,1) \}$}). Let us denote the nine cubes by $\omega_{jk}$. We label the vertices of these polytopes by the letters $E_{jk}$ and $F_{jk}$ as in Figure~\ref{prism4}. 
\begin{figure}[!ht] 
\begin{center}
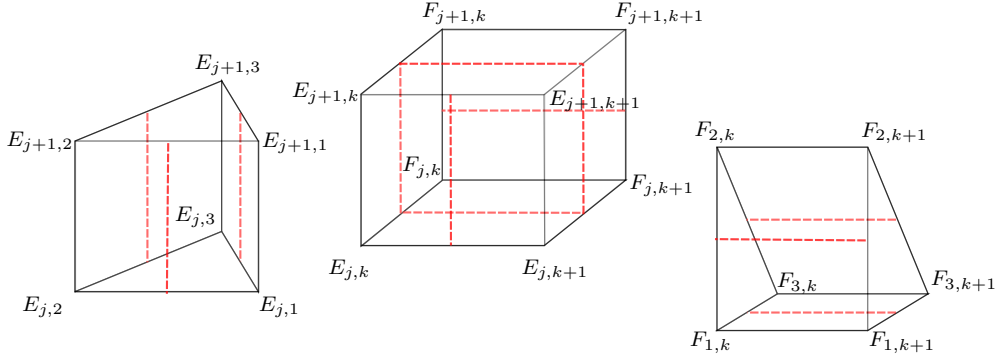
\caption{The polytopes $\sigma_j$ (left), $\omega_{jk}$ (center) and $\tau_k$ (right). The dashed red lines form the discriminant locus $\Delta$  }\label{prism4}
\end{center}
\end{figure}
The $2$-dimensional faces of these polytopes are glued by matching the vertices with the same labeling. Notice that by first gluing together $\sigma_1, \sigma_2$ and $\sigma_3$ on the one hand and $\tau_1, \tau_2$ and $\tau_3$ on the other, we obtain a pair of solid tori. Then, in Figure~\ref{prism4}, the front faces of the cubes $\omega_{jk}$ are glued to the square faces of $\sigma_j$ and the back faces are glued to the square faces of $\tau_k$. The final result is again a $3$-sphere. 

We now need to specify the fan structure at vertices. The fan at every vertex is the fan of $\PP^2 \times \PP^1$. In the fan structure at $E_{j,k}$ the primitive tangent vectors to the five edges going to $E_{j,k+1}$, to $E_{j,k-1}$, to $F_{j,k}$, to $E_{j+1,k}$ and to $E_{j-1,k}$ are mapped respectively to $e_1$, $e_2$, $-e_1-e_2$, $e_3$ and $-e_3$. Similarly (and symmetrically) we have the fan structure at the vertex $F_{k,j}$. By inspection one can see that the discriminant locus $\Delta$ is as depicted in Figure~\ref{prism4}. In fact monodromy around the edges of $\Delta$ contained in the union of the $\sigma_j$'s (or in the union of the $\tau_j$'s) is conjugate to (\ref{mon1_ex}). The remaining edges of $\Delta$, those which do not intersect the triangular prisms, have standard generic-singular monodromy. Notice that each edge going from $E_{j,k}$ to $F_{j,k}$ contains a positive node. So that $(\check B, \check{\mathcal P})$ has $9$ positive nodes, which are mirror to the $9$ negative ones in $(B, \mathcal P)$. We also have a strictly convex piecewise linear function $\check \phi$.

\subsection{Resolving/smoothing nodes} \label{global_res_sm}
We now apply the results of Section \ref{simult_res_sm} to simultaneously resolve and smooth certain sets of nodes in these two examples. We can cut each cube $\omega_{jk}$ with the unique plane passing through the four nodes. This gives us a tropical $2$-cycle $S$ containing the nodes. Thus every quadruple of nodes contained in a cube is in a configuration where we can apply Theorem~\ref{resolve_nodes_1}, thus it can be resolved. In fact any subset of the $9$ nodes which is a union of such quadruples can be resolved with the criteria of Section \ref{simult_res_sm}. For instance the six nodes inside two cubes sharing a common face (e.g. $\omega_{j,k}$ and $\omega_{j,k+1})$), are in a configuration like in Theorem~\ref{resolve_nodes_2}. Similarly the seven nodes inside a pair of cubes sharing an edge can be resolved using Corollary \ref{resolve_nodes_cor}. We can also resolve all $9$ nodes, by observing that we can find tropical $2$-cycles $S_1$, $S_2$ and $S_3$ which are the squares obtained by cutting the cubes $\omega_{11}$, $\omega_{22}$ and $\omega_{33}$ (with suitable choices of orientations and vector fields). The pairwise intersection of the $S_j$'s is just a node, therefore Corollary~\ref{resolve_nodes_cor} applies. The last case is given by $8$ nodes. For instance, take the $8$ nodes inside $\omega_{11}$, $\omega_{22}$ and $\omega_{23}$. Then $\omega_{11}$ shares one node with $\omega_{22}$ and one with $\omega_{23}$, forming configurations as in Corollary \ref{resolve_nodes_cor}. The nodes in $\omega_{22}$ and $\omega_{23}$ are as in Theorem \ref{resolve_nodes_2}. 

Let us now discuss the smoothing of the nodes in $\check B$. By mirror symmetry, smoothing nodes corresponds to resolving the mirror ones. So let us look at $(B, \mathcal P)$. Consider for instance the polytopes $\tau_{11}, \tau_{21}, \tau_{31}$ as in Figure \ref{prism3}. Then the three square faces of these polytopes which do not contain the vertices $Q_{0,2}$ and $Q_{0,3}$ have a negative node in their barycenter (see also Figure \ref{monodromy} (b)). 

\begin{figure}[!ht] 
\begin{center}
\includegraphics{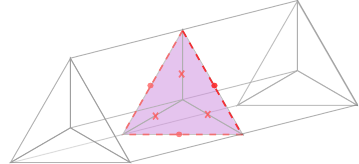}
\caption{The tropical $2$-cycle which relates $3$ nodes in $B$.}\label{sim_smooth}
\end{center}
\end{figure}

We now construct a tropical $2$-cycle containing these three nodes (see Figure \ref{sim_smooth}). Inside the triangular prism $T$, given by (\ref{prismT}), consider the triangle $T \cap \{ z=2\}$, obtained by cutting $T$ in the middle by a plane parallel to the triangular faces. Then define the tropical $2$-cycle $S$ as the union of the copies of this triangle contained in $\tau_{1k}, \tau_{2k}$ and $\tau_{3k}$. As the vector field $v$ of Definition \ref{tropical_cycle} we take the parallel transport along $S$ of the tangent vector to the edge from $Q_{0,2}$ to $Q_{0,3}$. Notice that $S \cap \Delta$ consists of the boundary of $S$ and three other points in the interior. It can be easily verified that the latter points satisfy condition $(i)$ of Definition \ref{tropical_cycle}. Also all other conditions of Definition \ref{tropical_cycle} are satisfied. It can be also verified that Theorem \ref{resolve_-ve_nodes} applies to this configuration. 
The same thing can of course be said about the three nodes in the union of $\sigma_{j1}, \sigma_{j2}$ and $\sigma_{j3}$.
Thus these configurations of nodes can be resolved and the corresponding mirror nodes in $\check B$ can be smoothed.

Inside $B$, let us consider the nine square faces containing the nine nodes. These form the boundary of the solid torus formed by the polytopes $\tau_{jk}$ (or $\sigma_{jk}$). 
Represent the boundary as a big square subdivided in nine small ones. Then Figure \ref{re_sm} represents the various possibilities we have of resolving and smoothing nodes in order to obtain a smooth tropical manifold. 
The families of nodes which can be resolved are those which lie on the horizontal  or vertical lines of the grid and these are represented by nodes which are circled (in blue). 
The nodes which are being smoothed are the ones whose faces are subdivided by (blue) diagonal lines. In fact these lines represent the subdivision given by Theorem \ref{resolve_nodes_2}. 
As we can see we have $4$ cases: we can resolve all nine nodes; smooth 4 and resolve 5; smooth 6 and resolve 3 or smooth all 9 nodes. Thus we have four different, smooth tropical manifolds. The Gross-Siebert reconstruction theorem gives four different Calabi-Yau manifolds. 
Let us denote them respectively by $X_0$, $X_4$, $X_6$ and $X_9$.  Now let us consider the mirror manifolds $\check X_0$, $\check X_4$, $\check X_6$ and $\check X_9$. 
These are obtained from $\check B$ respectively by smoothing all nine nodes; resolving $4$ and smoothing $5$; resolving $6$ and smoothing $3$ or resolving all $9$ nodes. We know that $\check X_0$ corresponds to an example of Schoen's Calabi-Yau \cite{Schoen}. This can be described as follows. Let $f_1: Y_1 \rightarrow \PP^1$ and $f_2: Y_2 \rightarrow \PP^1$ be two rational elliptic surfaces with a section such that for no point $x \in \PP^1$, $f_1^{-1}(x)$ and $f_2^{-1}(x)$ are both singular. Then Schoen's Calabi-Yau is the fibred product $Y_1 \times_{\PP^1} Y_2$. It was proved in \cite{schoen_ms} that a family of Calabi-Yaus of this sort can be represented as a complete intersection inside $\PP^1 \times \PP^2 \times \PP^2$ of hypersurfaces of tridegree $(1,3,0)$ and $(1,0,3)$. Later Gross showed in \cite{Gross_Batirev} that the associated tropical manifold is precisely $\check B$ with all nine nodes smoothed (see also \cite{Haase-ZharkovIII} for similar methods). So we expect that $\check X_0$ is homeomorphic to Schoen's Calabi-Yau (see  Theorem 0.1 in \cite{Gross_Batirev}). Thus we have $b_2(\check X_0) = 19$, $b_3(\check X_0) = 40$ and $\chi(\check X_0) = 0$. Now $\check X_4$, $\check X_6$ and $\check X_9$ are related to $\check X_0$ by a conifold transition at respectively $4$, $6$ and $9$ Lagrangian spheres in $\check X_0$. Therefore, applying (\ref{contran_top}), we obtain $\chi(\check X_4) = 8$, $\chi(\check X_6) = 12$ and $\chi(\check X_9) = 18$. 
\begin{figure}[!ht] 
\begin{center}
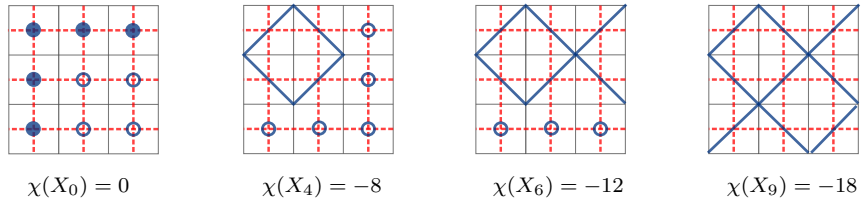
\caption{Smoothings and resolutions: the thicker (blue) lines are the subdivisions required by the smoothing. The (blue) circles indicate the nodes being resolved. The darker dots mark the nodes whose vanishing cycles in $\check X_0$ from a basis of the space spanned by the vanishing cycles of all $9$ nodes.} \label{re_sm}
\end{center}
\end{figure}
Applying the method described in \cite{Gross_Batirev}, we have computed that the tropical manifold obtained from $B$ by smoothing all nodes corresponds to a complete intersection in $\PP^3 \times \PP^3$ of polynomials of bidegree $(3,0)$, $(0,3)$ and $(1,1)$. Therefore we expect $X_9$ to be homeomorphic to such a manifold. We have $b_2(X_9) = 14$, $b_3(X_9) = 48$ (see also \cite{tian-lu}). Obviously its mirror $\check X_9$ has $b_2(\check X_9) = 23$ and $b_3(\check X_9) = 30$. Since $\check X_9$ is related to $\check X_0$ by a conifold transition at $9$ Lagrangian spheres in $\check X_0$,  from (\ref{contran_top}) we obtain $c=5$ and $d=4$. Therefore the vanishing cycles span a space of dimension $5$ in $H_3(\check X_0)$.  We believe, although we have not proved it, that five linearly independent Lagrangian spheres correspond to the nodes (in the mirror $\check B$) which are marked by dark dots in the grid of Figure \ref{re_sm}. Indeed, in the previous discussion, we observed that four nodes contained in a square (of the dashed grid in Figure \ref{re_sm}) are related by one relation. Observe that all $9$ vanishing cycles can be obtained from the given $5$ using these relations. In particular, in the case of the four nodes on a square, we expect $c=3$, $d=1$. In the case of six nodes on two adjacent squares we expect $c=4$, $d=2$, where the $4$ linearly independent spheres correspond to the dark dots in the first two rows of the grid in Figure \ref{re_sm}. 
Therefore, using (\ref{contran_top}), we conjecture that 
\[ b_2(\check X_4) = b_2(\check X_0) + 1 = 20, \ \ \ b_3(\check X_4) = b_3(\check X_0) - 6 = 34 \]
\[ b_2(\check X_6) = 21, \ \ \ \  b_3(\check X_6) = 32 \]
Moreover we also have the mirrors $X_0$, $X_4$, $X_6$ and $X_9$. 

It is likely that $\check X_4$, $\check X_6$ and $\check X_9$ can be obtained as conifold transitions by classical methods from the equations defining $\check X_0$. However we do not know if the corresponding tropical manifolds can be obtained from known toric degenerations, such as in toric Fano manifolds. Moreover we do not know if the mirrors have ever been computed by more standard methods. 

\subsection{More examples} \label{more_ex}
Here we generalize the above example.  For every pair of integers $(L, M)$, ranging from $3$ to $9$, we construct a tropical conifold as follows. Take $2LM$ copies of the triangular prism $T$ considered in (\ref{prismT}) and divide them in two families each containing $LM$ copies. We denote the two families by $\sigma_{jk}$ and $\tau_{jk}$ where $j,k$ are cyclic indices of order $L$ and $M$ respectively. Now, to form $B$ we do the same as above: we label the vertices of these prisms like in Figure \ref{prism1} and we glue the $2$-dimensional faces by matching the vertices with the same labels. Clearly, for $L=M=3$ we obtain the same as above. Observe that assembling $\tau_{11}, \tau_{21}, \ldots, \tau_{L1}$ with the above rule looks like the pictures represented in Figure \ref{triangles} multiplied by the interval $[0,4]$. Similarly we can say about $\sigma_{11}, \sigma_{12}, \ldots, \sigma_{1M}$. As above, the union of all the $\sigma_{jk}$'s on the one hand and of all the $\tau_{jk}$'s on the other gives two solid tori which are again glued together to form a $3$-sphere. 

\begin{figure}[!ht] 
\begin{center}
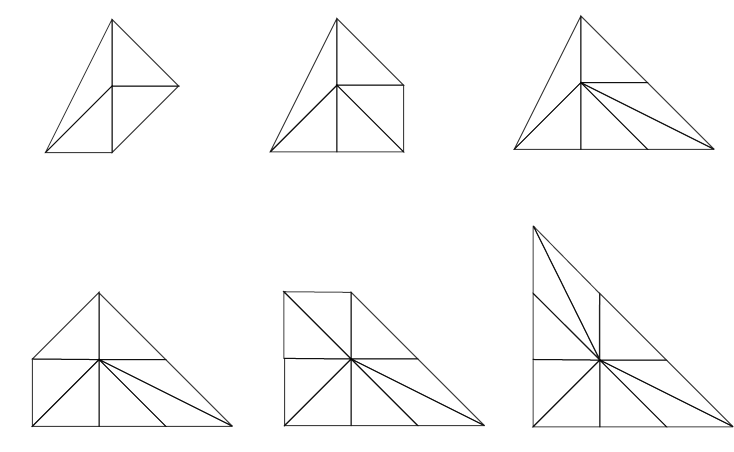
\caption{} \label{triangles}
\end{center}
\end{figure}

The fan structures at vertices are defined in a very similar way to the example above. In fact at points of type $P_{j,k}$, $k \neq 0$, the fan structure is exactly the same as in the example above, i.e. it is the fan of $\PP^1 \times \PP^1 \times \PP^1$. At a point of type $Q_{0,k}$ (or $P_{j,0}$) it is also like in the example above if $L=3$ (resp. if $M=3$). Otherwise we define it as follows. For integers $L=4,\ldots,9$, consider the fan $\Sigma_L$ in $\R^2$ which can be seen at the common point (i.e. the origin) of all simplices in Figure \ref{triangles}. The fan $\Sigma_{L+1}$ is the toric blow up of the fan $\Sigma_L$. Clearly $\Sigma_L$ has $L$ $2$-dimensional cones. Denote them in clockwise order by $C_1, \ldots, C_L$, where $C_1$ is the one generated by $\{ (-1,-1), (0,1) \}$. Then form the fan in $\R^3$ whose cones are $C_j^+ = C_j \times [0,+ \infty)$ and $C_j^- = C_j\times (-\infty,0]$. Now, the $3$-dimensional polytopes which contain the point $Q_{0,k}$ are $\tau_{j,k-1}$ and $\tau_{j, k-2}$, with $j=1, \ldots, L$. The fan structure at $Q_{0,k}$ identifies the tangent wedge of $\tau_{j,k-1}$ with $C_{j}^{+}$ and of $\tau_{j,k-2}$ with $C_{j}^{-}$. Similarly we can define the fan structure at points of type $P_{j,0}$ but with $M$ in place of $L$ and the roles of $j$ and $k$ inverted.
 Then we can also define a strictly convex piecewise linear function $\phi$, just by suitably choosing one on each of the two types of fans. This defines our tropical conifold $(B, \mathcal P, \phi)$, depending on the choice of integers $L,M = 3,\ldots, 9$. Notice that we cannot go beyond $9$ in this construction, because the polytopes in Figure \ref{triangles} would lose convexity and the tropical manifold would not be smooth in the sense of \S \ref{GSreconstruction}. Discrete Legendre transform gives the mirror $(\check B, \check{\mathcal P}, \check \phi)$. Notice that again $B$ has a negative node on every square face that does not contain a point of type $P_{j,0}$ or $Q_{0,k}$, i.e. square faces that are on the boundary of the solid tori. Therefore, there are $LM$ negative nodes. 

Notice that in  $(\check B, \check{\mathcal P}, \check \phi)$, the polytope mirror to a point $P_{j,k}$, $k \neq 0$, is a cube. Just as in the above example, this cube contains $4$ positive nodes which are related in the sense of Definition \ref{pos_nodes_rel} and can be simultaneously resolved using Theorem \ref{resolve_nodes_1}. The negative nodes in $B$ which are mirror to these $4$ nodes are those contained in the prisms $\sigma_{j-1,k-1}$, $\sigma_{j-1,k-2}$, $\sigma_{j-2,k-1}$, $\sigma_{j-2,k-2}$. These can be simultaneously smoothed.  Notice that for every fixed $j$ (or $k$) the $M$ (resp. $L$) nodes contained in the prisms $\sigma_{j,1}, \ldots, \sigma_{j,M}$ (resp. $\sigma_{1,k}, \ldots, \sigma_{L,k}$) are also related (see the example above), therefore they can be simultaneously resolved.  Depending on the choices of nodes to be simultaneously smoothed/resolved we get diagrams similar to those in Figure \ref{re_sm}, but on an $L \times M$ grid.  For instance, let us consider the case where $L=3$ and $M=4, \ldots, 9$ and in $B$ we smooth $2M$ nodes and resolve the remaining $M$. Denote by  $X_{2M}$ the corresponding Calabi-Yau. 
In Figure \ref{re_sm2} we have represented the diagrams for $L=3$ and $M=4, 5, 6$. We believe that by smoothing all nodes in $\check B$ we still get Schoen's Calabi-Yau. Therefore, with a similar argument as above, we conjecture that $b_2(X_{2M}) = 18-M$, $b_3(X_{2M})= 38+2M$.
 
\begin{figure}[!ht] 
\begin{center}
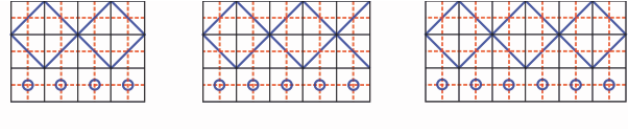
\caption{} \label{re_sm2}
\end{center}
\end{figure}

\subsection{Tropical cycles and good relations.} \label{example_good_rel}
Here we take a closer look at the structure of good relations among the vanishing cycles (or exceptional $\PP^1$'s) associated to the nodes in the above examples. Let $B$ be the tropical conifold constructed in \S \ref{more_ex} with $LM$ negative nodes, where $L$ and $M$ are fixed integers between $3$ and $9$. We can resolve all $LM$ nodes, obtaining the Calabi-Yau we denoted $X_0$ or smooth all nodes, obtaining $X_{LM}$. We also have the mirrors $\check X_0$ and $\check X_{LM}$, obtained from $\check B$, where $\check X_0$ is given by smoothing all (positive) nodes and $\check X_{LM}$ by resolving them. In the case $L=M=3$ we argued that the vanishing cycles in $\check X_0$ span a lattice of dimension $5$ in $H_3(\check X_0, \Z)$ and we conjectured that a basis of this lattice is given by the vanishing cycles associated to the nodes depicted in dark dots in Figure \ref{re_sm}. On the mirror side, it is the exceptional $\PP^1$'s of $X_0$ which span a lattice of rank $5$ in $H_2( X_0, \Z)$ and we expect this lattice to be generated by the $\PP^1$'s over the same dark nodes of Figure \ref{re_sm}. On the other hand, by formulas (\ref{contran_top}), the vanishing cycles in $X_9$ must span a lattice of dimension $4$ in $H_3(X_9, \Z)$. We conjecture that a basis for this lattice is given by the vanishing cycles over the complement of the dark nodes in Figure \ref{re_sm}. This is again reasonable, since all other nodes can be obtained from these four using the relation given by tropical $2$-cycles bounding horizontal and vertical lines in the grid (pictured in Figure \ref{sim_smooth}).

We now generalize to any pair $(L,M)$. Label the nodes in $\check B$ by $N_{j,k}$, where $j,k$ are cyclic indices of order $L$ and $M$ respectively and we assume they are displaced on the $L \times M$ grid so that $j$ denotes the row and $k$ the column.
 By slight abuse of notation, $N_{j,k}$ denotes also the homology class of the vanishing cycle in $\check X_0$ associated to the node. Let us denote by $S_{j,k}$ the tropical $2$-cycle in $\check B$ whose corners are the nodes $N_{j,k}$, $N_{j+1,k}$, $N_{j,k+1}$, $N_{j+1,k+1}$.
 We conjecture that the lattice spanned by the $N_{j,k}$'s in $H_3(\check X_0, \Z)$ has rank $L+M-1$ and that a basis is given by $\{ N_{1,1}, \ldots, N_{L,1}, N_{1,2}, \ldots, N_{1,M} \}$, i.e. by the first row and first column. 

To understand the good relations induced by the tropical $2$-cycles, we need to discuss orientations. 
Given the tropical cycle $S_{j,k}$, an orientation on the lifts $\tilde S_{j,k}$ or $\tilde S^*_{j,k}$ constructed in Theorem \ref{good_rel_pos} is given by the vector field $v$ and the choice of an orientation on $S_{j,k}$. 
Since all the $S_{j,k}$'s lie on a two dimensional submanifold of $\check B$ (a two torus), we can choose the orientation on the $S_{j,k}$'s to coincide with a fixed orientation of this submanifold. We also assume that all vector fields $v$ on the $S_{j,k}$'s point in the same direction when they meet at the nodes.  
This fixes the orientations on the lifts $\tilde S_{j,k}$ and $\tilde S^*_{j,k}$. Now consider a node $N_{j,k}$. It is a corner of $S_{j,k}$, $S_{j-1,k}$, $S_{j,k-1}$, $S_{j-1,k-1}$. One can prove that the orientation induced on the vanishing cycle at $N_{j,k}$ (resp. exceptional $\PP^1$) by $\tilde S_{j,k}$ (resp. $\tilde S^*_{j,k}$) is the same as the orientation induced by $\tilde S_{j-1,k-1}$ (resp.  $\tilde S^*_{j-1,k-1}$) and the opposite of the one induced by $\tilde S_{j-1,k}$ and $\tilde S_{j,k-1}$ (resp. $\tilde S^*_{j-1,k}$ and $\tilde S^*_{j,k-1}$). 
This is the motivation behind Definition \ref{coeff_eps}. In fact, if $p= N_{j,k}$, $S_1 = S_{j,k}$, $S_2 = S_{j-1,k-1}$ and $S_3 = S_{j-1,k}$, then $\epsilon_{S_1S_2}(p) = 1$, while $\epsilon_{S_1S_2}(p) = -1$ (see also Figure \ref{canc}).

On the vanishing cycle at $N_{j,k}$ choose the orientation induced by $S_{j,k}$. The above considerations imply that the good relation induced by $S_{j,k}$ is
\begin{equation} \label{ex_g_rel}
  N_{j,k} - N_{j+1,k}-N_{j, k+1}+N_{j+1,k+1} = 0
 \end{equation}
The same relation holds if $N_{j,k}$ denotes the class of the exceptional $\PP^1$ in $X_0$. 

\begin{rem} \label{lin_comb_rel} The relation induced by  $S_{j-1,k}$ is 
\[ N_{j-1,k} - N_{j,k}-N_{j-1, k+1}+N_{j,k+1} = 0. \]
If we let $S_1 = S_{j,k}$ and $S_2 = S_{j-1,k}$, then the nodes satisfying (\ref{rel_coef}) are precisely those which do not cancel when we sum the relation (\ref{ex_g_rel}) induced by $S_1$ and the relation induced by $S_2$. More generally, given tropical $2$-cycles $S_1, \ldots, S_r$, the nodes which do not cancel when we sum their corresponding relations are precisely the nodes which satisfy (\ref{rel_coef}).
\end{rem}

We have the following

\begin{prop} \label{grel_comb}
Any good relation among the vanishing cycles in $\check X_0$ (resp. exceptional $\PP^1$'s in $X_0$) is a linear combination of the good relations (\ref{ex_g_rel}). Therefore, in these examples, $\omega$-related implies related. 
\end{prop}

\begin{proof}
Let $V$ be the free abelian group generated by the $N_{j,k}$'s and $W$ the lattice they span inside $H_3(\check X_0, \Z)$. Then we have the natural map $\pi: V \rightarrow W$. Inside $V$ consider the elements
\begin{equation} \label{rjk}
\mathcal R_{j,k} = N_{j,k} - N_{j+1,k}-N_{j, k+1}+N_{j+1,k+1}. 
\end{equation}
Clearly $\mathcal R_{j,k} \in \ker \pi$ for all $j$ and $k$. It is enough to show that $\ker \pi$ is generated by the $\mathcal R_{j,k}$'s. It is easy to show that any element of $P \in V$ can be written as
\[ P = N + R \]
where $N$ is a linear combination of the elements $\{ N_{1,1}, \ldots, N_{L,1}, N_{1,2}, \ldots, N_{1,M} \}$ and $R$ is a linear combination of the $\mathcal R_{j,k}$'s. This can be first proved by induction when $P=N_{j,k}$ and then extended to any $P$ by linearity. If $P \in \ker \pi$, then $N =0$, since $\{ N_{1,1}, \ldots, N_{L,1}, N_{1,2}, \ldots, N_{1,M} \}$ forms a basis of $W$. The last statement follows from Remark \ref{lin_comb_rel}. \end{proof}

We also have

\begin{cor}
A set of vanishing cycles in $\check X_0$ satisfies a good relation if and only if the corresponding exceptional $\PP^1$'s in $X_0$ also satisfy a good relation. 
\end{cor}

\begin{rem} Notice that this corollary does not prove that $\omega$-related is equivalent to $\C$-related, as stated in item $(i)$ of Conjecture \ref{trop_cycle_conj}. In fact here we have chosen a fixed resolution of the conifold, one that comes from a resolution of $B$. In principle there could be other resolutions which change the topology. We do not know if this is true for this set of examples.  Notice that a good relation on the exceptional $\PP^1$'s which is a linear combination of relations induced from tropical $2$-cycles holds on all resolutions (the construction in Theorem \ref{good_rel_pos} is independent of the resolution). Therefore any additional good relation which might exist on some other resolution cannot be detected by tropical $2$-cycles. 
\end{rem}

Let us now return to the case $L=M=3$. In this case it is easy to classify all possible good relations. We look at the $3 \times 3$ grid of Figure \ref{re_sm}. It is periodic and it has some obvious symmetries. We have the following 

\begin{prop} \label{grel_class}
If $L=M=3$, any good relation among a set of $k$ vanishing cycles in $\check X_0$ is equivalent to one of the following (up to the symmetry of the grid):
\begin{equation}
 \begin{split}
& k=4: \ \ \ \  \mathcal R_{1,1} = 0, \\
& k=6: \ \ \ \ \mathcal R_{1,1}-\mathcal R_{1,2} = 0, \\
& k=6: \ \ \ \ \mathcal R_{1,1}- \mathcal R_{2,2} = 0, \\
& k=7: \ \ \ \ \mathcal R_{1,1}+ \mathcal R_{2,2} = 0, \\
& k=8: \ \ \ \ \mathcal R_{1,1}- \mathcal R_{1,2} + \mathcal R_{2,2} = 0,
 \end{split}
\end{equation}
where $\mathcal R_{j,k}$ is the relation (\ref{rjk}) induced by the tropical $2$-cycle $S_{j,k}$.
\end{prop}

\begin{proof} We sketch the proof, leaving the details to the reader. First, for every fixed $k=1, \ldots, 9$ one can classify all possible configurations of $k$-nodes up to periodicity and symmetries of the grid. 
For instance, there are only two configurations of $k=2$ nodes: they are either consecutive points on a diagonal or on a line (horizontal or vertical). There are four configurations of $k=3$ nodes: three corners on a square $S_{j,k}$,  three nodes on a line (horizontal or vertical), three nodes on the diagonal, the configuration $\{ N_{1,1}, N_{1,2}, N_{2,3} \}$. Similarly for all other $k$'s.  
It is then easy to verify which ones of these configurations gives good relations. For instance, any configuration with $k=2$ or $k=3$ nodes gives linearly independent vanishing cycles. For $k=4$ the only configuration which does not give linearly independent vanishing cycles is when the nodes are the four corners of a square $S_{j,k}$. In this case we obtain the first good relation in the list. When $k=5$ there is only one configuration whose vanishing cycles are not linearly independent. It is equivalent to $\{ N_{1,1}, N_{1,2}, N_{2,1}, N_{2,2}, N_{3,3} \}$. The first three span a lattice not containing $N_{3,3}$, so there cannot be a good relation. Proceeding this way one obtains only the list above. 
\end{proof}

We have already discussed how the first, second, fourth and fifth case in the above list can be resolved using the methods of Section \ref{simult_res_sm}. We have not been able to resolve the third case. The more general case when $L$ or $M$ is greater than $3$ is more complicated. We have not yet attempted a thorough classification.

\medskip

A similar analysis can be done for the vanishing cycles on $X_0$. Label them by $E_{j,k}$. We expect them to span a lattice of rank $(L-1)(M-1)$ in $H_3(X_{0}, \Z)$. We conjecture that a basis of this lattice is given by the $E_{j,k}$'s with $j \in \{1, \ldots, L-1 \}$ and $k \in \{ 1, \ldots, M-1\}$. We have tropical $2$-cycles as in Figure \ref{sim_smooth} which give the following relations 
\[ \mathcal R^{j} := \sum_{k=1}^{M} E_{j,k}= 0  \]
\[ \mathcal R_{k} :=  \sum_{j=1}^{L} E_{j,k}= 0 \]

Also in this case we can show that any good relation among the vanishing cycles in $X_0$ (or the exceptional $\PP^1$'s in $\check X_0$) is a linear combination of the above relations. The proof is the same as in Proposition \ref{grel_comb}. In the case $M=L=3$ we have the following classification of good relations 

\begin{prop}
If $L=M=3$, any good relation among a set of $k$ vanishing cycles in $X_0$ is equivalent to one of the following (up to the symmetry of the grid):
\begin{equation}
\begin{split}
& k=3 \ \ \ \ \mathcal R^j= 0, \\
& k=4 \ \ \ \ \mathcal R^k- \mathcal R_{j} = 0, \\
& k=5 \ \ \ \ \mathcal R^k+ \mathcal R_{j} = 0, \\
& k=5 \ \ \ \ \mathcal R^1+\mathcal R^2- \mathcal R_1=0, \\
& k=6 \ \ \ \ \mathcal R^1+ \mathcal R^2 = 0, \\
& k=6 \ \ \ \ \mathcal R^1-\mathcal R^3+ \mathcal R_1 - \mathcal R_3 = 0, \\
& k=6 \ \ \ \ \mathcal R^2- \mathcal R_2 - \mathcal R^3 = 0, \\
& k=7 \ \ \ \  \mathcal R^1+ \mathcal R^2 + \mathcal R_1= 0, \\
& k=7 \ \ \ \ 3 \mathcal R^1+ 2\mathcal R^2 + \mathcal R^3-  \mathcal R_2- 2 \mathcal R_3= 0, \\
& k=8 \ \ \ \ 3 \mathcal R^1+ 2\mathcal R^2 + 2\mathcal R^3-  \mathcal R_2- 2 \mathcal R_3= 0, \\
& k=9 \ \ \ \  \mathcal R^1+ \mathcal R^2 + \mathcal R^3 = 0,
\end{split}
\end{equation}
\end{prop}

The proof is just like in Proposition \ref{grel_class}. In this case, the only sets of nodes we can resolve using the methods of Section \ref{simult_res_sm} are the first, third, fifth and eighth case in the above list. We do not know if or how one can resolve the other cases.

\subsection*{Acknowledgments} This project was partially supported by NSF award DMS-0854989:FRG \emph{Mirror Symmetry and Tropical Geometry}, by MIUR (\emph{Geometria Differenziale e Analisi Globale}, PRIN07 and \emph{Moduli spaces and their applications}, FIRB 2012 ). This article started while the second author was at the Universit\`a del Piemonte Orientale, in Alessandria (Italy), where the first author was hosted various times, so we would like to thank this institution. The authors would like to thank Mark Gross and Bernd Siebert for useful discussions. We also thank the referee for suggesting to add the material in \S \ref{example_good_rel} and for helping us improve exposition.

\bibliographystyle{plain}
\bibliography{biblio}

\begin{flushleft}
\vspace{1cm}
Ricardo CASTA\~NO-BERNARD \\
Mathematics Department \\
Kansas State University \\
138 Cardwell Hall \\
Manhattan, KS 66506, U.S.A. \\
E-mail address: \email{rcastano@math.ksu.edu}

\vspace{1cm}
Diego MATESSI \\
Dipartimento di Matematica \\
Universit\`a degli Studi di Milano \\
Via Cesare Saldini 50 \\
I-20133 Milan, Italy \\
E-mail address: \email{diego.matessi@unimi.it}
\end{flushleft}

\end{document}